\documentclass[11pt]{amsart}

\usepackage[T1]{fontenc}
\usepackage{lmodern}
\usepackage{amsmath,amssymb,mathtools}
\usepackage{graphicx}
\usepackage{float}
\usepackage{flafter}
\usepackage{booktabs}
\usepackage{array}
\usepackage{longtable}
\usepackage{enumitem}
\usepackage{needspace}
\usepackage[hidelinks]{hyperref}
\usepackage[nameinlink,noabbrev]{cleveref}

\graphicspath{{figures/cropped/}{media/media/}}
\numberwithin{equation}{section}

\newtheorem{theorem}{Theorem}[section]
\newtheorem{proposition}[theorem]{Proposition}
\newtheorem{lemma}[theorem]{Lemma}
\newtheorem{corollary}[theorem]{Corollary}
\newtheorem*{lemma*}{Lemma}
\newtheorem*{corollary*}{Corollary}
\theoremstyle{definition}
\newtheorem{definition}[theorem]{Definition}
\theoremstyle{remark}
\newtheorem{remark}[theorem]{Remark}
\newtheorem*{remark*}{Remark}

\newcommand{\TV}{\operatorname{TV}}
\newcommand{\Adm}{\operatorname{Adm}}
\newcommand{\Cross}{\operatorname{Cross}}
\newcommand{\Col}{\operatorname{Col}}
\newcommand{\Tet}{\operatorname{Tet}}

\title[Reconstructing Turaev--Viro invariants]{Reconstructing Turaev--Viro Invariants from Four-Fold Simple Branched Covering Presentations}
\author{Eri Hatakenaka}
\date{}

\begin{document}
\raggedbottom

\begin{abstract}
We construct a finite state sum from a link diagram labelled by transpositions.  The diagram presents a closed $3$-manifold as a four-fold simple branched cover, and the state sum equals the $\mathrm{SU}(2)$-type Turaev--Viro invariant of that manifold.  For each admissible colouring of the diagram's edges and complementary regions, a global factor is combined with crossing weights obtained by summing over compatible local colourings.  We then sum these products over all diagram colourings.  The formula is proved by assembling crossing-wise local triangulations and matching the two state sums.
\end{abstract}

\maketitle

\nocite{hatakenaka-2010,bobtcheva-piergallini-2004,turaev-viro-1992,
  burton-maria-spreer-2018,roberts-1995,kauffman-lins-1994,
  turaev-1992-shadow,carbone-2000,benedetti-petronio-1995,
  benedetti-petronio-1996,hatakenaka-nosaka-2012,nosaka-2014,
  ishii-oshiro-2022,rourke-sanderson-1982,lackenby-2021}
\section{Introduction}\label{sec:introduction}

This paper constructs a finite state sum that computes the Turaev--Viro invariant directly from a transposition-labelled link diagram presenting a four-fold simple branched cover.  It combines global diagram colourings with finite local sums at the crossings.  A closed generalized triangulation $K(D)$, assembled from crossing-wise pullbacks in the second half of the paper, is used only to prove agreement with the Turaev--Viro state sum and is not needed to define the diagrammatic state sum.

The Turaev--Viro invariant is an invariant of closed $3$-manifolds defined from edge colourings of a triangulation and quantum $6j$-symbols~\cite{turaev-viro-1992}.  We work with the $\mathrm{SU}(2)$-type Turaev--Viro invariant associated with $U_q(\mathfrak{sl}_2)$, with the conventions specified in Section~5.  On the other hand, every closed oriented $3$-manifold can be represented as a four-fold simple branched cover of $S^3$ whose branch set is a link labelled by transpositions.  We refer to a presentation by such a transposition-labelled link diagram as a \emph{covering presentation}.  The covering moves MI and MII in \cref{fig:covering-moves} preserve the oriented covering manifold represented by a four-fold simple branched covering presentation~\cite[Figure~1]{bobtcheva-piergallini-2004}.  Labelled isotopy is represented diagrammatically by labelled Reidemeister moves~\cite[Section~1]{bobtcheva-piergallini-2004}.  Here a \emph{labelled Reidemeister move} means an ordinary Reidemeister move in which the transposition labels on the strand segments in the move disk are matched according to the Wirtinger relations.

\begin{figure}[H]
  \centering
  \includegraphics[width=.63\textwidth]{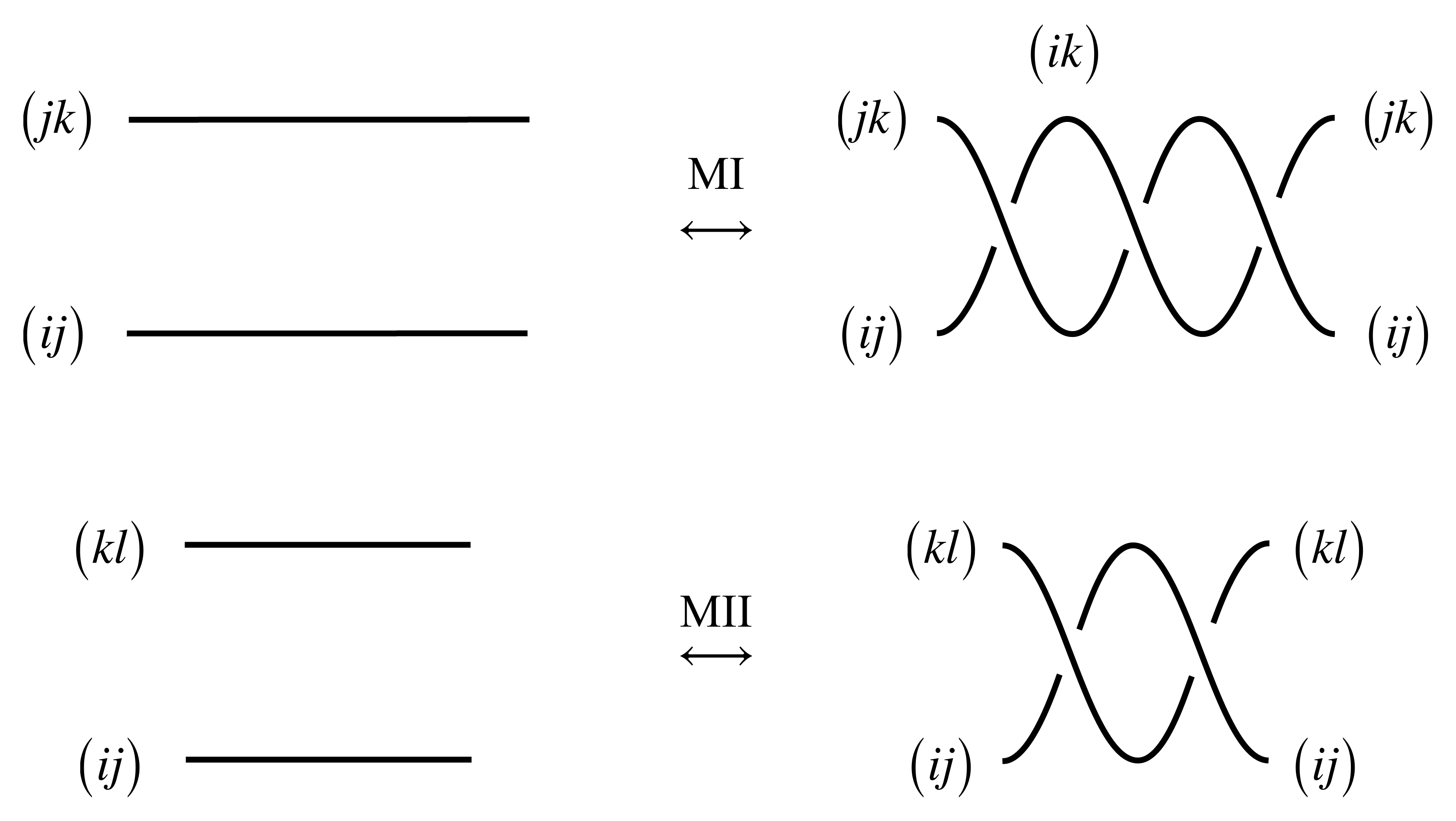}
  \caption{Covering moves MI and MII~\cite[Figure~1]{bobtcheva-piergallini-2004}.  In MI, the indices $i,j,k$ are distinct; in MII, $\{i,j\}\cap\{k,l\}=\varnothing$.  In both moves the labelled tangles agree outside the displayed $3$-ball.}
  \label{fig:covering-moves}
\end{figure}

Hatakenaka~\cite{hatakenaka-2010} introduced a framework for state sums based on covering presentations, applied it to the Dijkgraaf--Witten invariant, and associated a banana $3$-ball with each crossing.  The present construction starts from that setting but does not simply substitute Turaev--Viro weights into the map $R$ and weight $X$ of~\cite{hatakenaka-2010}.  Instead, it specifies the local triangulations, boundary labels, ordered face pairings, and two levels of summation required by the quantum $6j$-symbols.  A first-stage admissible colouring $\chi$ fixes the colours on the whole diagram; the compatible local colourings at each crossing $c$ are then summed to give $W_c(\chi)$.  Definition~4.2 multiplies these crossing weights by the global factor $W_D(\chi)$ and sums over $\chi$ to define $Z_{r,s}(D)$.  Proposition~4.3 gives the equivalent simultaneous expansion, and Section~8.3 applies the construction to $L(2,1)$.

Known descriptions of the Turaev--Viro invariant use skein theory and chain-mail links~\cite{roberts-1995}, recoupling theory and special spines~\cite{kauffman-lins-1994}, shadows~\cite{turaev-1992-shadow}, surgery or Heegaard presentations~\cite{carbone-2000}, or $o$-graphs~\cite{benedetti-petronio-1995,benedetti-petronio-1996}.  Our state sum instead begins directly with the transposition-labelled covering presentation.  Related constructions from branched covers and diagram colourings appear in~\cite{hatakenaka-nosaka-2012,nosaka-2014,ishii-oshiro-2022}, but use different colourings and weights.

Throughout the paper, $M(D)$ denotes the closed $3$-manifold represented by the labelled diagram $D$, $\TV_{r,s}(M(D))$ the Turaev--Viro invariant with the conventions of Section~5, and $Z_{r,s}(D)$ the diagrammatic state sum defined in Sections~4 and~8.

\begin{theorem}[Reconstruction theorem]\label{thm:main}
Let $D$ be a diagram of a four-fold simple branched covering presentation whose monodromy action is transitive.  If the projection graph $\Gamma_D$ is connected, define $Z_{r,s}(D)$ as in Section~4; if it is disconnected, use the connectedization described in Section~8.  Then, for every Turaev--Viro parameter $(r,s)$ specified in Section~3.3, the diagrammatic state sum is equal to the Turaev--Viro invariant of the manifold represented by $D$:
\[
  Z_{r,s}(D)=\TV_{r,s}(M(D)).
\]
\end{theorem}

The proof proceeds as follows.  We place a finite local triangulation at each crossing and glue these triangulations along the graph edges to construct a closed generalized triangulation $K(D)$.  We show that an admissible edge colouring of $K(D)$ decomposes uniquely into a first-stage admissible colouring $\chi$ and one local colouring at each crossing.  We then verify that every edge, triangle, tetrahedron, and vertex factor in the Turaev--Viro state sum occurs exactly once in the diagrammatic state sum.  This gives $Z_{r,s}(D)=\TV_{r,s}(K(D))$.  Finally, a PL homeomorphism identifies $|K(D)|$ with $M(D)$, and the reconstruction theorem follows.  When the projection graph is disconnected, Section~8 connects it by labelled Reidemeister II moves and proves that the result is independent of the choices made in this procedure.

Sections~2--4 define the diagram colourings, crossing weights, and diagrammatic state sum.  Sections~5--7 construct the corresponding generalized triangulation and prove the reconstruction theorem by comparing the two state sums.  Section~8 gives consequences and the calculation for $L(2,1)$, and Section~9 concludes the paper.

\section{Covering diagrams and first-stage admissible colourings}\label{sec:covering-colouring}

This section introduces transposition-labelled link diagrams, the finite colour set, and the first-stage admissible colourings that form the first sum in the diagrammatic state sum.  We begin with the covering presentations and their Wirtinger relations; all data are defined on the planar diagram, and their three-dimensional meaning is explained in Section~6.

\subsection{Four-fold simple branched covering presentations}\label{subsec:covering-presentations}

Let $D$ be a link diagram on $S^2$, and let $\Cross(D)$ denote its set of crossings.  Replace every crossing by a $4$-valent vertex and denote the resulting planar graph by $\Gamma_D$.  Each edge of $\Gamma_D$ corresponds to a crossing-to-crossing segment obtained by removing the crossings from $D$.  We call it a \emph{graph edge} and denote the set of graph edges by $E(D)$.  We write $\mathcal{R}(D)$ for the set of complementary regions of $S^2\setminus\Gamma_D$.  When the two sides of a graph edge $e$ must be distinguished, denote the incident complementary regions by $R_1(e)$ and $R_2(e)$; the numbering is arbitrary, and the two regions may coincide.  See \cref{fig:diagram-graph}.

\begin{figure}[H]
  \centering
  \includegraphics[width=.82\textwidth]{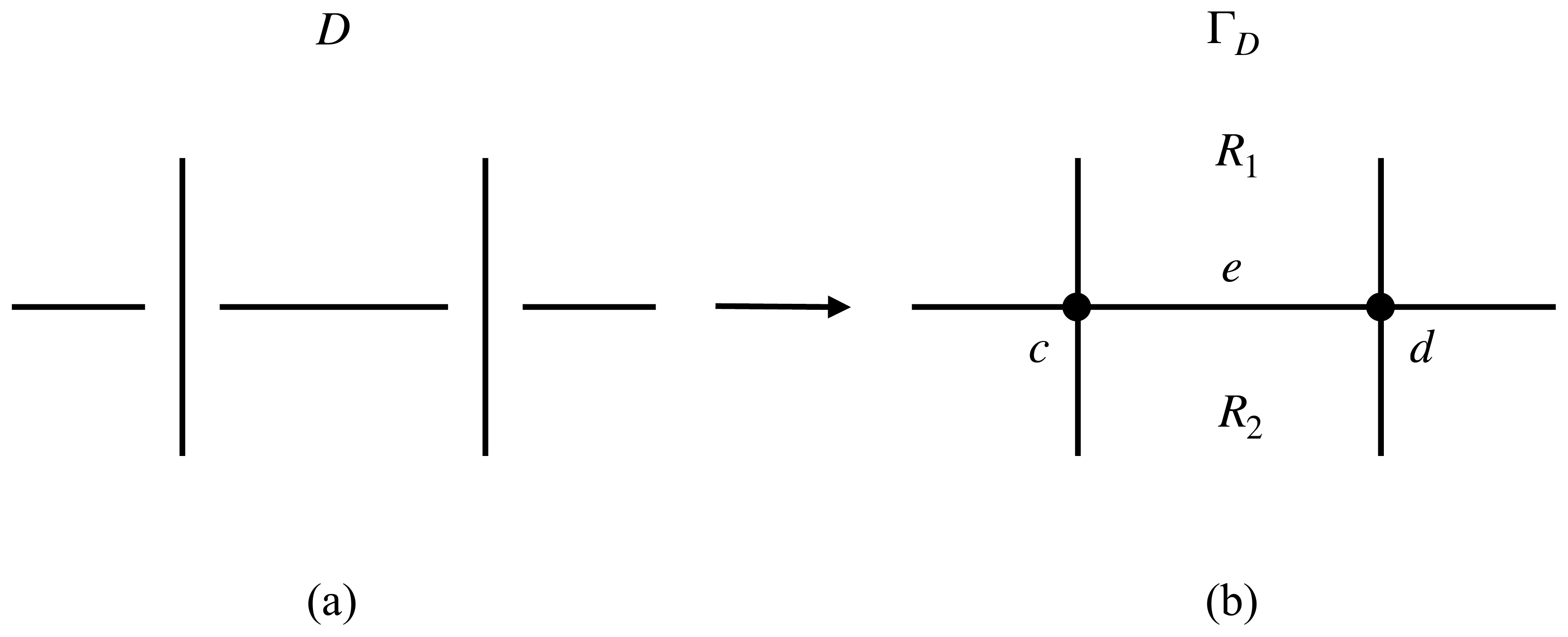}
  \caption{A local part of the link diagram $D$ and the planar graph $\Gamma_D$.  The regions $R_1(e)$ and $R_2(e)$ lie on the two sides of the graph edge $e$.}
  \label{fig:diagram-graph}
\end{figure}

Assign a transposition $\tau_e\in S_4$ to every graph edge $e\in E(D)$.  At a crossing, the two graph edges belonging to the over-strand carry the same transposition $\sigma$.  If one side of the under-strand carries $\tau$, the Wirtinger relation requires the other side to carry $\sigma\tau\sigma^{-1}$; see \cref{fig:wirtinger}.  Since $\sigma$ is a transposition, $\sigma^{-1}=\sigma$, and the latter label may also be written $\sigma\tau\sigma$.  We retain the notation $\sigma\tau\sigma^{-1}$ to display the conjugation in the Wirtinger relation.

\begin{figure}[H]
  \centering
  \includegraphics[width=.38\textwidth]{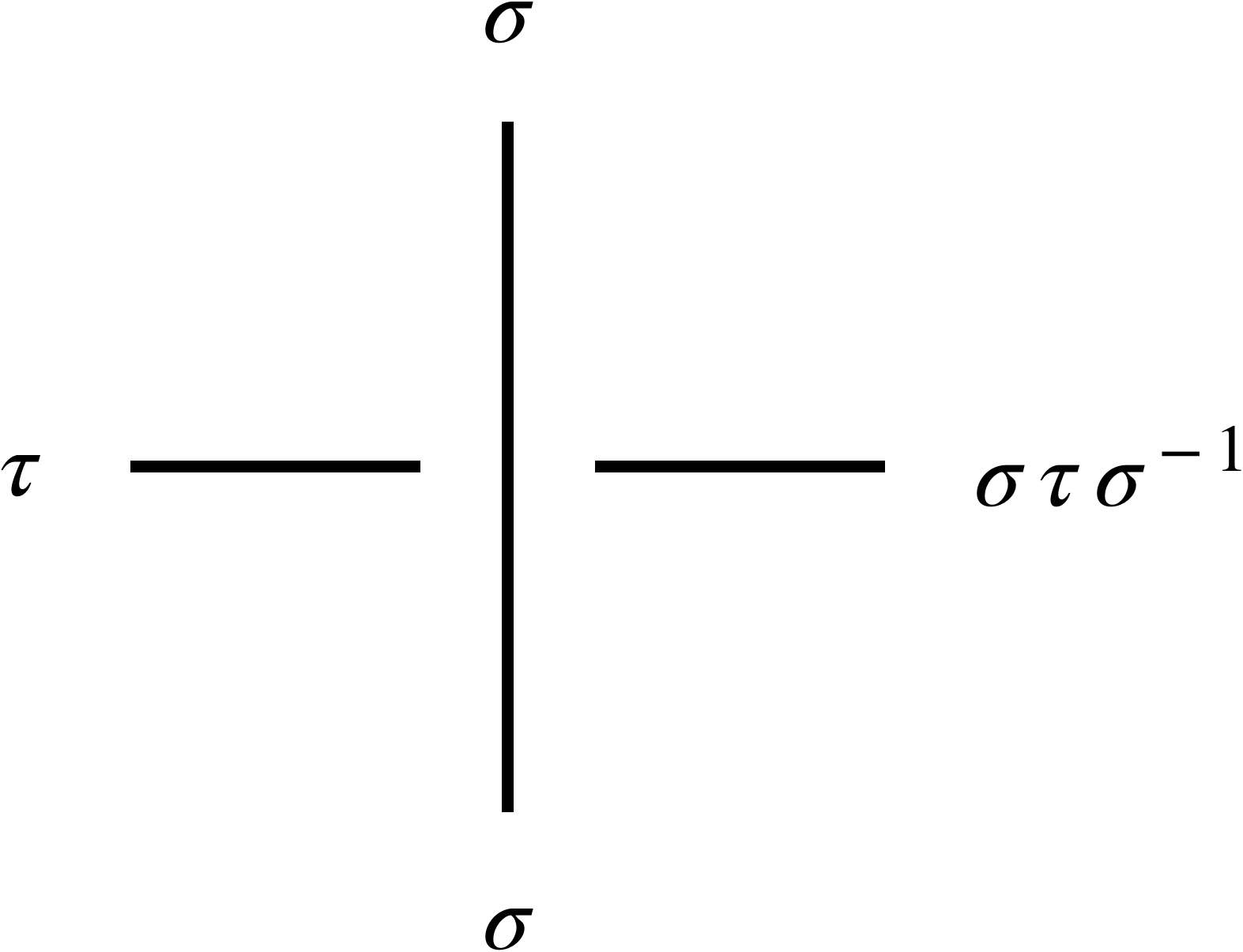}
  \caption{The Wirtinger relation at a labelled crossing.  The two sides of the over-strand are both labelled by $\sigma$.  If one side of the under-strand is labelled by $\tau$, the other is labelled by $\sigma\tau\sigma^{-1}$.}
  \label{fig:wirtinger}
\end{figure}

These conditions determine a monodromy homomorphism from the fundamental group of the link complement to $S_4$, and hence a four-fold simple branched cover
\[
  p_D\colon M(D)\longrightarrow S^3
\]
as in~\cite[Section~2]{hatakenaka-2010}.  We assume that the subgroup $G_D$ generated by the transposition labels acts transitively on $\{1,2,3,4\}$.  This condition is equivalent to the connectedness of $M(D)$.

\subsection{The finite colour set and admissibility}\label{subsec:colour-set}

We next fix the finite colour set used both in the first stage and inside the crossings, together with the admissibility condition for a triple of colours.

Fix an integer $r\geq 3$ and set
\[
  C_r:=\{0,1,\ldots,r-2\}.
\]
We use the doubled-colour convention $x=2\lambda$, where $\lambda$ is the
half-integer label of Turaev--Viro~\cite[Section~7.1]{turaev-viro-1992}; this
is the integer encoding also used in
\cite[Section~2.1]{burton-maria-spreer-2018}.
The second integer $s$ needed for the Turaev--Viro weights is introduced in Section~3.3.  A triple $(x,y,z)\in C_r^3$ is \emph{admissible} if
\[
\begin{gathered}
  x+y+z \equiv 0 \pmod 2,\\
  x\leq y+z,\qquad y\leq z+x,\qquad z\leq x+y,\\
  x+y+z\leq 2r-4.
\end{gathered}
\]

\Needspace{7\baselineskip}
\subsection{State-sum colours and first-stage admissible colourings}\label{subsec:first-stage-colouring}

We now place finitely many colours on the graph edges and complementary regions.  This gives the first-stage admissible colourings before any colours inside the crossings are chosen.

We use $\{1,2,3,4\}$ as two sets of sheet labels.  For every graph edge $e\in E(D)$, assign a colour $p_i(e)$ to the first sheet label $i$ and a colour $q_j(e)$ to the second sheet label $j$.  For every complementary region $R\in\mathcal{R}(D)$, assign a colour $b_i(R)$ to the first sheet label $i$.  The two sets of labels have different roles, but over the base region $R_0$ fixed below, the first sheet label $i$ corresponds to the second sheet label $i$.  Over a general region $R$, the correspondence is recorded by the region transport $\mu_R$ defined below.  These colours are illustrated in \cref{fig:first-stage-colours}.

\begin{figure}[H]
  \centering
  \includegraphics[width=.64\textwidth]{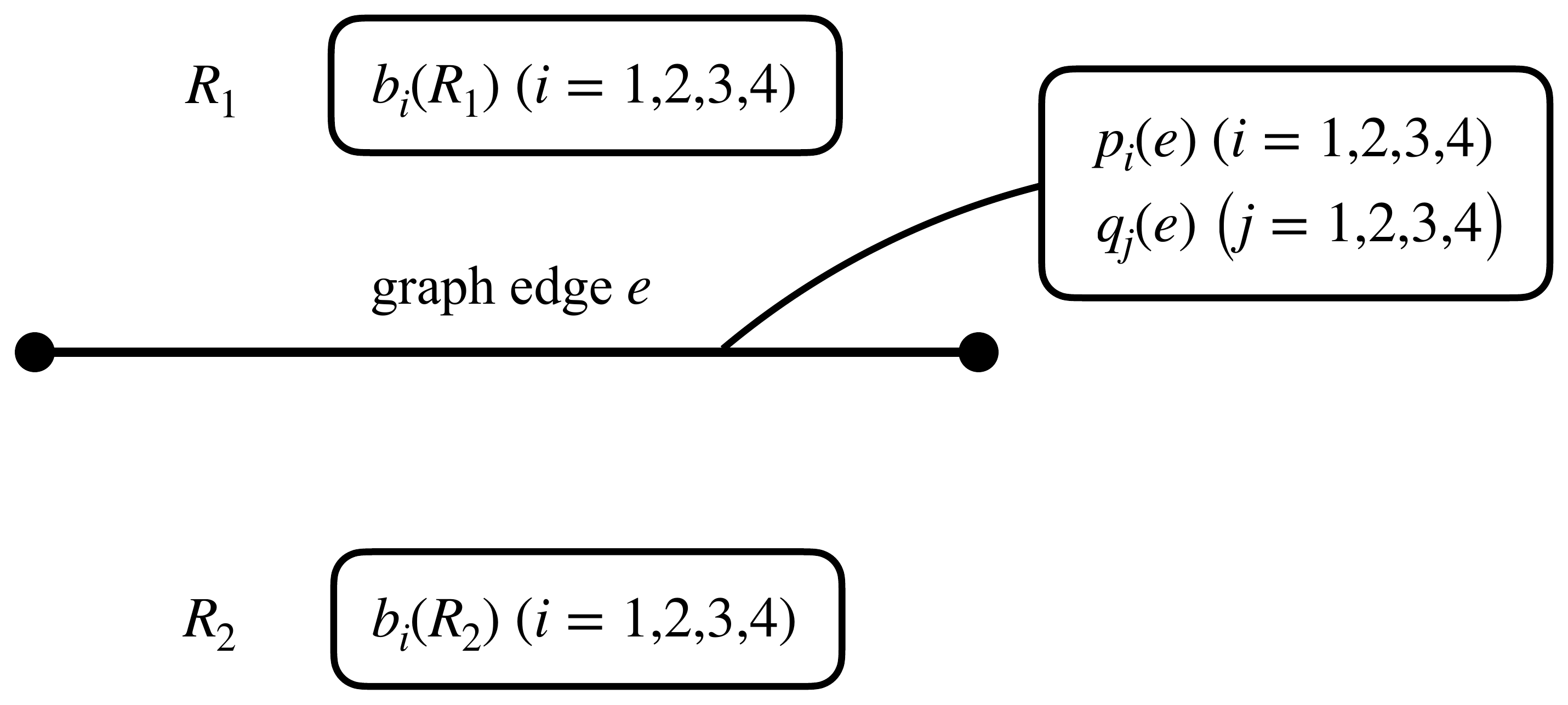}
  \caption{First-stage colours on a graph edge $e$ and its two incident regions.  The indices $i$ and $j$ denote first and second sheet labels, respectively.}
  \label{fig:first-stage-colours}
\end{figure}

Until Section~7, we assume that $\Gamma_D$ is connected.  The disconnected case is treated in Section~8 by applying the following definitions to a connectedized diagram.  Fix one base region $R_0$.  Elements of $S_4$ act as maps on $\{1,2,3,4\}$, and products mean composition of maps, with
\[
  (\alpha\circ\beta)(i)=\alpha\bigl(\beta(i)\bigr).
\]
For $R\in\mathcal{R}(D)$, choose a \emph{dual path} from $R_0$ to $R$, that is, a path in the dual graph of $\Gamma_D$.  If the dual path crosses, in order, the graph edges $e_1,\ldots,e_m$, define
\[
  \mu_R:=\tau_{e_1}\circ\cdots\circ\tau_{e_m}\in S_4,
  \qquad
  \mu_{R_0}:=\mathrm{id}
\]
for the empty dual path.  The same $\tau_e$ is used in either crossing direction because a transposition satisfies $\tau_e^{-1}=\tau_e$.  We use the right-update convention: if the current correspondence is $\mu$ and the dual path is extended across an edge $e$, the new correspondence sends $i$ to $\mu(\tau_e(i))$ and is therefore $\mu\circ\tau_e$.  Thus the displayed factors are written in the order in which the edges are crossed.

\begin{lemma}[Well-definedness of the region transport]\label{lem:region-transport}
The permutation $\mu_R$ defined above is independent of the dual path from $R_0$ to $R$.
\end{lemma}

\begin{proof}
Draw $\Gamma_D$ on $S^2$ and consider a small closed dual path around a crossing.  Let $\sigma$ be the over-strand label and $\tau$ one of the under-strand labels; the other under-strand label is $\sigma\tau\sigma^{-1}$ by the Wirtinger relation.  In one direction, the closed dual path crosses edges labelled, in order,
\[
  \sigma,\quad \tau,\quad \sigma,\quad \sigma\tau\sigma^{-1}.
\]
Since $\sigma$ and $\tau$ are transpositions, $\sigma^2=\tau^2=\mathrm{id}$, and
\[
  \sigma\circ\tau\circ\sigma\circ
  (\sigma\circ\tau\circ\sigma^{-1})=\mathrm{id}.
\]
The product in the opposite direction is also the identity.  Because $S^2$ is simply connected, every closed dual path can be contracted using these small closed paths around the crossings, so its product is the identity.  Joining one dual path from $R_0$ to $R$ to the reverse of another gives a closed path.  The two paths therefore determine the same product.
\end{proof}

For a region $R$, write
\[
  j:=\mu_R(i)
\]
for the second sheet label corresponding to the first sheet label $i$.  In particular, $j=i$ over $R_0$.  The three edge colours of the same lifted triangle are
\[
  p_i(e),\qquad q_j(e),\qquad b_i(R).
\]
After the lifted boundary triangles have been introduced, Section~6.5 establishes the geometric meaning of the correspondence $j=\mu_R(i)$ and hence of the possibly different indices on $p_i(e)$ and $q_j(e)$.

\begin{definition}[First-stage admissible colouring]\label{def:first-stage-colouring}
A colouring $\chi$ assigns a colour in $C_r$ to $p_i(e)$ for every $e\in E(D)$ and $i\in\{1,2,3,4\}$, to $q_j(e)$ for every $e\in E(D)$ and $j\in\{1,2,3,4\}$, and to $b_i(R)$ for every $R\in\mathcal{R}(D)$ and $i\in\{1,2,3,4\}$.

Let $R$ be the complementary region incident to one side of a graph edge $e$.  For $i\in\{1,2,3,4\}$, put $j:=\mu_R(i)$ and associate with this side the triple
\[
  \bigl(p_i(e),q_j(e),b_i(R)\bigr).
\]
The region $R$ is chosen separately for each side of $e$.  Hence, if the same region is incident to several graph edges, the four triples indexed by $i=1,2,3,4$ are considered separately for each incident graph edge.  The colouring $\chi$ is \emph{first-stage admissible} if the displayed triple is admissible for both sides of every graph edge, for the region incident to each side, and for every $i\in\{1,2,3,4\}$.  We denote the set of first-stage admissible colourings by $\Adm_1(D;r)$.
\end{definition}

\begin{corollary}[Change of base region]\label{cor:base-region}
When the base region is to be displayed, write the set in \cref{def:first-stage-colouring} as $\Adm_1^{R_0}(D;r)$.  Change the base region from $R_0$ to $R_0'$, let $\mu_R'$ be the resulting region transports, and put $\kappa:=\mu_{R_0'}$, where the right-hand side is computed from the old base region $R_0$.  Then
\[
  \mu_R'=\kappa^{-1}\circ\mu_R
\]
for every $R$.  Moreover, the assignments
\[
  p_i'(e):=p_i(e),\qquad
  b_i'(R):=b_i(R),\qquad
  q_j'(e):=q_{\kappa(j)}(e)
\]
define a bijection
\[
  \Adm_1^{R_0}(D;r)\longrightarrow \Adm_1^{R_0'}(D;r).
\]
\end{corollary}

\begin{proof}
Concatenate a dual path from the old base region $R_0$ to $R_0'$ with a dual path from the new base region $R_0'$ to $R$.  By \cref{lem:region-transport},
\[
  \mu_R=\kappa\circ\mu_R',
  \qquad\text{and hence}\qquad
  \mu_R'=\kappa^{-1}\circ\mu_R.
\]
For every $i$, it follows that
\[
  q'_{\mu_R'(i)}(e)
  =q_{\kappa(\mu_R'(i))}(e)
  =q_{\mu_R(i)}(e).
\]
Together with $p_i'(e)=p_i(e)$ and $b_i'(R)=b_i(R)$, this gives
\[
  \bigl(p_i'(e),q'_{\mu_R'(i)}(e),b_i'(R)\bigr)
  =\bigl(p_i(e),q_{\mu_R(i)}(e),b_i(R)\bigr)
\]
for every side of every graph edge and every first sheet label $i$.  The relabelling therefore preserves admissibility and defines the stated bijection.  Its inverse reindexes $q_j(e)$ by $\kappa^{-1}$.
\end{proof}

\section{Crossing weights}\label{sec:crossing-weights}

The crossing tables in Appendix~A determine the local colourings placed at each crossing.  After defining these colourings, we shall introduce the coefficients that evaluate them and form the crossing weight by a finite sum.  Everything in this construction is read from the colours of Section~2.3 and the finite tables; no lifted triangulation is needed for the definition.

\subsection{The three crossing types and the first-stage colours}\label{subsec:crossing-types}

We classify the crossings into three types and introduce notation for referring to the first-stage colours at the four local positions around a crossing.

Let $c$ be a crossing, let $\sigma$ be the transposition on its over-strand, and let $\tau$ be the transposition drawn on the left side of its under-strand in \cref{fig:crossing-types}.  The Wirtinger relation gives $\sigma\tau\sigma^{-1}$ on the other side of the under-strand.  The crossing is of \emph{type 1} if $\sigma=\tau$, of \emph{type 2} if $\sigma$ and $\tau$ have exactly one sheet label in common, and of \emph{type 3} if they have no sheet label in common.  The representative labellings in \cref{fig:crossing-types} use $\sigma=\tau=(12)$ for type 1; $\sigma=(23)$, $\tau=(12)$, and $\sigma\tau\sigma^{-1}=(13)$ for type 2; and $\sigma=(34)$ and $\tau=(12)$ for type 3.

\begin{figure}[H]
  \centering
  \includegraphics[width=.86\textwidth]{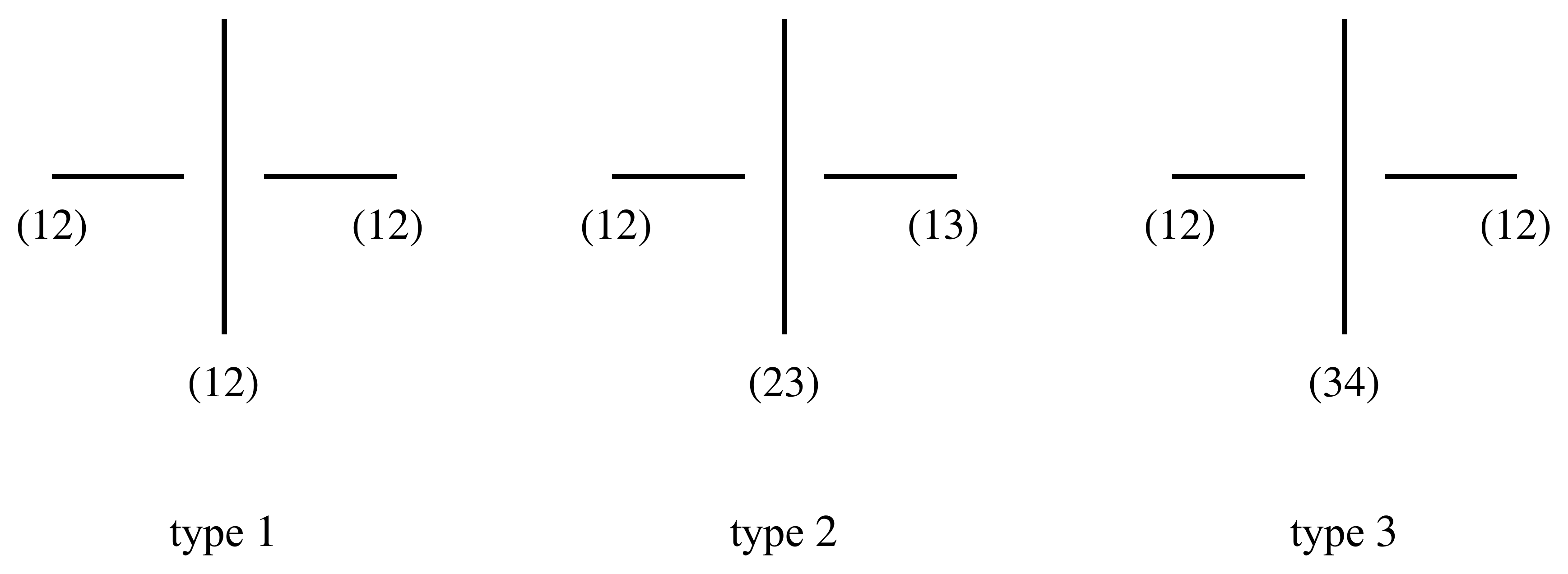}
  \caption{Representative labellings of crossing types 1, 2, and 3.  The vertical strand is the over-strand.}
  \label{fig:crossing-types}
\end{figure}

Each end of a graph edge at its incident crossing is called a \emph{half-edge}.  If both ends of a loop edge meet the same crossing, they are still regarded as distinct half-edges.  No colour is assigned to a half-edge itself.  If a half-edge $h$ is an end of a graph edge $e$, the first-stage colours referred to at $h$ are the colours $p_i(e)$ and $q_j(e)$ assigned to $e$ in Section~2.3.  We denote the set of half-edges by $\mathcal{H}(D)$.

At a crossing $c$, denote the half-edges meeting it from the left, top, right, and bottom by $h_1,h_2,h_3,h_4$, respectively, and denote their graph edges by $e_1,e_2,e_3,e_4$; see \cref{fig:crossing-positions}.  The global complementary regions containing the four local sectors between the consecutive pairs $(h_1,h_2)$, $(h_2,h_3)$, $(h_3,h_4)$, and $(h_4,h_1)$ are denoted by $R_{12},R_{23},R_{34},R_{41}$, respectively.  These subscripts record positions around the crossing.  Thus, even if $e_1=e_4$, the two ends $h_1$ and $h_4$ are distinct and the region between them is still denoted by $R_{41}$.  The same global region may occupy more than one local sector.  For example, if $R_{12}=R_{34}=R$, the same $b_i(R)$ and $\mu_R$ are used at both positions.

\begin{figure}[H]
  \centering
  \includegraphics[width=.70\textwidth]{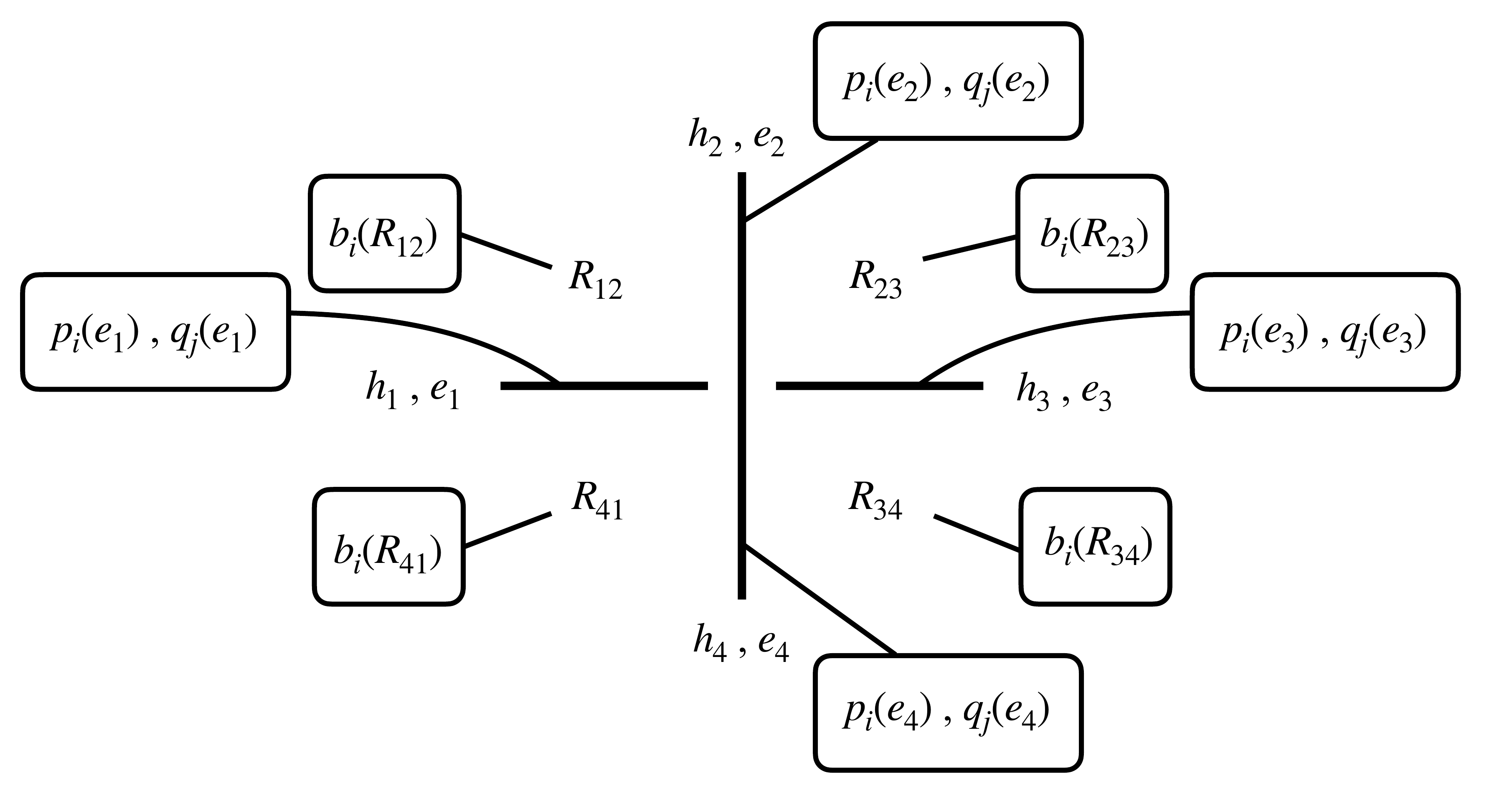}
  \caption{The half-edges $h_1,\ldots,h_4$, graph edges $e_1,\ldots,e_4$, and global complementary regions $R_{12},R_{23},R_{34},R_{41}$ at a crossing.  The displayed first-stage colours are those referred to at the four local positions.}
  \label{fig:crossing-positions}
\end{figure}

\subsection{Local colourings at a crossing}\label{subsec:local-colourings}

Fix a first-stage admissible colouring of a labelled diagram.  We now use the finite tables in Appendix~A to impose the admissibility conditions on the colours inside each crossing and to define the finite set of permitted local colourings.

For a crossing $c$ of type $t\in\{1,2,3\}$, use Table~A.2, A.3, or A.4 in Appendix~A when $t=1,2$, or $3$, respectively.  Write the local colours in the corresponding table as $x_1,\ldots,x_{m_t}$.  As recorded in Table~A.1, $m_1=22$ and $m_2=m_3=20$.

The symbols $p_i(e_k)$, $q_j(e_k)$, and $b_i(R_{12})$, $b_i(R_{23})$, $b_i(R_{34})$, $b_i(R_{41})$ in Tables~A.2--A.4 refer to the first-stage colours on the graph edges and global complementary regions shown in \cref{fig:crossing-positions}.  The colours $x_1,\ldots,x_{m_t}$ are chosen locally at the fixed crossing $c$.  An assignment of colours in $C_r$ to these $m_t$ symbols is called a \emph{local colouring} $\beta_c$ at $c$.

To use a table at $c$, rotate or reflect the representative of the same type in \cref{fig:crossing-types} so that its over-strand and under-strand match those at $c$.  The resulting plane symmetry identifies the four half-edge positions $h_1,h_2,h_3,h_4$, and hence the associated graph edges $e_1,e_2,e_3,e_4$ and complementary regions $R_{12},R_{23},R_{34},R_{41}$ of \cref{fig:crossing-positions}, with the corresponding data at $c$.  At the same time, relabel all four sheet labels uniformly so that the transposition labels agree.  For a symbol $p_i$ or $b_i$ in the table, substitute the first-stage colour whose index is the relabelled first sheet label.

Let $R$ be the global complementary region at $c$ corresponding to $R_{12}$ in the representative.  The convention in Appendix~A identifies the first and second sheet numbers over $R_{12}$.  Suppose that the first sheet label corresponding to an index $j$ in the table is sent by the above relabelling to the actual label $\ell$.  Then substitute $q_{\mu_R(\ell)}$ for $q_j$ and use this rule for every $q$-symbol in the table.  For example, suppose that the type 1 label $(12)$ is matched with the actual label $(23)$ by the relabelling
\[
  1\longmapsto2,\qquad 2\longmapsto3,\qquad
  3\longmapsto1,\qquad 4\longmapsto4,
\]
and that $\mu_R=(23)$.  Then the local table symbols are read at the actual crossing as follows: $p_1(e_k)$ as $p_2(e)$, $b_1(R_{12})$ as $b_2(R)$ for the global complementary region $R$, and $q_1(e_k)$ as
\[
  q_{\mu_R(2)}(e)=q_3(e).
\]
We substitute the colours of $\chi$ and $\beta_c$ into the tables according to these rules.

For a crossing of type $t$, use the triples A1--A52 in Table~A.2(a) when $t=1$, the triples A1--A48 in Table~A.3(a) when $t=2$, and the triples A1--A48 in Table~A.4(a) when $t=3$.  Each list consists of triples formed from the first-stage colours in \cref{fig:crossing-positions} and the local colours $x_k$.  Denote this finite set of triples by $\mathcal{A}_t$.  A local colouring $\beta_c$ is \emph{compatible with $\chi$} if every triple in $\mathcal{A}_t$ is admissible after the colours of $\chi$ and $\beta_c$ have been substituted.  The test is summarized in \cref{fig:local-admissibility}.

We write $\Col_c(\chi)$ for the set of all local colourings at $c$ compatible with $\chi$.  Thus
\[
  \Col_c(\chi)
  =\left\{\beta_c\in C_r^{m_t}\ \middle|\
  \begin{array}{c}
  \text{every triple in $\mathcal{A}_t$ is admissible}\\[-2pt]
  \text{under $\chi$ and $\beta_c$}
  \end{array}
  \right\}.
\]

\begin{figure}[!tbp]
  \centering
  \includegraphics[width=.52\textwidth]{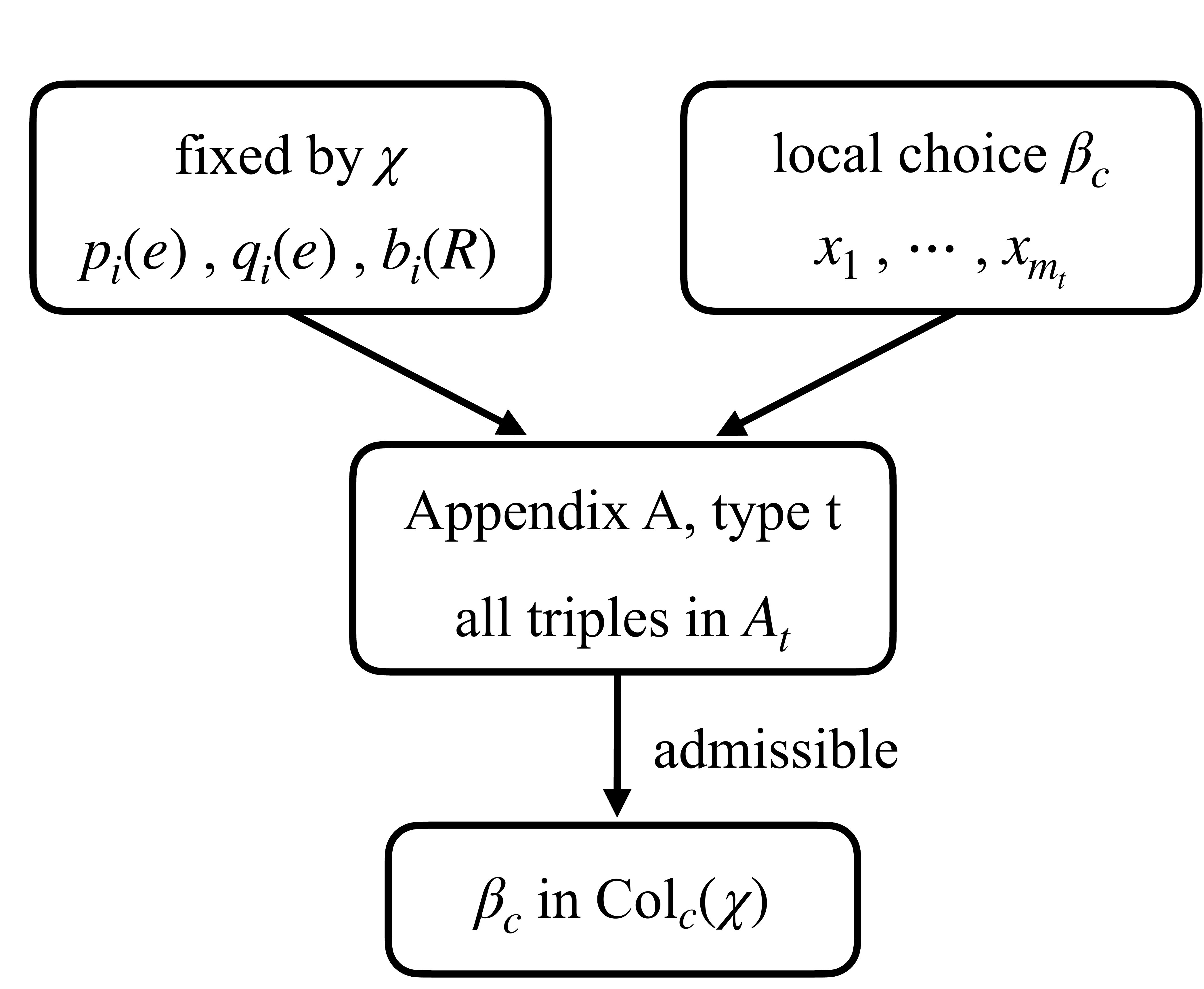}
  \caption{The admissibility test using the first-stage colours and the local colours at a crossing.}
  \label{fig:local-admissibility}
\end{figure}

\Needspace{8\baselineskip}
\begin{definition}[Admissible diagram colouring]\label{def:diagram-colouring}
An admissible diagram colouring of $D$ consists of a first-stage admissible colouring $\chi\in\Adm_1(D;r)$ and a choice of one local colouring $\beta_c\in\Col_c(\chi)$ for every crossing $c\in\Cross(D)$.  We write it as
\[
  \theta=\bigl(\chi,\{\beta_c\mid c\in\Cross(D)\}\bigr)
\]
and denote the set of all such colourings by $\Adm_2(D;r)$.
\end{definition}

\subsection{Algebraic coefficients}\label{subsec:algebraic-coefficients}

We now fix the coefficients $E$, $\Theta$, and $\Tet$ and the normalization factor $\eta_{r,s}$ used in both state sums.

Following the standard Turaev--Viro conventions in \cite[Section~7.1]{turaev-viro-1992} and \cite[Section~2.1]{burton-maria-spreer-2018}, fix integers $r\geq3$ and $s$ such that
\[
  0<s<2r,
  \qquad
  \gcd(r,s)=1.
\]
We continue to use the doubled-colour convention of Section~2.2.  Put
\[
  \zeta:=\exp\!\left(\frac{\pi\sqrt{-1}\,s}{r}\right),
  \qquad
  [n]:=\frac{\zeta^n-\zeta^{-n}}{\zeta-\zeta^{-1}}
  \quad(n\geq1),
\]
and define the quantum factorials by
\[
  [0]!:=1,
  \qquad
  [n]!:=\prod_{k=1}^{n}[k]
  \quad(n\geq1).
\]
The normalization factor is
\[
  \eta_{r,s}:=-\frac{(\zeta-\zeta^{-1})^2}{2r}.
\]
Since $\zeta$ lies on the unit circle, this is also $|\zeta-\zeta^{-1}|^2/(2r)$ and agrees with the convention in \cite[Section~2.1]{burton-maria-spreer-2018}.

Under these assumptions, $\zeta^2$ is a primitive root of unity of order $r$.  Admissibility ensures that the indices of the quantum factorials appearing in the denominators below lie between $0$ and $r-1$.  The denominators are therefore nonzero; see \cite[Section~7.1]{turaev-viro-1992} and \cite[Section~2.1]{burton-maria-spreer-2018}.

For $x\in C_r$, define the edge weight by
\[
  E(x):=(-1)^x[x+1].
\]
For an admissible triple $(x,y,z)$, define the triangle weight by
\[
  \Theta(x,y,z)
  :=(-1)^{(x+y+z)/2}
  \frac{
    \bigl[(x+y-z)/2\bigr]!
    \bigl[(x+z-y)/2\bigr]!
    \bigl[(y+z-x)/2\bigr]!
  }{
    \bigl[(x+y+z)/2+1\bigr]!
  }.
\]
If the triple is not admissible, set $\Theta(x,y,z):=0$.  This convention does not admit such a triple as a state; it merely allows a non-admissible term in a finite sum to be written as zero.

For an ordered six-colour tuple $U=(u_1,\ldots,u_6)$, take the four face triples to be
\[
  (u_1,u_2,u_4),\qquad
  (u_1,u_3,u_5),\qquad
  (u_2,u_3,u_6),\qquad
  (u_4,u_5,u_6).
\]
For this definition, put
\[
\begin{gathered}
  A_1:=\frac{u_1+u_2+u_4}{2},\qquad
  A_2:=\frac{u_1+u_3+u_5}{2},\\
  A_3:=\frac{u_2+u_3+u_6}{2},\qquad
  A_4:=\frac{u_4+u_5+u_6}{2},\\[2pt]
  B_1:=\frac{u_1+u_2+u_5+u_6}{2},\qquad
  B_2:=\frac{u_1+u_3+u_4+u_6}{2},\\
  B_3:=\frac{u_2+u_3+u_4+u_5}{2}.
\end{gathered}
\]
If all four face triples are admissible, define the tetrahedron weight by
\[
  \Tet(U)
  :=\sum_{\ell=\max\{A_1,A_2,A_3,A_4\}}^{\min\{B_1,B_2,B_3\}}
  \frac{(-1)^\ell[\ell+1]!}
  {\displaystyle
    \prod_{j=1}^{4}[\ell-A_j]!
    \prod_{k=1}^{3}[B_k-\ell]!}.
\]
If any of the four face triples is not admissible, set $\Tet(U):=0$.

\subsection{Local contribution at a crossing}\label{subsec:crossing-local-contribution}

For each local colouring at a crossing, we now combine the coefficients of Section~3.3 according to the finite crossing tables.  The resulting finite product is the local contribution that will be summed over all compatible local colourings in Section~3.5.

Recall that $\mathcal{A}_t$ is the finite set of admissibility triples from the relevant part~(a) of Appendix~A.  Let $\mathcal{S}_t$ be the finite set of ordered six-colour tuples listed in the corresponding part~(b).  In every row of part~(b), the six colours are read in the column order
\[
  01,\ 02,\ 03,\ 12,\ 13,\ 23
\]
and form an ordered tuple $U=(u_1,\ldots,u_6)$ for the tetrahedron weight $\Tet(U)$.

\begin{definition}[Local contribution at a crossing]\label{def:crossing-local-contribution}
Fix a first-stage admissible colouring $\chi\in\Adm_1(D;r)$, a crossing $c$ of type $t$, and a compatible local colouring $\beta_c\in\Col_c(\chi)$.  After substituting the colours of $\chi$ and $\beta_c$ in the relevant tables as in Section~3.2, define the local contribution at $c$ by
\[
  w_c(\chi,\beta_c)
  :=
  \prod_{k=1}^{m_t} E(x_k)
  \prod_{(u,v,w)\in\mathcal{A}_t} \Theta(u,v,w)
  \prod_{U\in\mathcal{S}_t} \Tet(U).
\]
Thus $w_c(\chi,\beta_c)$ is determined by $\chi$, $\beta_c$, and the corresponding finite crossing table in Appendix~A.
\end{definition}

\subsection{Crossing weight}\label{subsec:crossing-weight}

We obtain the crossing weight by summing the local contributions over all compatible local colourings at the crossing.

\begin{definition}[Crossing weight]\label{def:crossing-weight}
Fix a first-stage admissible colouring $\chi\in\Adm_1(D;r)$ and a crossing $c$.  Define the crossing weight by
\[
  W_c(\chi)
  :=\sum_{\beta_c\in\Col_c(\chi)} w_c(\chi,\beta_c).
\]
If $\Col_c(\chi)$ is empty, set $W_c(\chi):=0$.
\end{definition}

\begin{lemma}[Representative matching]\label{lem:representative-matching}
Any two choices of rotation or reflection and simultaneous sheet relabelling used in Section~3.2 to match the representative in \cref{fig:crossing-types} with an actual crossing $c$ induce a weight-preserving bijection between the resulting sets of compatible local colourings.  Consequently, $W_c(\chi)$ is independent of these choices.
\end{lemma}

\begin{proof}
Fix the actual crossing $c$, including its four transposition labels.  We compare only the different ways of matching this fixed crossing with the representative in \cref{fig:crossing-types}.  Two such matchings differ by one of the four plane symmetries preserving the over-strand and under-strand: the identity, the reflections in the two strand axes, and their composition, which is a half-turn.  For type~1 the four relative sheet relabellings are $\mathrm{id},(12),(34),(12)(34)$; the six possible transposition labels of an actual type~1 crossing have already been accounted for by the initial matching in Section~3.2 and are not additional ambiguities.  For type~2 the three incident labels determine the relative sheet relabelling uniquely, whereas for type~3 the two entries in each of the two disjoint transpositions may be interchanged independently.  Thus the numbers of relative cases are respectively $4\cdot4=16$, $4\cdot1=4$, and $4\cdot4=16$.

In each case the plane symmetry and sheet relabelling identify the boundary colours in the relevant table, with the $q$-indices transported as in Section~3.2.  Direct comparison of Tables~A.2--A.4 then gives a bijection of the local variables $x_k$ under which the rows of $\mathcal{A}_t$ agree up to reordering their three entries, and the rows of $\mathcal{S}_t$ agree up to relabelling the four vertices of a tetrahedron.  Hence the products of the $E$, $\Theta$, and $\Tet$ factors are unchanged, so the bijection preserves $w_c(\chi,\beta_c)$ term by term and therefore preserves $W_c(\chi)$.

This proves the statement using only the tables in Appendix~A.  A separately supplied machine-readable certificate reproduces the same $36$ finite comparisons for independent verification; it is not an additional assumption in the proof.
\end{proof}

\section{The diagrammatic state sum and the main theorem}

This section defines the global factor, combines it with the crossing weights
to form the diagrammatic state sum, and states the reconstruction theorem.

\subsection{The global diagram factor}

\begin{definition}[Global diagram factor]\label{def:global-diagram-factor}
Fix a first-stage admissible colouring
$\chi\in\Adm_1(D;r)$.  For every graph edge $e$, take $E(p_i(e))$ for
each first sheet label $i$ and $E(q_j(e))$ for each second sheet label $j$.
For every complementary region $R$ and every first sheet label $i$, also
take $E(b_i(R))$.

Use the side numbering $R_1(e),R_2(e)$ from Section~2.1.
For $a\in\{1,2\}$ and $i\in\{1,2,3,4\}$, take the triangle weight
$\Theta\bigl(p_i(e),q_{\mu_{R_a(e)}(i)}(e),b_i(R_a(e))\bigr)$.  The
\emph{global diagram factor}
$W_D(\chi)$ is the product containing each of these factors once:
\begin{align}
W_D(\chi)
={}&
\prod_{e\in E(D)}
  \left(\prod_{i=1}^{4}E(p_i(e))\right)
  \left(\prod_{j=1}^{4}E(q_j(e))\right)
\prod_{R\in\mathcal R(D)}\prod_{i=1}^{4}E(b_i(R))
\notag\\
&\times
\prod_{e\in E(D)}\prod_{a=1}^{2}\prod_{i=1}^{4}
  \Theta\bigl(p_i(e),q_{\mu_{R_a(e)}(i)}(e),b_i(R_a(e))\bigr).
\label{eq:global-diagram-factor}
\end{align}
\end{definition}

Suppose that $\Gamma_D$ is connected and has $n$ vertices.  Transitivity of
the monodromy action implies $n>0$.
We define the normalization exponent
\begin{equation}\label{eq:normalization-exponent}
N(D):=8+8n.
\end{equation}
Proposition~\ref{prop:global-vertex-count} will identify $N(D)$ with the
number of global vertices of the triangulation $K(D)$.

\subsection{The diagrammatic state sum}\label{subsec:diagrammatic-state-sum}

We multiply the global diagram factor by all crossing weights and sum the
result over the first-stage admissible colourings.  We use the Turaev--Viro
parameters $(r,s)$ and the vertex normalization $\eta_{r,s}$ fixed in
Section~3.3.

\begin{definition}[Diagrammatic state sum]\label{def:diagrammatic-state-sum}
Fix the base region $R_0$ from Section~2.3.  For a first-stage admissible
colouring $\chi\in\Adm_1(D;r)$, let $W_D(\chi)$ be the global diagram factor
of Definition~4.1 and let $W_c(\chi)$ be the crossing weight of
Definition~3.3 for each crossing $c$.  The \emph{diagrammatic state sum}
relative to $R_0$ is
\begin{equation}\label{eq:diagrammatic-state-sum}
Z_{r,s}^{R_0}(D)
:=
\eta_{r,s}^{N(D)}
\sum_{\chi\in\Adm_1(D;r)}
W_D(\chi)
\prod_{c\in\Cross(D)}W_c(\chi).
\end{equation}
\end{definition}

\begin{remark*}[Notation for the base region]
Definition~4.2 gives $Z_{r,s}^{R_0}(D)$ for the base region fixed in
Section~2.3.  We write $Z_{r,s}(D)$ when this choice is understood and retain
the superscript only when comparing choices.  After Theorem~4.4 has been
proved for an arbitrary fixed $R_0$, Corollary~7.5 proves independence of the
base region.
\end{remark*}

Substituting
\[
W_c(\chi)=\sum_{\beta_c\in\Col_c(\chi)}w_c(\chi,\beta_c)
\]
from Definition~3.3 into Definition~4.2 expands the state sum over simultaneous
choices of local colourings at all crossings.

For an admissible diagram colouring
\[
\theta=\bigl(\chi,\{\beta_c\mid c\in\Cross(D)\}\bigr)
\in\Adm_2(D;r)
\]
from Definition~3.1, define its weight by
\begin{equation}\label{eq:admissible-diagram-colouring-weight}
w_{r,s}(\theta)
:=
W_D(\chi)
\prod_{c\in\Cross(D)}w_c(\chi,\beta_c).
\end{equation}

\begin{proposition}[Expansion over admissible diagram colourings]
\label{prop:complete-colouring-expansion}
With the fixed base region understood and its superscript suppressed as
above, the diagrammatic state sum of Definition~4.2 is the finite sum
\begin{equation}\label{eq:complete-colouring-expansion}
Z_{r,s}(D)
=
\eta_{r,s}^{N(D)}
\sum_{\theta\in\Adm_2(D;r)}w_{r,s}(\theta).
\end{equation}
\end{proposition}

\begin{proof}
All sums are finite.  Expand each $W_c(\chi)$ in Definition~4.2 as the sum
over $\Col_c(\chi)$.  Once $\chi$ is fixed, the local colourings $\beta_c$ at
different crossings can be chosen independently.  Thus the terms in the
expanded product are in one-to-one correspondence with the admissible diagram
colourings $\theta\in\Adm_2(D;r)$, and the corresponding term is
$w_{r,s}(\theta)$.  This gives the displayed sum.
\end{proof}

\subsection{The reconstruction theorem}\label{subsec:reconstruction-theorem}

We now state the main result for a covering presentation with a connected
projection graph.  Its proof is given in Section~7.

\begin{theorem}[Reconstruction theorem]\label{thm:reconstruction-connected}
Let $D$ be a diagram of a four-fold simple branched covering presentation
whose monodromy action is transitive, and let $M(D)$ be the closed
$3$-manifold represented by $D$.  Assume that the projection graph
$\Gamma_D$ is connected.  For any choice of $R_0\in\mathcal R(D)$ as the
base region, let
$Z_{r,s}^{R_0}(D)$ be the diagrammatic state sum of Definition~4.2 formed
using $R_0$.  If $r\geq3$ and $s$ is an integer satisfying
\[
  0<s<2r,
  \qquad
  \gcd(r,s)=1,
\]
then
\begin{equation}\label{eq:reconstruction-connected}
Z_{r,s}^{R_0}(D)=\TV_{r,s}(M(D)),
\end{equation}
where the right-hand side denotes the Turaev--Viro invariant with the
conventions fixed in Section~5.
\end{theorem}

\section{The Turaev--Viro state sum}\label{sec:turaev-viro-state-sum}

This section assigns the coefficients from Section~3.3 to the global
simplices and admissible edge colourings of a closed generalized triangulation,
fixing the Turaev--Viro state sum used in the comparison.

\subsection{Generalized triangulations and local weights}
\label{subsec:generalized-triangulations}

A \emph{closed generalized triangulation} $K$ is face-pairing data obtained by
pairing the triangular faces of finitely many abstract tetrahedra in pairs by
affine maps, such that the geometric realization is a closed $3$-manifold.
Distinct vertices or edges of one tetrahedron may become globally identified.
We write
\[
\mathcal V(K),\qquad \mathcal E(K),\qquad
\mathcal F(K),\qquad \mathcal T(K)
\]
for the sets of global vertex, edge, triangle, and tetrahedron classes,
respectively.

An \emph{edge colouring} of $K$ is a map
\[
\varphi\colon\mathcal E(K)\longrightarrow C_r.
\]
It is admissible if the three edge colours of every global triangle class are
admissible.  We denote the set of admissible edge colourings by $\Adm(K)$.
This is a different object from an admissible diagram colouring $\theta$ in
Definition~3.1; Theorem~7.1 will give the correspondence between them.

Fix $\varphi\in\Adm(K)$.  Each global vertex class contributes one factor
$\eta_{r,s}$, and each global edge class $e$ contributes
$E(\varphi(e))$.  For a global triangle class $F$, write
$\varphi|_{\partial F}$ for the triple of colours on its three edges.  Its
order does not affect $\Theta$, and $F$ contributes
$\Theta(\varphi|_{\partial F})$.  A global tetrahedron class $T$ contributes
$\Tet(\varphi|_T)$, where $\varphi|_T$ denotes its six edge colours.  The
coefficients $\eta_{r,s}$, $E$, $\Theta$, and $\Tet$ are those fixed in
Section~3.3.

For each abstract tetrahedron, order its vertices by
$0<1<2<3$.  If $x_{mn}$ is the colour of the edge joining vertices $m$ and
$n$, set
\[
\varphi|_T=(x_{01},x_{02},x_{03},x_{12},x_{13},x_{23}).
\]
This is the order in which the six colours are supplied to the tetrahedron
weight in Section~3.3.  Figure~\ref{fig:tetrahedron-edge-order} shows the
vertex and edge labels.

\begin{figure}[H]
  \centering
  \includegraphics[width=.36\textwidth]{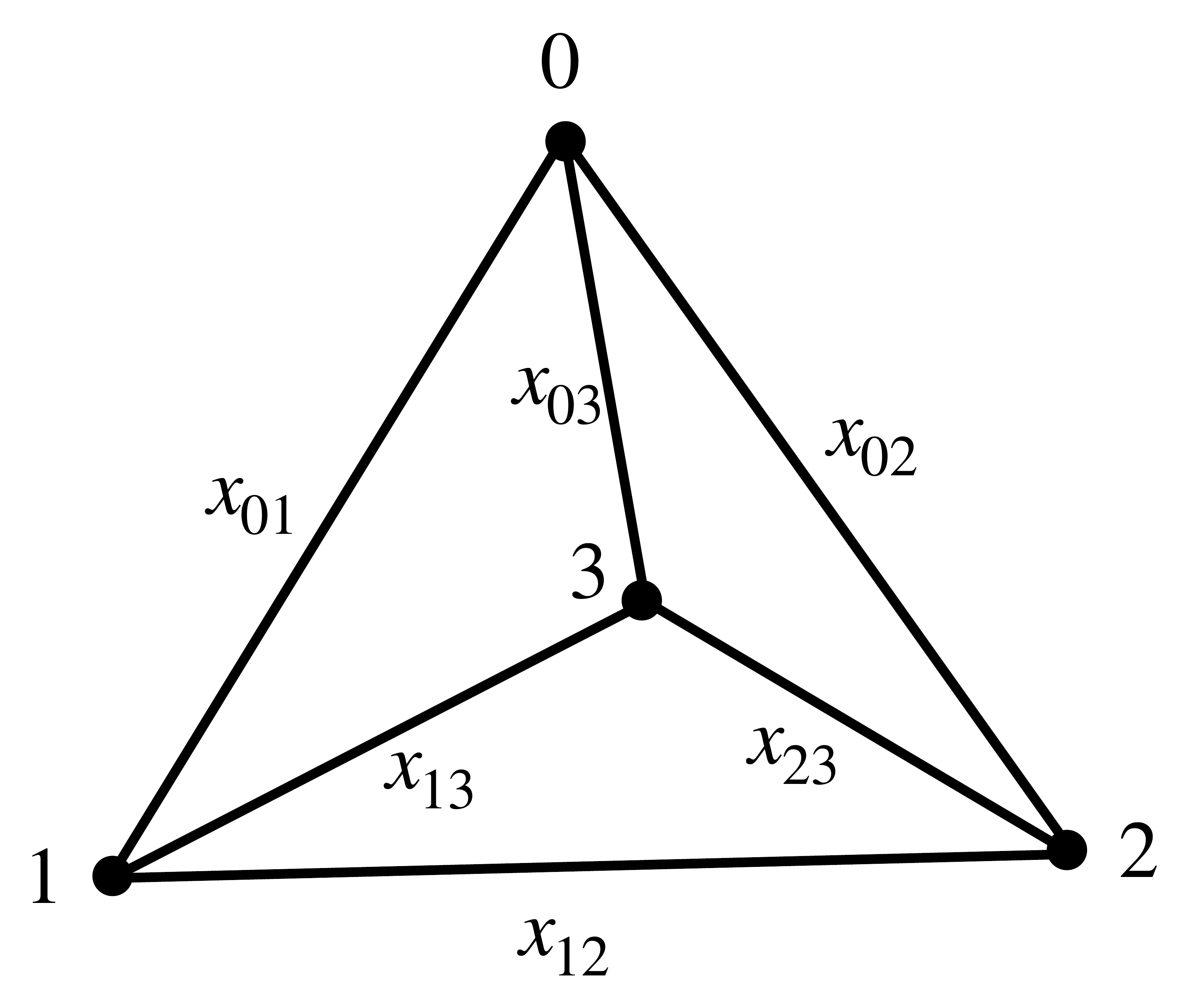}
  \caption{Vertex and edge labels of an ordered tetrahedron.  The planar
  position of vertex $3$ is schematic.}
  \label{fig:tetrahedron-edge-order}
\end{figure}

\subsection{The state sum for a closed generalized triangulation}%
\label{subsec:closed-tv-state-sum}

We now combine the local weights from Section~5.1 to form the Turaev--Viro
state sum of a closed generalized triangulation.

Following \cite[Section~2.1]{burton-maria-spreer-2018}, we use the expression
in which the factor assigned to each global vertex, edge, triangle, and
tetrahedron class appears once.  With the conventions used here, no
square-root factors occur.  For a closed generalized triangulation $K$, set
\begin{equation}\label{eq:closed-tv-state-sum}
\begin{aligned}
\TV_{r,s}(K)
={}&\eta_{r,s}^{\lvert\mathcal V(K)\rvert}
\sum_{\varphi\in\Adm(K)}
\prod_{e\in\mathcal E(K)}E(\varphi(e)) \\
&{}\times
\prod_{F\in\mathcal F(K)}\Theta(\varphi|_{\partial F})
\prod_{T\in\mathcal T(K)}\Tet(\varphi|_T).
\end{aligned}
\end{equation}
Here $\lvert\mathcal V(K)\rvert$ is the number of global vertex classes,
and each product runs once over the indicated global simplex classes.  The
invariance theorem for the Turaev--Viro state sum shows that this quantity is
independent of the chosen generalized triangulation.  It therefore depends
only on the closed $3$-manifold realized by $K$ and is the Turaev--Viro
invariant of that manifold
\cite{turaev-viro-1992,burton-maria-spreer-2018}.

\begin{remark}[No subdivision is needed]\label{rem:no-subdivision}
The particular generalized triangulation $K(D)$ constructed later in
Definition~\ref{def:face-pairing-complex} pairs all triangular faces of its
abstract tetrahedra by affine maps.  The state sum in
\cite[Section~2.1]{burton-maria-spreer-2018} therefore applies directly to
this face-pairing data; subdivision into an ordinary simplicial complex is
not required.
\end{remark}

\section{Local triangulations and global gluing}
\label{sec:triangulation}

This section constructs a closed generalized triangulation of the covering
manifold.  Following \cite[Section~6.1 and Figures~23--24]{hatakenaka-2010},
we glue abstract quadrilaterals along the dual graph, suspend them to form
bananas, and glue their lifted local models by the induced boundary markings.
We begin with the bananas and their boundary triangles, assuming as in
Section~2.3 that $\Gamma_D$ is connected.

\subsection{Bananas and their boundary triangles}
\label{subsec:bananas}

For each crossing $c$, take an abstract quadrilateral $U_c$.  Its four sides
are indexed in cyclic order by the four half-edges incident to $c$, and a
point $v(h)$ is placed on the side indexed by $h$.  The four vertices of
$U_c$ correspond to the four local sectors at $c$.  Suppose that a graph
edge $e$ has half-edges $h$ and $k$ at crossings $c$ and $d$, respectively.
Identify the corresponding sides of $U_c$ and $U_d$ so that the complementary
regions on their two sides agree.  The quotient obtained after all such
identifications is the dual cell decomposition of the embedded graph
$\Gamma_D$, and hence is $S^2$.  For a loop edge, two distinct sides of the
same quadrilateral are identified.
Thus $U_c$ is always treated as an abstract quadrilateral; its image in the
quotient need not have a boundary that is a simple closed curve.

Take points $P$ and $Q$ on the two sides of $S^2$.  Cone $U_c$ from $P$ and
from $Q$, and join the two cones along their common base $U_c$.  The resulting
abstract $3$-ball is denoted by $B_c$ and called a \emph{banana}.  Coning the
side identifications of the quadrilaterals gives face identifications of the
bananas.  Their quotient is the suspension of $S^2$, and therefore is $S^3$.
For a loop edge, two distinct faces of the same banana are identified.  If
the local crossing tangle determined by the over--under information of $D$ is
placed in each $B_c$, the quotient contains the labelled link represented by
$D$ \cite[Figures~23--24]{hatakenaka-2010}.  The quadrilateral and banana at a
crossing are shown in \cref{fig:banana}.

\begin{figure}[!tbp]
  \centering
  \includegraphics[width=.76\textwidth]{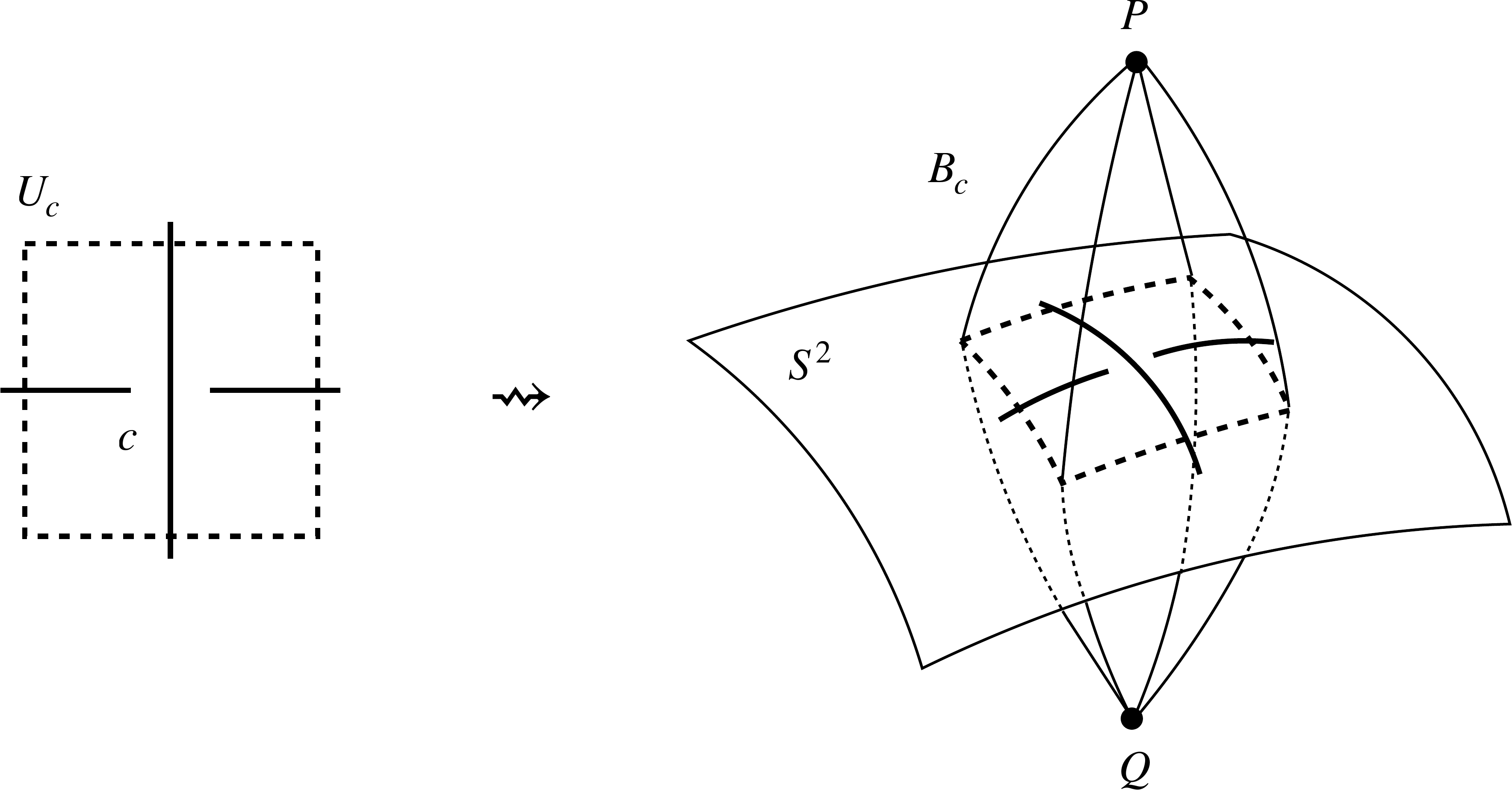}
  \caption{The abstract quadrilateral $U_c$ associated with a crossing $c$
  (dashed square, left) and the banana $B_c$ (right).  The local crossing
  tangle lies inside $B_c$.}
  \label{fig:banana}
\end{figure}

Each vertex of $U_c$ corresponds to a local sector $R$.  On the boundary of
$B_c$, that vertex determines a $P$--$Q$ edge associated with $R$.  The
suspension face of the side indexed by a half-edge $h$ contains $v(h)$.
Joining $v(h)$ to $P$ and $Q$ divides this face into two PL triangles,
corresponding to the two local sectors adjacent to $h$.  Consequently,
$\partial B_c$ consists of eight such boundary triangles.  If the same global
complementary region occupies more than one local sector, the associated
$P$--$Q$ edges are nevertheless distinct before the bananas are glued.

Write the four half-edges incident to $c$ as
$h_1,h_2,h_3,h_4$ in the common cyclic order used in
Figures~\ref{fig:crossing-positions} and~\ref{fig:banana-boundary}.  The
subscripts $1,2,3,4$ record these cyclic positions; they are not sheet labels
from Section~2.3.  The point on the side indexed by $h_k$ is also denoted by
$v(h_k)$ when regarded as a boundary vertex of $B_c$.  Thus the four vertices
in \cref{fig:banana-boundary} are $v(h_k)$ for $k=1,2,3,4$.

As in Figure~\ref{fig:crossing-positions}, denote the global complementary
regions containing the local sectors between the consecutive pairs
$(h_1,h_2)$, $(h_2,h_3)$, $(h_3,h_4)$, and $(h_4,h_1)$ by
$R_{12},R_{23},R_{34},R_{41}$, respectively.  The same global region may
occupy more than one of these sectors, but the four local sectors remain
distinct.  Each sector determines one boundary edge with endpoints $P$ and
$Q$.  Figure~\ref{fig:banana-boundary}(b) shows these four edges, and
Figure~\ref{fig:banana-boundary}(c) displays the two boundary triangles
sharing each such edge as a diamond.

\begin{definition}[Base triangle]\label{def:base-triangle}
Let $h\in\mathcal H(D)$ be a half-edge at $c$, and choose one of the two local
sectors $R$ adjacent to $h$.  The boundary triangle of $B_c$ with vertices
$P,Q,v(h)$ is denoted by $\Delta_{h,R}$; see
\cref{fig:banana-boundary}(c).  Here $R$ is one of the local positions
$R_{12},R_{23},R_{34},R_{41}$.  For example, the two triangles forming the
$R_{12}$ diamond are $\Delta_{h_1,R_{12}}$ and
$\Delta_{h_2,R_{12}}$.
\end{definition}

\begin{figure}[!tbp]
  \centering
  \includegraphics[width=.96\textwidth]{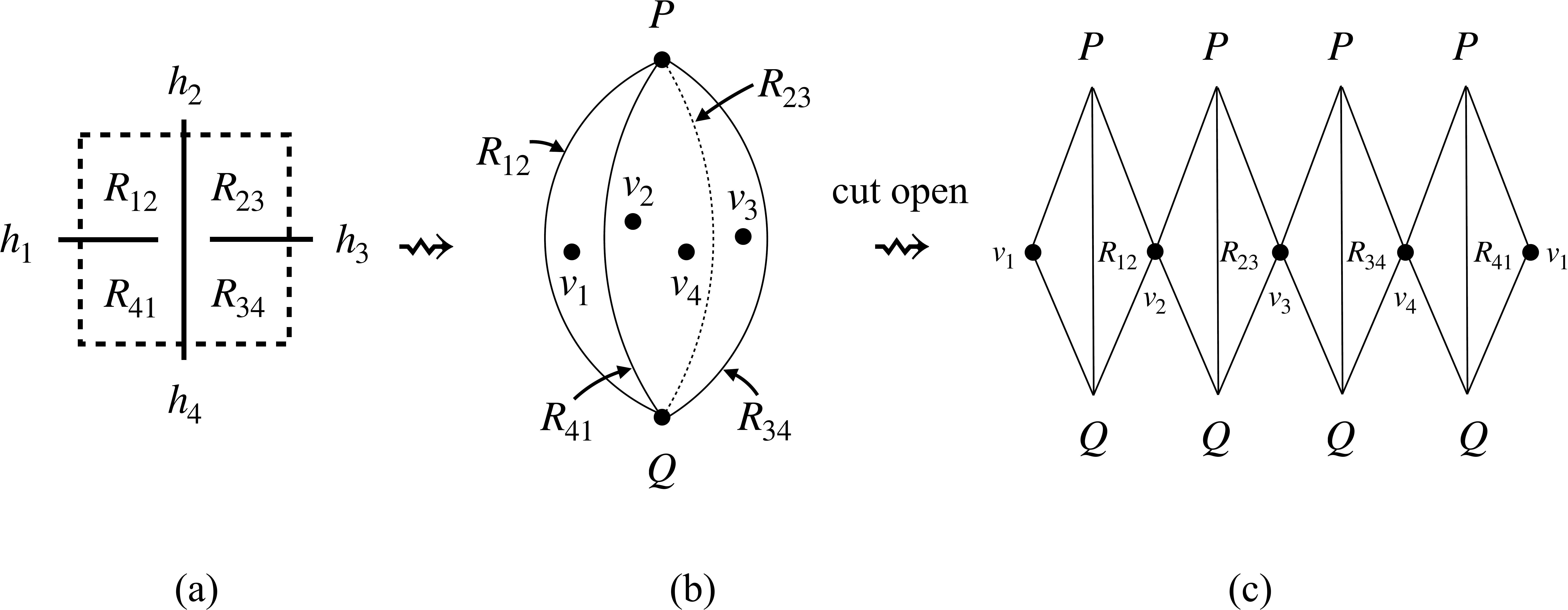}
  \caption{The boundary subdivision of $B_c$.  (a) Half-edges and local
  sectors at $c$.  (b) Boundary vertices and the four $P$--$Q$ edges.
  (c) The boundary cut into four diamonds, each consisting of two triangles.}
  \label{fig:banana-boundary}
\end{figure}

\subsection{Lifted boundary triangles}
\label{subsec:lifted-boundary-triangles}

We next describe the lifts of the boundary triangles of a banana and record
the correspondence between their vertex labels.  This is the boundary data
used when the lifted local models are glued.

Let $p_D\colon M(D)\longrightarrow S^3$ be the four-fold simple branched
covering map.  The gluing in Section~6.1
gives each $B_c$ a natural map to $S^3$.  Let $\widetilde B_c$ be the pullback
of $p_D$ along this map.  Thus a point of $\widetilde B_c$ is a pair consisting
of a point of $B_c$ and a point of $M(D)$ that have the same image in $S^3$.

Fix one of the representative labellings in
Figure~\ref{fig:crossing-types}, and consider the corresponding
$\widetilde B_c$.  Denote the four lifts of $P$ by
$P_1,P_2,P_3,P_4$ and call their indices the \emph{local sheet labels}.
Likewise, denote the four lifts of $Q$ by $Q_1,Q_2,Q_3,Q_4$.

Normalize the $Q$-labels so that, over the representative sector $R_{12}$,
the lift of the $P$--$Q$ edge beginning at $P_i$ ends at $Q_i$.  The branch
labels then determine the $Q$-endpoint over each of the other three sectors.
When the standard local model is placed at an actual crossing, its local
$P$-labels are uniformly relabelled as global first sheet labels.  Over a
sector occupying a global complementary region $R$, the lifted edge beginning
at the global $P_\ell$ then ends at the global $Q_{\mu_R(\ell)}$.  Thus these
local endpoint assignments are the representative-model form of the region
transport defined in Section~2.3.

Let $h$ be a half-edge, let $e$ be the graph edge containing it, and let
$\tau_e$ be the transposition label of $e$.  The orbit of a sheet label $i$
under $\tau_e$ is denoted by
\[
  I=[i]_{\tau_e}:=\{i,\tau_e(i)\}.
\]
The point $v(h)$ is a branch endpoint on the boundary of the banana.  Denote
the lifted vertex over $v(h)$ corresponding to the orbit $I$ by $v(h)^I$.

For $i\in\{1,2,3,4\}$, define $\Delta_{h,R}^{\,i}$ to be the lift of the base
triangle $\Delta_{h,R}$ that contains $P_i$.  If the entry for the sector
$R$ in Appendix~B sends the $P$-index $i$ to the $Q$-index $j$, its ordered
vertices are
\[
  \bigl(P_i,Q_j,v(h)^I\bigr),
  \qquad I=[i]_{\tau_e}.
\]

For the example in \cref{fig:lifted-base-triangle}, choose the half-edge
$h_1$ and the adjacent local sector $R_{12}$ in
Figure~\ref{fig:banana-boundary}(a).  Thus the shaded triangle in
Figure~\ref{fig:lifted-base-triangle}(a) is
$\Delta_{h_1,R_{12}}$.  These subscripts record the local positions in
Figure~\ref{fig:banana-boundary}(a), not sheet labels.  In the representative
labelled banana considered here, the graph edge $e_1$ containing $h_1$ has
label $\tau_{e_1}=(12)$, and the normalization over $R_{12}$ makes the
$Q$-endpoint of the lift beginning at $P_i$ equal to $Q_i$.  The ordered
vertices of the four lifted triangles are
\[
\begin{gathered}
  (P_1,Q_1,v(h_1)^{12}),\qquad (P_2,Q_2,v(h_1)^{12}),\\
  (P_3,Q_3,v(h_1)^{3}),\qquad (P_4,Q_4,v(h_1)^{4}).
\end{gathered}
\]

\begin{figure}[!tbp]
  \centering
  \includegraphics[width=.88\textwidth]{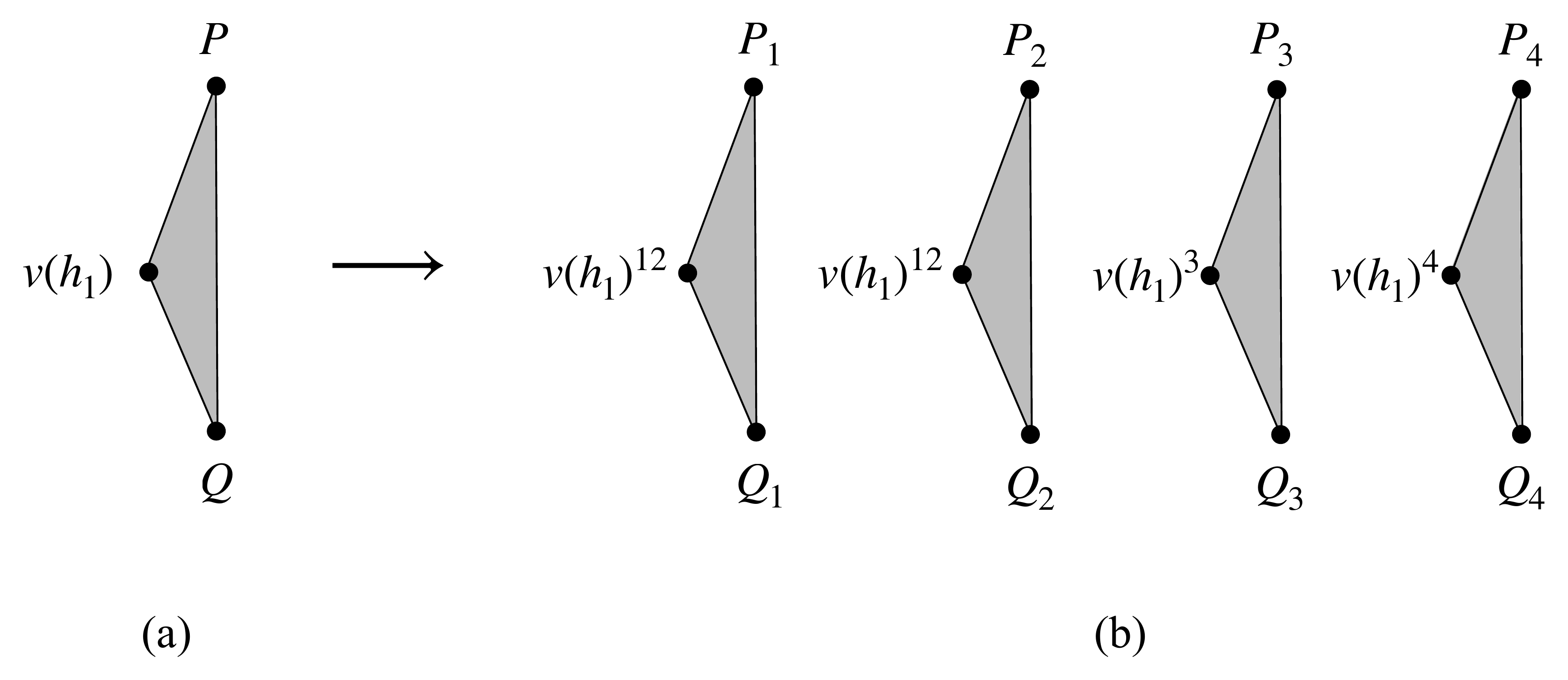}
  \caption{One boundary triangle of $B_c$ and its four lifts.
  (a) $\Delta_{h_1,R_{12}}$.  (b) The four lifted triangles for
  the normalization over $R_{12}$ and $\tau_{e_1}=(12)$.}
  \label{fig:lifted-base-triangle}
\end{figure}

The boundary of $B_c$ consists of eight base triangles $\Delta_{h,R}$, each
with four lifts because its interior is disjoint from the branch set.  Thus
the lifted boundary has $32$ triangles.  Appendix~B lists the local
$Q$-endpoint assignment for each base triangle and crossing type; their four
lifts are precisely these $32$ triangles.

\subsection{The three standard local triangulations}
\label{subsec:standard-local-triangulations}

We now introduce the boundary-marked local PL models for the three crossing
types.  They will be used in the global gluing of Section~6.5 and in the
comparison with the finite tables of Appendix~A in Section~7.

For $t\in\{1,2,3\}$, write $B^{(t)}$ for the banana with the representative
labelling of type $t$ in Figure~\ref{fig:crossing-types}; this is the
\emph{base banana} when it is distinguished from its lifts.  Write
$\widetilde B^{(t)}$ for the pullback of the covering map over $B^{(t)}$.
We construct a finite boundary-marked generalized triangulation $X^{(t)}$
whose geometric realization, denoted by $|X^{(t)}|$, is PL homeomorphic to
$\widetilde B^{(t)}$.  We call $X^{(t)}$ the \emph{standard local
triangulation} of type $t$.

The orbits of the subgroup of $S_4$ generated by all labels at the crossing
are the \emph{component orbits} $\Omega$ and index the connected components
of the pullback.  They need not equal the local edge-label orbits
$[i]_{\tau_e}$ of Section~6.2, which index lifted vertices over one boundary
point.  For example, type~2 has component orbits $\{1,2,3\}$ and $\{4\}$,
whereas a half-edge labelled $(12)$ has local orbits $\{1,2\}$, $\{3\}$, and
$\{4\}$.

Table~\ref{tab:local-components-summary} summarizes the component orbits
and their topology in $\widetilde B^{(t)}$, together with the number of
tetrahedra in the corresponding parts of $X^{(t)}$.

\begin{table}[H]
  \centering
  \small
  \setlength{\tabcolsep}{4pt}
  \renewcommand{\arraystretch}{1.15}
  \begin{tabular}{@{}c c >{\raggedright\arraybackslash}p{.43\textwidth}
      >{\centering\arraybackslash}p{.18\textwidth}@{}}
    \toprule
    Type & Shared sheets & Components of $\widetilde B^{(t)}$
      & Tetrahedra in $X^{(t)}$ \\
    \midrule
    1 & 2 & solid torus over $\{1,2\}$; $3$-balls over $\{3\}$ and $\{4\}$
      & $18+8+8$ \\
    2 & 1 & $3$-balls over $\{1,2,3\}$ and $\{4\}$
      & $24+8$ \\
    3 & 0 & $3$-balls over $\{1,2\}$ and $\{3,4\}$
      & $16+16$ \\
    \bottomrule
  \end{tabular}
  \caption{Components of $\widetilde B^{(t)}$ and tetrahedron counts in
  $X^{(t)}$.}
  \label{tab:local-components-summary}
\end{table}

The same nine-tetrahedron subdivision of the base banana is used for all
three types; it is shown in
\cref{fig:base-banana-subdivision}.  This is the subdivision of
\cite[Figure~28]{hatakenaka-2010}, expressed in the present vertex notation.

\begin{figure}[!tbp]
  \centering
  \includegraphics[width=.94\textwidth]{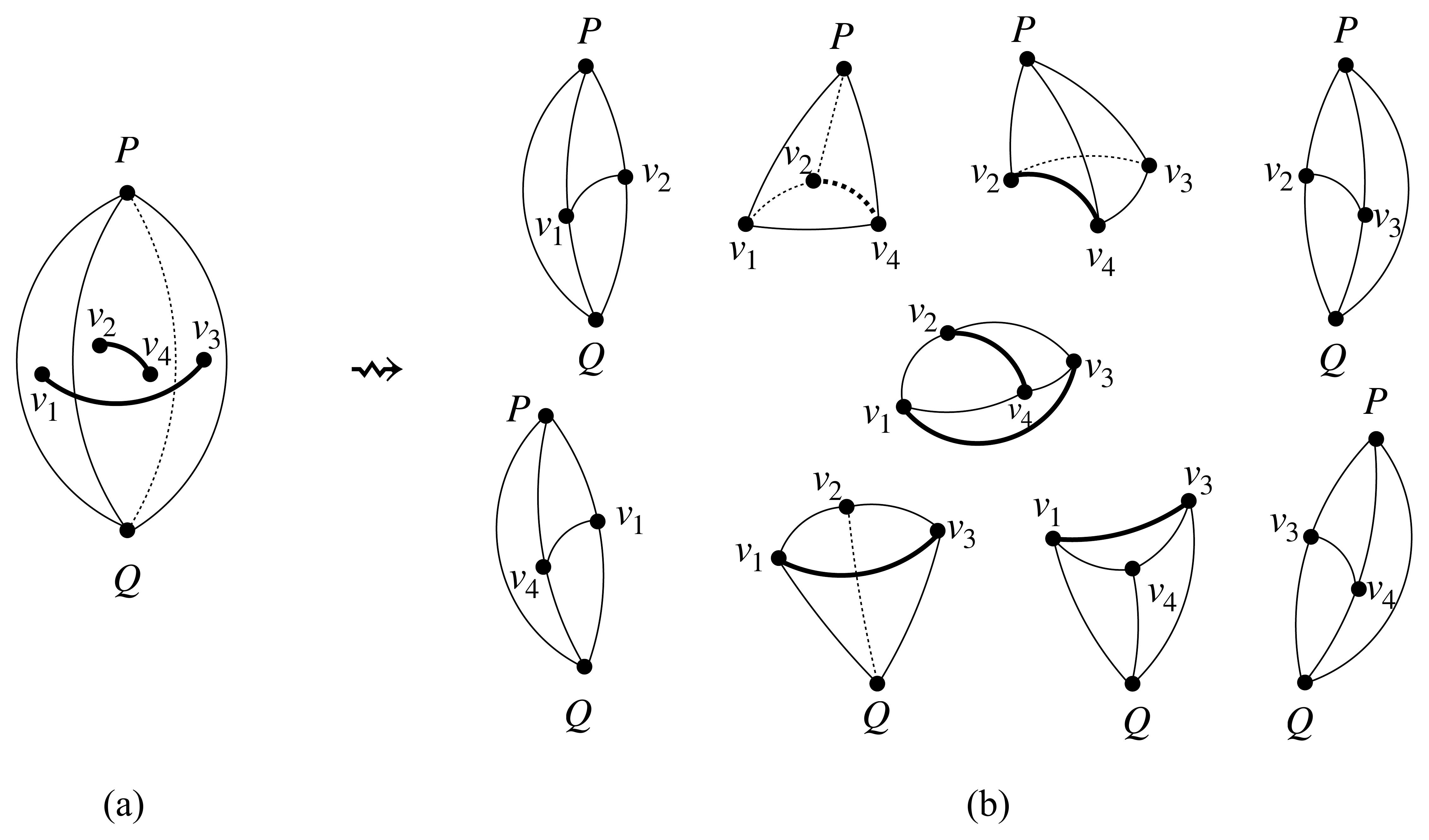}
  \caption{The base banana $B^{(t)}$ and its nine-tetrahedron subdivision.
  The thick subcomplex in (b) represents the two branch arcs; $v_k$
  abbreviates $v(h_k)$.}
  \label{fig:base-banana-subdivision}
\end{figure}

The next lemma identifies every component in
Table~\ref{tab:local-components-summary} from the two trivial branch arcs and
their labels.

\begin{lemma}[Topology of the lifted local components]
\label{lem:local-component-topology}
Let $\Omega$ be a component orbit for the two branch arcs in $B^{(t)}$.
The component of $\widetilde B^{(t)}$ corresponding to $\Omega$ is a solid
torus when $t=1$ and $\Omega=\{1,2\}$.  It is a $3$-ball in all the other
cases in Table~\ref{tab:local-components-summary}.
\end{lemma}

\begin{proof}
The local covering types are recorded in
\cite[Figures~25--27]{hatakenaka-2010}.  For type~1 and
$\Omega=\{1,2\}$, the component is the double cover of a $3$-ball branched
over a trivial two-string tangle, hence is a solid torus; equivalently, the
tangle is $(D^2\times I,P\times I)$ and its double branched cover is
$A\times I$, where $A$ is an annulus
\cite[Section~2]{lackenby-2021}.  A singleton orbit gives one copy of the
base $3$-ball.  In type~2, the three sheets over $\{1,2,3\}$ are joined
successively along the two branch disks, and the result is again a
$3$-ball.  In type~3, each two-element orbit sees one trivial branch arc,
whose double branched cover is a $3$-ball.  These are exactly the remaining
cases in Table~\ref{tab:local-components-summary}.
\end{proof}

The component of $X^{(1)}$ over the component orbit $\{1,2\}$ is obtained
from two lifts of the nine-tetrahedron subdivision in
Figure~\ref{fig:base-banana-subdivision}.  In the representative type~1
model, the normalization over $R_{12}$ joins $P_i$ to $Q_i$.  Since all four
half-edge labels are $(12)$, the $Q$-endpoint alternates around the crossing:
over $R_{12}$ and $R_{34}$ the lift beginning at $P_i$ ends at $Q_i$, whereas
over $R_{23}$ and $R_{41}$ it ends at $Q_{(12)(i)}$.  These assignments
identify the internal faces within and between the two lifts and give the
eighteen-tetrahedron solid-torus component shown in
Figure~\ref{fig:type1-solid-torus}.  Appendices~C.1--C.2 record the eighteen
ordered tetrahedra and the twenty-eight internal face pairings, while
Appendices~C.3--C.5 verify the sixteen boundary faces, the deck involution
and quotient, and hence the boundary-marked identification with the
corresponding lifted component.

\begin{figure}[!tbp]
  \centering
  \includegraphics[width=.50\textwidth]{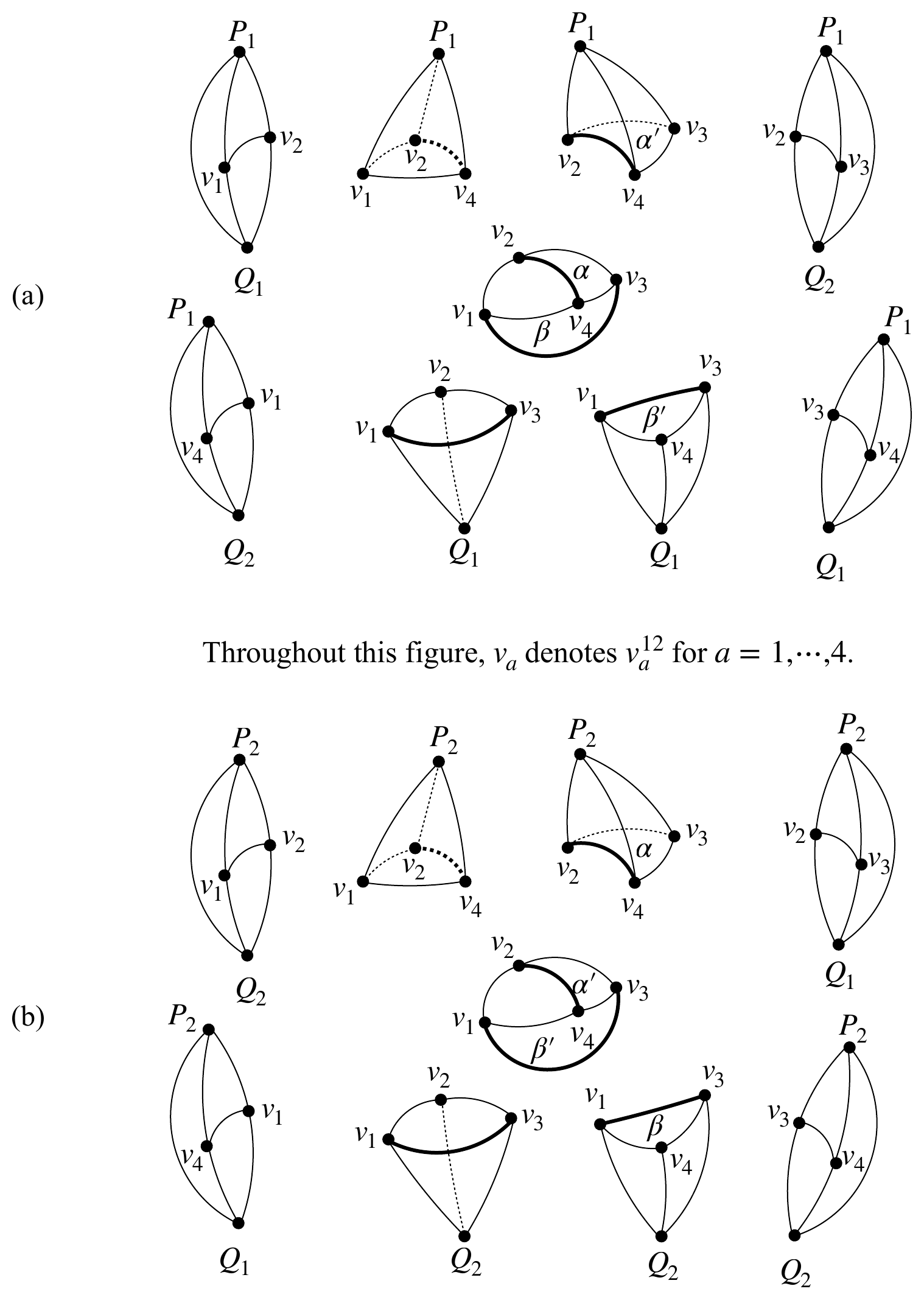}
  \caption{The eighteen-tetrahedron component of $X^{(1)}$ over
  $\Omega=\{1,2\}$.  Faces with the same vertex set or the same pairing tag
  are paired; the other sixteen faces form the boundary.  Here $v_a=v_a^{12}$.}
  \label{fig:type1-solid-torus}
\end{figure}

The sixteen unpaired faces in Figure~\ref{fig:type1-solid-torus} form the
boundary torus displayed in Figure~\ref{fig:type1-boundary-torus}.

\begin{figure}[H]
  \centering
  \includegraphics[width=.78\textwidth]{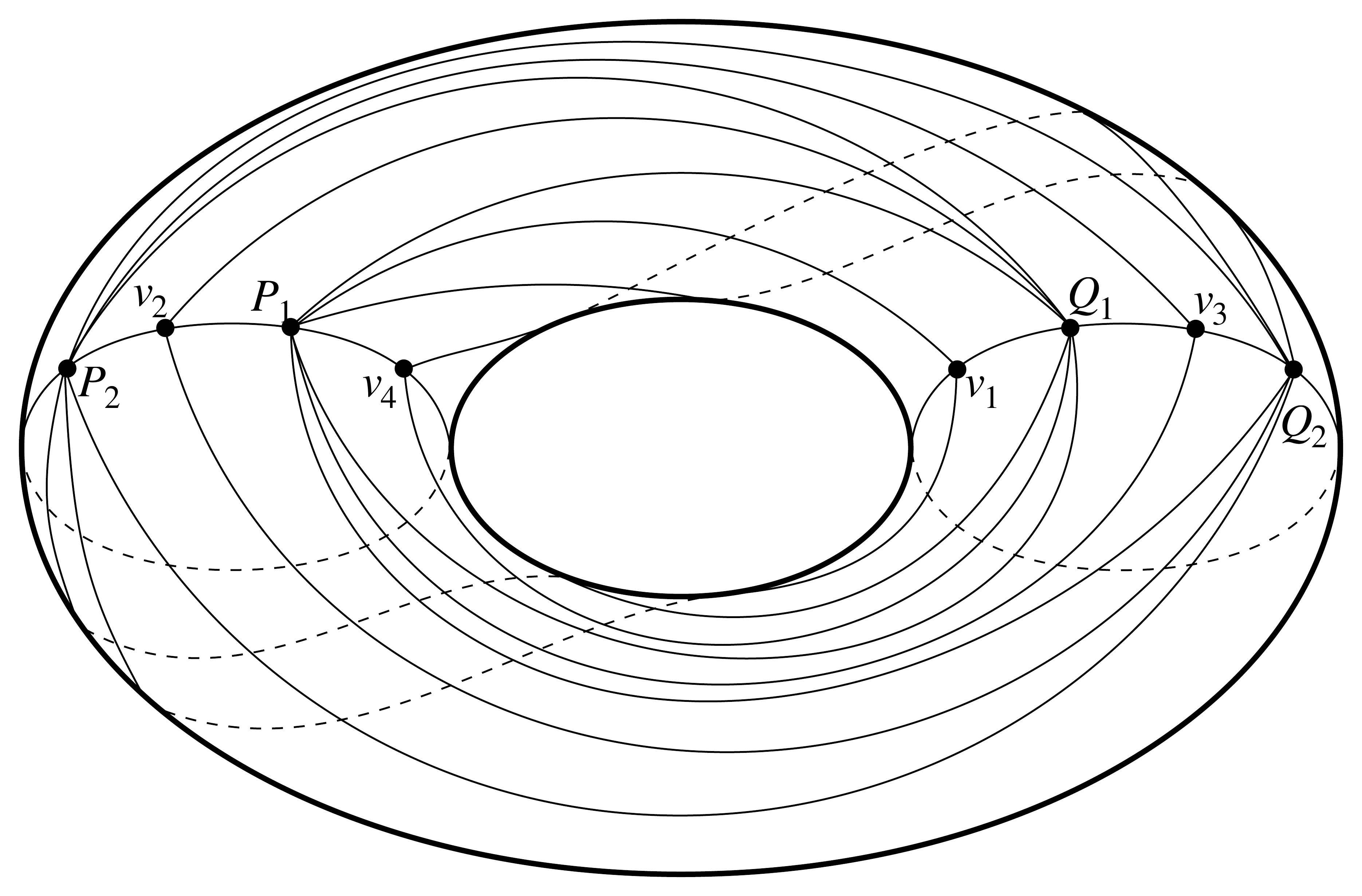}
  \caption{The boundary torus of the $X^{(1)}$-component over
  $\Omega=\{1,2\}$.  The thick outlines are not triangulation edges; the
  dashed segment indicates continuation across the back.  Here $v_a=v_a^{12}$.}
  \label{fig:type1-boundary-torus}
\end{figure}

The remaining components of the standard models $X^{(t)}$ are defined as
cones.  For the component orbit $\{1,2,3\}$ of type~2 and the two component
orbits of type~3, we cone the triangulated $2$-spheres shown in
Figure~\ref{fig:cone-boundary-spheres}.  The resulting $3$-balls contain
$24$, $16$, and $16$ tetrahedra, respectively.  In the figure, the
superscript $I$ on $v_a^I$ is the local edge-label orbit determined by the
transposition labelling the graph edge containing the corresponding
half-edge; it is distinct from the component orbit $\Omega$.

For the singleton component orbits $\{3\}$ and $\{4\}$ of type~1 and
$\{4\}$ of type~2, the lifted component is one sheet copy of the base banana
and hence is a $3$-ball.  Its boundary carries the eight-triangle subdivision
described in Section~6.1 and shown in
Figure~\ref{fig:banana-boundary}(c).  Adding one interior vertex and coning the
boundary gives a triangulation with eight tetrahedra.

\begin{figure}[H]
  \centering
  \includegraphics[width=.94\textwidth]{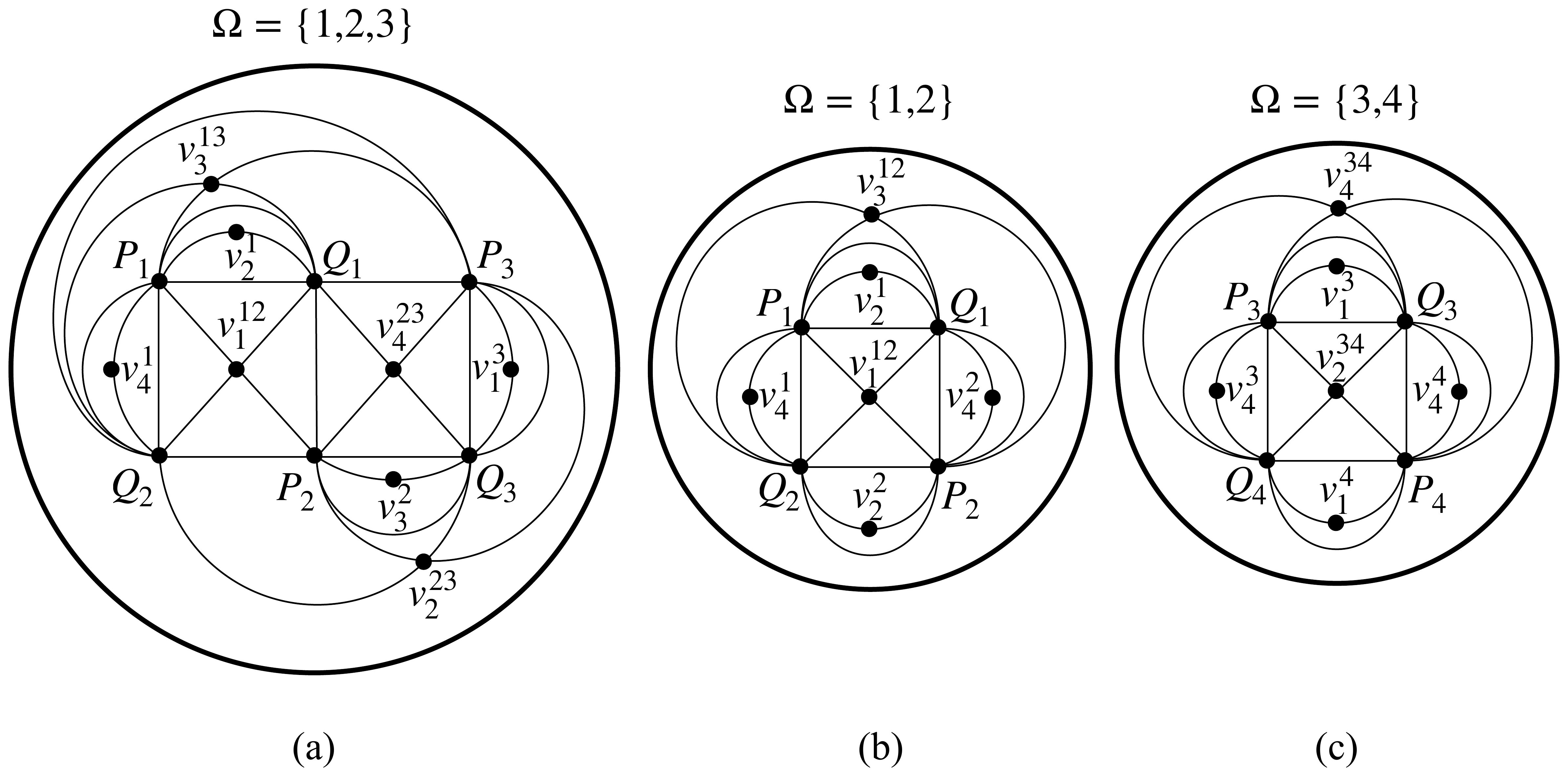}
  \caption{Boundary triangulations for the cone components of $X^{(t)}$:
  (a) type~2 over
  $\Omega=\{1,2,3\}$; (b) type~3 over $\Omega=\{1,2\}$; and (c) type~3 over
  $\Omega=\{3,4\}$.  The heavy outer circles belong only to the planar
  display.}
  \label{fig:cone-boundary-spheres}
\end{figure}

We now identify each standard local triangulation with the corresponding
component of the pullback while preserving the boundary labels.

\begin{proposition}[Standard local triangulations and local lifts]
\label{prop:standard-local-triangulations}
For each $t\in\{1,2,3\}$, let $\widetilde B^{(t)}$ be the pullback over the
labelled banana $B^{(t)}$, and let $X^{(t)}$ be the standard local
triangulation constructed above.  There is a PL homeomorphism from its
geometric realization,
\[
  \phi^{(t)}\colon |X^{(t)}|\longrightarrow \widetilde B^{(t)}
\]
that preserves the component-orbit labels and the boundary marking specified
in Section~6.2 and listed completely in Appendix~B.
\end{proposition}

\begin{proof}
For the type~1 component over $\{1,2\}$, Appendices~C.1--C.5 give the
required boundary-marked PL identification.  Every other lifted component
is a $3$-ball by Lemma~\ref{lem:local-component-topology}, while the
corresponding component of $X^{(t)}$ is the cone on its boundary $2$-sphere.
Appendix~B identifies the boundaries simplicially, and the PL Alexander
extension theorem~\cite{rourke-sanderson-1982} extends each identification
over the two $3$-balls.  The componentwise maps assemble to $\phi^{(t)}$.
\end{proof}

\subsection{Transport to an arbitrary crossing}
\label{subsec:transport-to-crossing}

Let $c$ be a crossing of type $t$.  Match the representative crossing of
type~$t$ in Figure~\ref{fig:crossing-types} with $c$ by a rotation or reflection
that takes its over-strand and under-strand to those of $c$, and relabel the
four sheets simultaneously.  Transport the boundary labels of $X^{(t)}$ by
this matching and the region transports $\mu_R$, and denote the resulting
boundary-labelled copy by $X_c$.  Thus only the boundary labels have changed:
the underlying abstract simplicial complex is unchanged, and
$X_c\cong X^{(t)}$.

\subsection{Gluing boundary triangles along graph edges}
\label{subsec:boundary-gluing}

We now glue the local models $X_c$ along corresponding lifted boundary
triangles.  Let a graph edge $e$ have half-edges $h$ and $k$ at crossings
$c$ and $d$, respectively, where $c=d$ is allowed.  For either
$R\in\{R_1(e),R_2(e)\}$ and each sheet label $i$, put
$j=\mu_R(i)$ and $I=[i]_{\tau_e}$.  Sections~6.2 and 6.4 give unique boundary
triangles with ordered vertices
\[
  \Delta_{c;h,R}^{\,i}
    =\bigl(P_i,Q_j,v(h)^I\bigr),
  \qquad
  \Delta_{d;k,R}^{\,i}
    =\bigl(P_i,Q_j,v(k)^I\bigr),
\]
one at each end of $e$.

\begin{definition}[The face-pairing complex]
\label{def:face-pairing-complex}
Let
\[
  g_{e,R,i}\colon
  \Delta_{c;h,R}^{\,i}\longrightarrow\Delta_{d;k,R}^{\,i}
\]
be the simplicial isomorphism determined by
\[
  P_i\text{ at }c\longmapsto P_i\text{ at }d,\qquad
  Q_j\text{ at }c\longmapsto Q_j\text{ at }d,\qquad
  v(h)^I\longmapsto v(k)^I.
\]
Perform these pairings for both regions incident to every graph edge and for
all four sheet labels.  If $e$ is a loop, the two triangles are distinct
boundary faces of the same local model $X_c$.  These pairings define a finite
face-pairing complex $K(D)$.  Its geometric realization is denoted by
$|K(D)|$; explicitly,
\[
  |K(D)|=\left(\bigsqcup_{c\in\Cross(D)}|X_c|\right)\big/\!\sim,
\]
where $\sim$ is generated by the face pairings above.
\end{definition}

Figure~\ref{fig:boundary-triangle-pairing} illustrates this pairing for a
non-loop edge.

\begin{figure}[H]
  \centering
  \includegraphics[width=.88\textwidth]{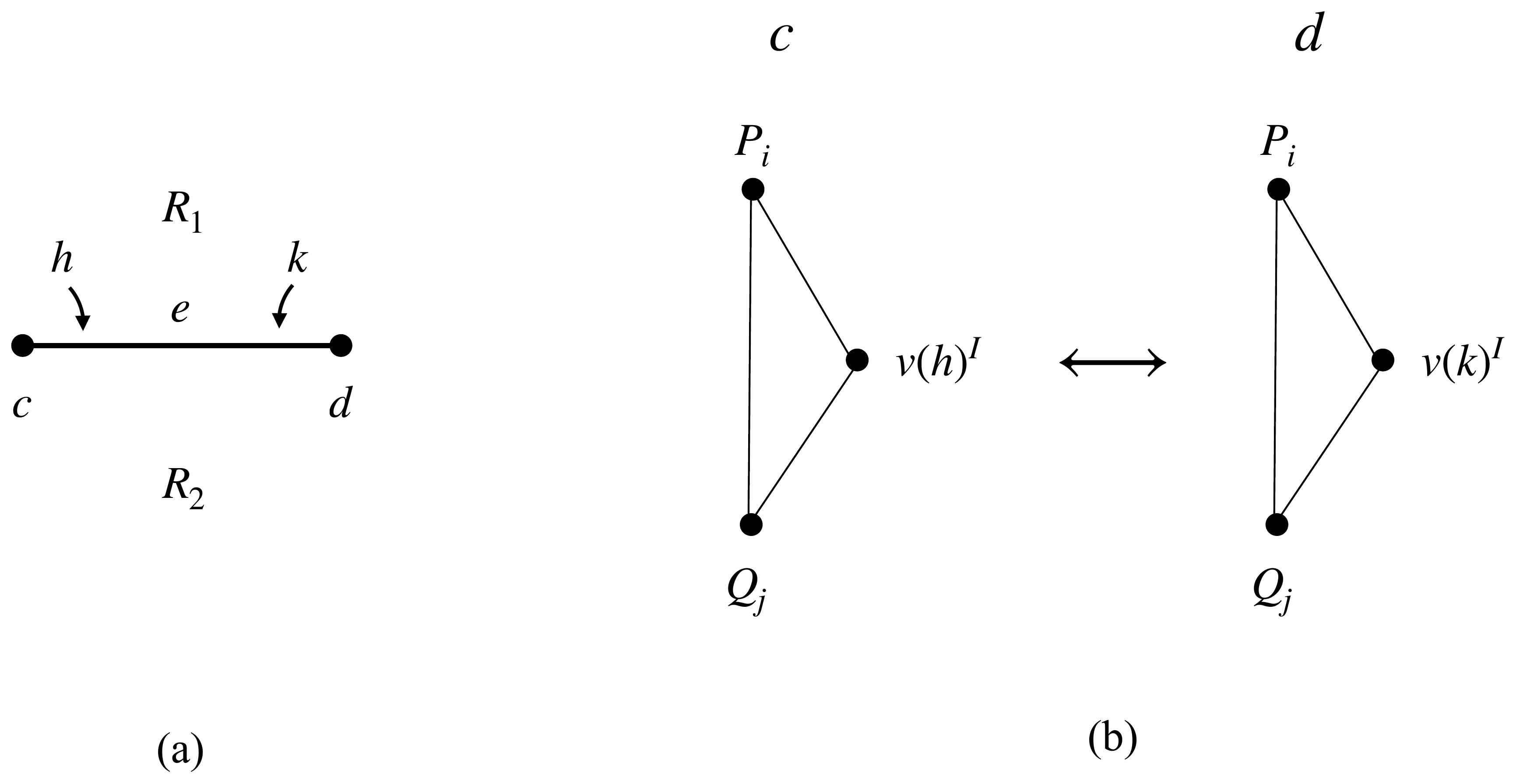}
  \caption{Boundary-triangle pairing across a non-loop graph edge $e$.
  For $R\in\{R_1(e),R_2(e)\}$, $j=\mu_R(i)$ and
  $I=[i]_{\tau_e}$.}
  \label{fig:boundary-triangle-pairing}
\end{figure}

Every boundary triangle belongs to exactly one of these pairings: it is
specified by its half-edge, adjacent region, and sheet label, and the other
end of the graph edge determines its unique partner.  Thus no boundary face
remains unpaired.  By
Corollary~\ref{cor:base-region}, changing the base region or uniformly
renumbering the sheets only relabels its vertices and gives a combinatorially
isomorphic face-pairing complex.

\subsection{Identification with the covering manifold}
\label{subsec:covering-identification}

We now identify the two spaces.

Let $\widetilde B_c$ be the pullback introduced in Section~6.2.  By construction,
gluing the spaces $\widetilde B_c$ along the lifted banana faces gives
$M(D)$.  Proposition~\ref{prop:standard-local-triangulations} and
Section~6.4 give boundary-marked PL homeomorphisms
$\phi_c\colon |X_c|\to\widetilde B_c$.  Because they preserve the boundary
markings, these maps take every pairing $g_{e,R,i}$ to the corresponding
lifted face pairing.

\begin{theorem}[Identification of the global complex]
\label{thm:global-complex-identification}
There is a PL homeomorphism
\[
  |K(D)|\cong_{\mathrm{PL}} M(D).
\]
\end{theorem}

\begin{proof}
The disjoint union of the maps $\phi_c$ therefore descends to a homeomorphism
$|K(D)|\to M(D)$.  After a common subdivision of the finitely many pieces,
this map and its inverse are simplexwise linear, so the homeomorphism is PL.
\end{proof}

\begin{proposition}[Global vertex count]
\label{prop:global-vertex-count}
Suppose that $\Gamma_D$ is connected and has $m$ graph edges and $n$
vertices.  The number of global vertices of $K(D)$ is
\[
  8+3m+2n=N(D).
\]
In particular, $m=2n$ and this number is $8+8n$.
\end{proposition}

\begin{proof}
The global $P$- and $Q$-vertices contribute eight classes.  The transposition
on each graph edge has one two-element orbit and two singleton orbits, so the
lifted banana vertices contribute three classes per graph edge.  Each local
model has two interior vertices, which remain distinct from all boundary
vertices and from the interior vertices of the other local models.  Hence the
total is $8+3m+2n$.  Finally, the graph has $2m$ edge ends and four edge ends
at each of its $n$ vertices, including when a graph edge is a loop.  Therefore
$2m=4n$, so $8+3m+2n=8+8n=N(D)$.
\end{proof}

\section{Proof of the reconstruction theorem}
\label{sec:proof-reconstruction}

We continue to use the base region $R_0$ fixed in Section~2.3.  This section
matches the colourings and weights in the diagrammatic and Turaev--Viro state
sums.  After the bijection of Section~7.1, Sections~7.2--7.4 identify each
summand with
$\eta_{r,s}^{N(D)}W_D(\chi)\prod_{c\in\Cross(D)}w_c(\chi,\beta_c)$ and then sum over the
local colourings, proving Theorem~\ref{thm:reconstruction-connected}.

\subsection{Correspondence of colourings}
\label{subsec:colouring-correspondence}

We first match the edge colours of $K(D)$ with the first-stage colours and
the compatible local colours in the finite tables of Section~3.

\begin{theorem}[Colouring correspondence]
\label{thm:colouring-correspondence}
The set $\Adm(K(D))$ of admissible edge colourings is naturally bijective
with the set $\Adm_2(D;r)$ of admissible diagram colourings.
\end{theorem}

\begin{proof}
The face pairings of Definition~\ref{def:face-pairing-complex} preserve the
endpoint types, the graph edge, and the relevant sheet label.  Hence the
global edges joining a $P$-vertex to a lifted banana vertex are indexed by
$(e,i)$ and coloured by $p_i(e)$; those joining a $Q$-vertex to a lifted
banana vertex are indexed by $(e,j)$ and coloured by $q_j(e)$.

For a complementary region $R$ and a first sheet label $i$, the endpoint
rule of Sections~6.2 and 6.5 gives the local $P$--$Q$ edge from $P_i$ to
$Q_{\mu_R(i)}$.  The pairings along the boundary of $R$ join precisely the
local edges with the same $(R,i)$.  Thus the global $P$--$Q$ edges are
indexed by $(R,i)$ and coloured by $b_i(R)$.  The remaining global edges lie
inside one $X_c$ and correspond to the local colours
$x_1,\ldots,x_{m_t}$ in the type-$t$ table of Appendix~A.  These four edge
families therefore recover, and are recovered from, a first-stage colouring
$\chi$ and one local colouring $\beta_c$ at each crossing.

The boundary triangles give exactly the first-stage admissibility triples,
and the remaining triangles give those in Tables~A.2(a)--A.4(a).  Thus the
edge colouring is admissible exactly when $\chi$ is first-stage admissible
and every $\beta_c$ lies in $\Col_c(\chi)$.  Independence of the representative
matching follows from Lemma~\ref{lem:representative-matching}.  This proves
the bijection.
\end{proof}

\subsection{Crossing-local factors}
\label{subsec:crossing-local-factors}

We next identify the factors contributed by the local triangulations
$X_c$.  Proposition~\ref{prop:crossing-local-factors} gives the termwise
comparison for a complete diagram colouring.  Figure~\ref{fig:crossing-local-sum}
recalls how these termwise local contributions combine to give $W_c(\chi)$.

\begin{figure}[H]
  \centering
  \includegraphics[width=.88\textwidth]{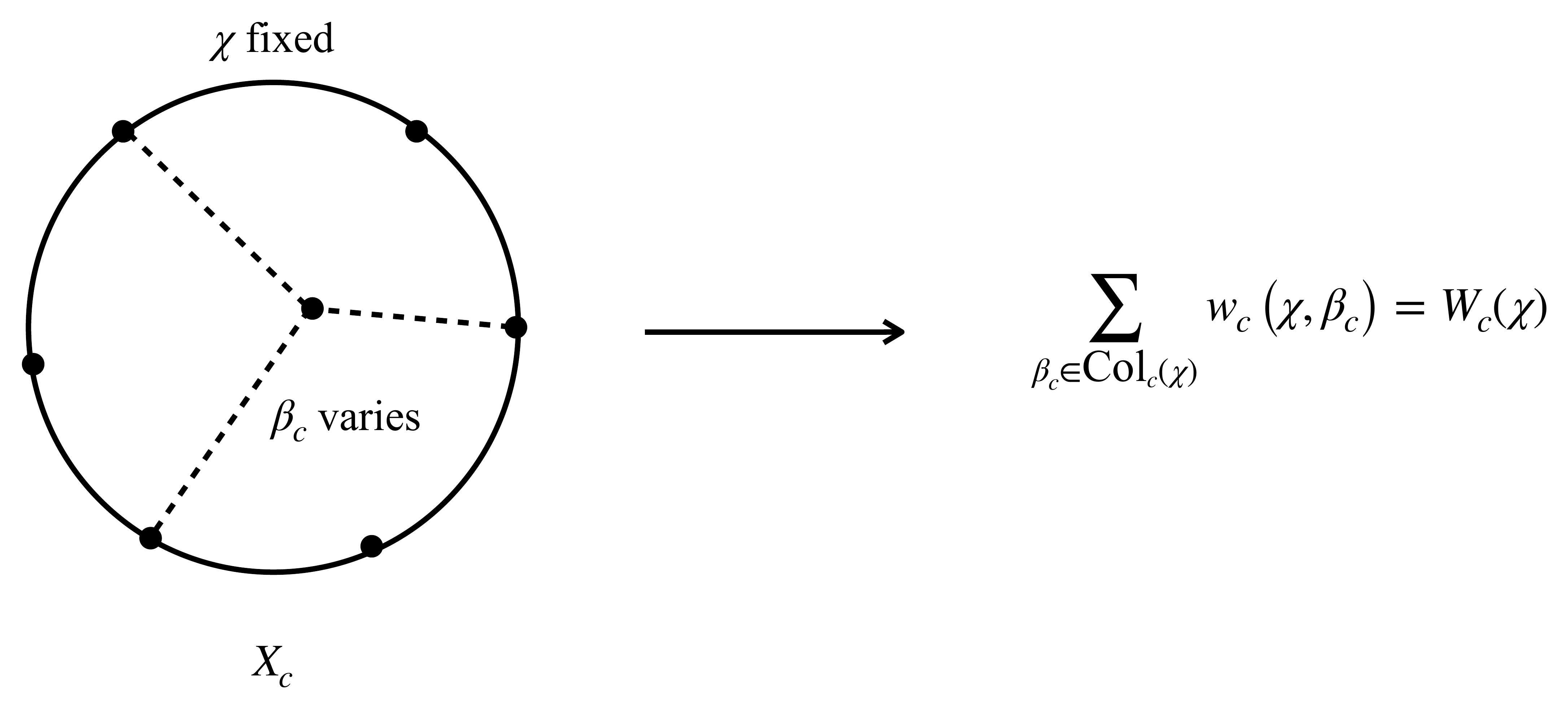}
  \caption{For fixed $\chi$, summing the local contributions over
  $\beta_c\in\Col_c(\chi)$ gives $W_c(\chi)$.}
  \label{fig:crossing-local-sum}
\end{figure}

\begin{proposition}[Crossing-local factors]
\label{prop:crossing-local-factors}
Let a diagram colouring
$\theta=(\chi,\{\beta_c\mid c\in\Cross(D)\})\in\Adm_2(D;r)$ correspond under
Theorem~\ref{thm:colouring-correspondence} to
$\varphi\in\Adm(K(D))$.  The product of the edge, triangle, and tetrahedron
weights belonging to the local triangulations $X_c$ is
\[
  \prod_{c\in\Cross(D)} w_c(\chi,\beta_c).
\]
\end{proposition}

\begin{proof}
For type $t$, Appendices~A--C list the $m_t$ interior edges, triangle classes
$\mathcal A_t$, and ordered tetrahedra $\mathcal S_t$ of $X_c$, each once
and with edge colours ordered as in Section~3.4.  The product is therefore
$w_c(\chi,\beta_c)$ by Definition~\ref{def:crossing-local-contribution};
multiplying over the crossings proves the assertion.
\end{proof}

\subsection{Global factors}
\label{subsec:global-factors}

We now identify the weights not included in the crossing-local factors,
together with the vertex normalization.

\begin{proposition}[Global factors]
\label{prop:global-factors}
Let $\theta\in\Adm_2(D;r)$ have first-stage part $\chi$, and let
$\varphi\in\Adm(K(D))$ correspond to $\theta$ under
Theorem~\ref{thm:colouring-correspondence}.
The product of all edge and triangle weights not included in the
crossing-local factors is $W_D(\chi)$, and the vertex normalization is
$\eta_{r,s}^{N(D)}$.
\end{proposition}

\begin{proof}
The three global edge families classified in the proof of
Theorem~\ref{thm:colouring-correspondence} give the three edge-weight
products in Definition~\ref{def:global-diagram-factor}.  The boundary
triangles of Section~6.5 are indexed by a graph edge, an incident region,
and a sheet label, and give the remaining triangle-weight product there.
Finally, Proposition~\ref{prop:global-vertex-count} identifies the number
of vertices with $N(D)$, giving the stated normalization.
\end{proof}

\subsection{Equality of the state sums}
\label{subsec:equality-state-sums}

Combining Sections~7.1--7.3 proves the reconstruction theorem.

\begin{proposition}
\label{prop:equality-fixed-base-region}
Fix the base region $R_0$ from Section~2.3, and construct $K(D)$ using this
choice.  If the projection graph of $D$ is connected, then
\[
  Z_{r,s}^{R_0}(D)=\TV_{r,s}(K(D)).
\]
\end{proposition}

\begin{proof}
Theorem~\ref{thm:colouring-correspondence} matches the colourings, and
Propositions~\ref{prop:crossing-local-factors} and
\ref{prop:global-factors} identify the term for
$\theta=(\chi,\{\beta_c\})$ with
\[
  \eta_{r,s}^{N(D)}W_D(\chi)
  \prod_{c\in\Cross(D)}w_c(\chi,\beta_c).
\]
Summing and applying
Proposition~\ref{prop:complete-colouring-expansion} gives the result.
\end{proof}

\begin{proof}[Proof of Theorem~\ref{thm:reconstruction-connected}]
For any base region $R_0$, use the associated $K(D)$.
Proposition~\ref{prop:equality-fixed-base-region},
Theorem~\ref{thm:global-complex-identification}, and invariance of the
Turaev--Viro state sum give
\[
  Z_{r,s}^{R_0}(D)=\TV_{r,s}(K(D))=\TV_{r,s}(M(D)).\qedhere
\]
\end{proof}

Since Theorem~\ref{thm:reconstruction-connected} holds for every fixed base
region, two applications give the following independence.

\begin{corollary}[Independence of the base region]
\label{cor:base-region-independence}
For any two base regions $R_0,R'_0\in\mathcal R(D)$,
\[
  Z_{r,s}^{R_0}(D)=Z_{r,s}^{R'_0}(D).
\]
Thus the common value may be denoted by $Z_{r,s}(D)$.
\end{corollary}

\begin{proof}
By Theorem~\ref{thm:reconstruction-connected}, both state sums equal
$\TV_{r,s}(M(D))$.
\end{proof}

\section{Consequences and an explicit calculation}\label{sec:consequences}

This section records consequences of the reconstruction theorem and gives an
explicit calculation for the lens space $L(2,1)$.

\subsection{Independence of the presentation}\label{subsec:presentation-independence}

We apply Theorem~\ref{thm:reconstruction-connected} to compare two four-fold
simple branched covering presentations of the same oriented closed
$3$-manifold.

\begin{corollary}\label{cor:presentation-independence}
Let $D$ and $D'$ be four-fold simple branched covering presentations with
connected projection graphs.  If they represent the same oriented closed
$3$-manifold, then
\[
  Z_{r,s}(D)=Z_{r,s}(D').
\]
\end{corollary}

\begin{proof}
Let $M$ be the represented manifold.  Theorem~\ref{thm:reconstruction-connected}
gives
\[
  Z_{r,s}(D)=\TV_{r,s}(M)=Z_{r,s}(D').
\]
\end{proof}

The covering moves MI and MII of Bobtcheva--Piergallini preserve the
represented oriented covering manifold, as do labelled Reidemeister moves
representing labelled isotopy~\cite[Section~1 and Figure~1]{bobtcheva-piergallini-2004}.
Hence, whenever the projection graphs before and after such a move are
connected, Corollary~\ref{cor:presentation-independence} shows that the
diagrammatic state sum is unchanged.
\subsection{Disconnected projections}\label{subsec:disconnected-projections}

We use labelled Reidemeister II moves to pass from a disconnected projection
graph to a connected one and then apply the preceding results.

Suppose that $\Gamma_D$ is disconnected.  Choose a complementary region whose
boundary meets two distinct connected components; such a region exists because
$S^2$ is connected.  Performing the labelled Reidemeister II move shown in
Figure~\ref{fig:connecting-reidemeister-two} inside this region joins those two
components and does not change the represented oriented covering manifold.
After finitely many repetitions, we obtain a diagram $D^\sharp$ with connected
projection graph and define
\[
  Z_{r,s}(D):=Z_{r,s}(D^\sharp).
\]
This value is independent of the connectedization.  Indeed, any two resulting
diagrams $D^\sharp$ and $D^{\sharp\prime}$ are connected presentations of the
same oriented closed $3$-manifold $M(D)$, so Corollary~\ref{cor:presentation-independence}
gives $Z_{r,s}(D^\sharp)=Z_{r,s}(D^{\sharp\prime})$.  By
Theorem~\ref{thm:reconstruction-connected}, their common value is
$\TV_{r,s}(M(D))$.

\begin{figure}[H]
  \centering
  \includegraphics[width=.75\textwidth]{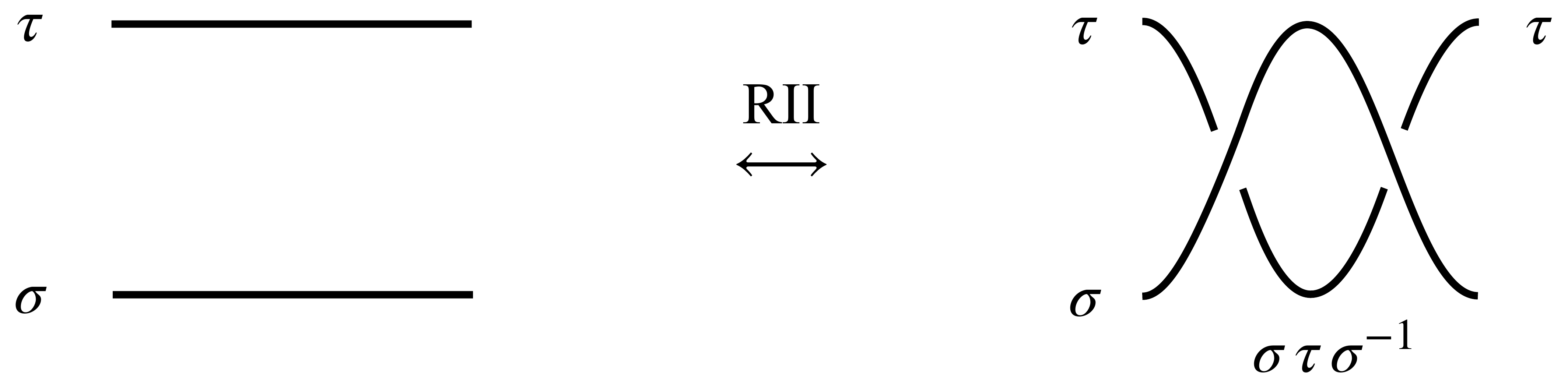}
  \caption{A labelled Reidemeister II move connecting two components of the
  projection graph; the middle under-strand is labelled
  $\sigma\tau\sigma^{-1}$.}
  \label{fig:connecting-reidemeister-two}
\end{figure}

\subsection{The lens space \texorpdfstring{$L(2,1)$}{L(2,1)}}
\label{subsec:lens-space-calculation}

We apply the diagrammatic state sum to a four-fold simple branched covering
presentation of $L(2,1)$ and evaluate it exactly for $r=3$.

For fixed $D$, $r$, and $s$, Definition~\ref{def:diagrammatic-state-sum} is a
finite sum.  Its weights are rational functions of $\zeta$, whose denominators
are nonzero under the parameter conditions fixed in Section~3.3, so the sum
can be evaluated exactly in $\mathbb Q(\zeta)$.  When $r=3$,
$C_3=\{0,1\}$ and the admissible triples are
\[
  (0,0,0),\qquad (0,1,1),\qquad (1,0,1),\qquad (1,1,0).
\]
Thus admissibility is the parity equation $a+b+c=0$ over $\mathbb F_2$,
and the compatible local colourings at a crossing are obtained from the
finite linear systems specified by Tables~A.2(a)--A.4(a).

Let $D_4$ be the presentation in Figure~\ref{fig:l21-presentation}.  It consists
of two Hopf-linked components labelled $(12)$ and two split stabilizing
unknots labelled $(23)$ and $(34)$.  Since its projection graph is
disconnected, we use the connectedization $D_4^\sharp$ from Section~8.2.

\begin{figure}[H]
  \centering
  \includegraphics[width=.95\textwidth]{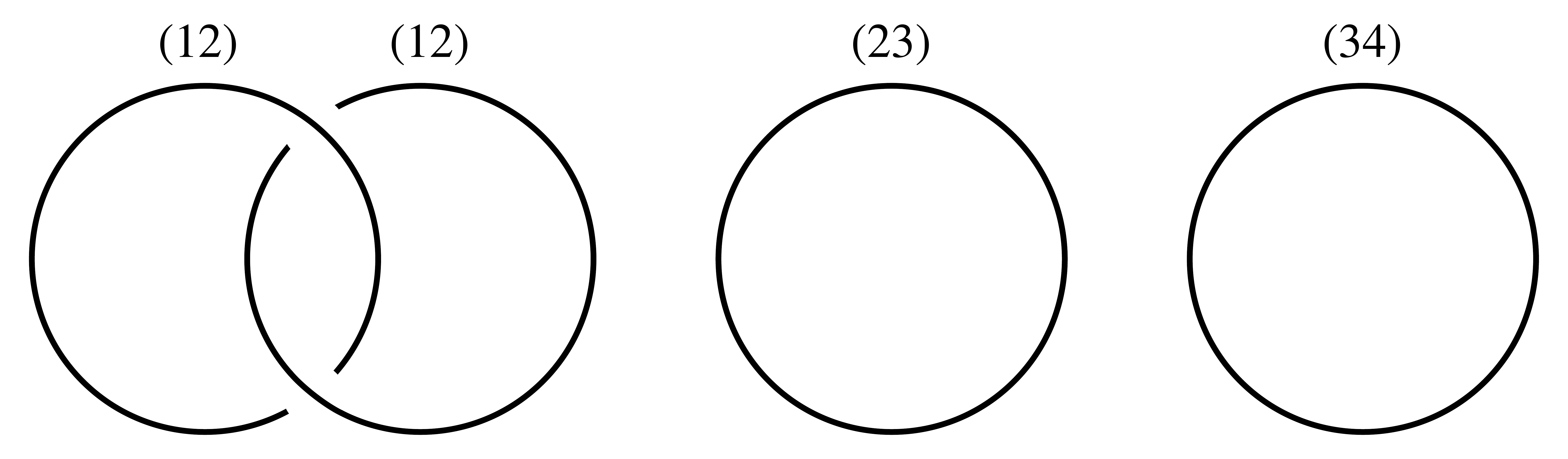}
  \caption{A four-fold simple branched covering presentation $D_4$ of
  $L(2,1)$.}
  \label{fig:l21-presentation}
\end{figure}

Write the two Hopf crossings as $c_0,c_1$.  Insert a labelled Reidemeister II
move between the $(12)$- and $(23)$-components, with crossings $c_2,c_3$ of
opposite signs, and a second such move between the $(23)$- and
$(34)$-components, with crossings $c_4,c_5$.  The two middle labels are
$(13)$ and $(24)$, respectively.  For this example, denote the twelve graph
edges of $D_4^\sharp$ by $e_1,\ldots,e_{12}$.  Table~\ref{tab:l21-graph-edges}
records their endpoints and labels; $c_j:h_i$ means that the endpoint is the
half-edge $h_i$ incident at $c_j$, with the cyclic order of Figures~6 and
10(a).

\begin{table}[H]
  \centering
  \small
  \caption{Graph-edge incidences and transposition labels in $D_4^\sharp$.}
  \label{tab:l21-graph-edges}
  \begin{minipage}[t]{.43\textwidth}
    \centering
    \begin{tabular}{@{}cccc@{}}
      \toprule
      edge & endpoint & endpoint & label \\
      \midrule
      $e_1$ & $c_0:h_2$ & $c_1:h_3$ & $(12)$ \\
      $e_2$ & $c_0:h_4$ & $c_1:h_1$ & $(12)$ \\
      $e_3$ & $c_0:h_3$ & $c_1:h_2$ & $(12)$ \\
      $e_4$ & $c_0:h_1$ & $c_2:h_1$ & $(12)$ \\
      $e_5$ & $c_3:h_3$ & $c_1:h_4$ & $(12)$ \\
      $e_6$ & $c_2:h_3$ & $c_3:h_1$ & $(13)$ \\
      \bottomrule
    \end{tabular}
  \end{minipage}\hspace{.04\textwidth}%
  \begin{minipage}[t]{.43\textwidth}
    \centering
    \begin{tabular}{@{}cccc@{}}
      \toprule
      edge & endpoint & endpoint & label \\
      \midrule
      $e_7$ & $c_2:h_4$ & $c_3:h_4$ & $(23)$ \\
      $e_8$ & $c_2:h_2$ & $c_4:h_1$ & $(23)$ \\
      $e_9$ & $c_3:h_2$ & $c_5:h_3$ & $(23)$ \\
      $e_{10}$ & $c_4:h_3$ & $c_5:h_1$ & $(24)$ \\
      $e_{11}$ & $c_4:h_4$ & $c_5:h_4$ & $(34)$ \\
      $e_{12}$ & $c_4:h_2$ & $c_5:h_2$ & $(34)$ \\
      \bottomrule
    \end{tabular}
  \end{minipage}
\end{table}

The crossings $c_0,c_1$ have type~1 and $c_2,\ldots,c_5$ have type~2.
Thus $D_4^\sharp$ has six crossings, twelve graph edges, and eight
complementary regions, so $N(D_4^\sharp)=56$.  If $f_k$ denotes the number
of $k$-simplices, the triangulation $K(D_4^\sharp)$ has $f$-vector
\[
  (f_0,f_1,f_2,f_3)=(56,252,392,196).
\]

The first-stage admissibility conditions give
\[
  |\Adm_1(D_4^\sharp;3)|=2^{45}.
\]
Precisely $2^{44}$ of these colourings satisfy
$\Col_c(\chi)\ne\varnothing$ at every crossing; the rest contribute zero
because Definition~\ref{def:crossing-weight} gives $W_c(\chi)=0$ when some
$\Col_c(\chi)$ is empty.  For each contributing $\chi$, all six sets
$\Col_c(\chi)$ have $2^2$ elements.  Definition~\ref{def:diagram-colouring}
therefore gives
\[
  2^{44}(2^2)^6=2^{56}
\]
contributing admissible diagram colourings.

For $s=1,5$, after the common $2^{12}$ local multiplicity is factored out,
the sign is a nonzero linear character on the $44$-dimensional contributing
subspace.  Each sign occurs $2^{43}$ times, so the sum is zero.  For $s=2,4$,
all factors apart from the normalization are $1$, and $2^{56}$ cancels the
vertex normalization $2^{-56}$.  Thus
\[
  Z_{3,s}(D_4)=Z_{3,s}(D_4^\sharp)=
  \begin{cases}
    0, & s=1,5,\\
    1, & s=2,4.
  \end{cases}
\]

This agrees with Turaev and Viro's formula for
$L(2,1)\cong\mathbb{R}P^3$, which is $(\zeta-1)(\zeta^{-1}-1)/r$ when
$(-\zeta)^r=-1$ and $0$ otherwise
\cite[Section~8.1.B, equation~(25)]{turaev-viro-1992}; at $r=3$ it gives the
four values above.

\section{Conclusion}\label{sec:conclusion}

The two-stage diagrammatic state sum constructed above agrees with the
Turaev--Viro invariant of the represented manifold, and its evaluation for
$L(2,1)$ at $r=3$ agrees with Turaev and Viro's formula for
$\mathbb{R}P^3$.  Further work includes analysing the computational complexity
of evaluating the diagrammatic state sum, investigating whether the
triangulations $K(D)$ can be simplified effectively by Pachner moves,
computing further examples, and extending the construction to other spherical
fusion categories.

\subsection*{Acknowledgements}

The author gratefully acknowledges financial support from the Women's Future
Development Organization, Tokyo University of Agriculture and Technology,
for the computing equipment used in this research.

\subsection*{Data availability}

The main theorem can be verified from the text and Appendices~A--C.  The
machine-readable package is available at
\url{https://doi.org/10.5281/zenodo.21884849} and contains the $16/4/16$
finite equivariance certificate, a reproducible correspondence check relating the finite data in
Appendices~A--C to Proposition~\ref{prop:crossing-local-factors}, and exact
data, scripts, and expected results for
Section~\ref{subsec:lens-space-calculation}.  SHA-256 checksums verify file
integrity.

\appendix
\section{Finite tables for the three crossing types}
\label{app:crossing-tables}

This appendix records the finite data used in Section~3 to compute the local
colourings and the crossing weights directly from the diagram.  Every entry
uses only the first-stage colours of Section~2 and the finitely many local
colours introduced at the crossing.  Their correspondence with the local
triangulations is proved in Section~7 and is not needed to execute the
diagrammatic calculation.

\setcounter{table}{0}
\renewcommand{\thetable}{\thesection.\arabic{table}}
\renewcommand{\theHtable}{\thesection.\arabic{table}}

\subsection{Reading conventions}

We use the position notation common to Figures~6 and~10.  The half-edges
meeting the crossing from the left, top, right, and bottom are
$h_1,h_2,h_3,h_4$, their graph edges are $e_1,e_2,e_3,e_4$, and the four
local sectors are denoted by $R_{12},R_{23},R_{34},R_{41}$.  If the same
graph edge or global complementary region occupies more than one of these
positions, its same first-stage colour is substituted at each such position.
The $q$-indices are transferred to an actual crossing by the rule in
Section~3.2.

For each type, part~(a) of the corresponding table lists all admissibility
triples.  A local assignment is retained precisely when every listed triple
is admissible.  Part~(b) lists the ordered six-colour tuples used once each
in the tetrahedron factors.  Their entries are already ordered as

\[
  01,\ 02,\ 03,\ 12,\ 13,\ 23,
\]

as in Sections~3.4 and~5.1.  Thus the two parts of each table completely
specify both the admissibility test and the product of tetrahedron factors.

\begin{table}[H]
  \centering
  \small
  \caption{Sizes of the three finite crossing tables.}
  \label{tab:appendix-a-sizes}
  \begin{tabular}{@{}c c c c@{}}
    \toprule
    Type & Local colours & Admissibility triples & Six-colour tuples \\
    \midrule
    1 & $x_1,\ldots,x_{22}$ & 52 & 34 \\
    2 & $x_1,\ldots,x_{20}$ & 48 & 32 \\
    3 & $x_1,\ldots,x_{20}$ & 48 & 32 \\
    \bottomrule
  \end{tabular}
\end{table}

\begin{table}[H]
  \centering
  \caption{Finite crossing table for type~1.}
  \label{tab:appendix-a-type1}
  \scriptsize
  \setlength{\tabcolsep}{2.5pt}
  \renewcommand{\arraystretch}{1.02}
  \textbf{(a) Admissibility triples}\par\smallskip
  \begin{tabular}{@{}c l@{}}
    \toprule
    No. & Triple \\ 
    \midrule
    A1 & $(p_{1}(e_{1}), p_{1}(e_{4}), x_{2})$ \\
    A2 & $(p_{1}(e_{1}), p_{1}(e_{2}), x_{1})$ \\
    A3 & $(p_{1}(e_{4}), p_{1}(e_{3}), x_{6})$ \\
    A4 & $(p_{1}(e_{4}), p_{1}(e_{2}), x_{5})$ \\
    A5 & $(p_{1}(e_{3}), p_{1}(e_{2}), x_{4})$ \\
    A6 & $(p_{2}(e_{1}), p_{2}(e_{4}), x_{10})$ \\
    A7 & $(p_{2}(e_{1}), p_{2}(e_{2}), x_{9})$ \\
    A8 & $(p_{2}(e_{4}), p_{2}(e_{3}), x_{3})$ \\
    A9 & $(p_{2}(e_{4}), p_{2}(e_{2}), x_{5})$ \\
    A10 & $(p_{2}(e_{3}), p_{2}(e_{2}), x_{8})$ \\
    A11 & $(q_{2}(e_{1}), q_{2}(e_{4}), x_{2})$ \\
    A12 & $(q_{2}(e_{1}), q_{2}(e_{3}), x_{7})$ \\
    A13 & $(q_{2}(e_{1}), q_{2}(e_{2}), x_{9})$ \\
    A14 & $(q_{2}(e_{4}), q_{2}(e_{3}), x_{3})$ \\
    A15 & $(q_{2}(e_{3}), q_{2}(e_{2}), x_{4})$ \\
    A16 & $(q_{1}(e_{1}), q_{1}(e_{4}), x_{10})$ \\
    A17 & $(q_{1}(e_{1}), q_{1}(e_{3}), x_{7})$ \\
    A18 & $(q_{1}(e_{1}), q_{1}(e_{2}), x_{1})$ \\
    A19 & $(q_{1}(e_{4}), q_{1}(e_{3}), x_{6})$ \\
    A20 & $(q_{1}(e_{3}), q_{1}(e_{2}), x_{8})$ \\
    A21 & $(x_{2}, x_{7}, x_{3})$ \\
    A22 & $(x_{2}, x_{1}, x_{5})$ \\
    A23 & $(x_{10}, x_{7}, x_{6})$ \\
    A24 & $(x_{10}, x_{9}, x_{5})$ \\
    A25 & $(x_{7}, x_{9}, x_{4})$ \\
    A26 & $(x_{7}, x_{1}, x_{8})$ \\
    \bottomrule
  \end{tabular}
  \hspace{14pt}%
  \begin{tabular}{@{}c l@{}}
    \toprule
    No. & Triple \\ 
    \midrule
    A27 & $(x_{3}, x_{5}, x_{8})$ \\
    A28 & $(x_{6}, x_{5}, x_{4})$ \\
    A29 & $(x_{11}, x_{12}, b_{3}(R_{41}))$ \\
    A30 & $(x_{11}, x_{12}, b_{3}(R_{34}))$ \\
    A31 & $(x_{11}, x_{12}, b_{3}(R_{23}))$ \\
    A32 & $(x_{11}, x_{12}, b_{3}(R_{12}))$ \\
    A33 & $(x_{11}, x_{13}, p_{3}(e_{1}))$ \\
    A34 & $(x_{11}, x_{14}, p_{3}(e_{4}))$ \\
    A35 & $(x_{11}, x_{15}, p_{3}(e_{3}))$ \\
    A36 & $(x_{11}, x_{16}, p_{3}(e_{2}))$ \\
    A37 & $(x_{12}, x_{13}, q_{3}(e_{1}))$ \\
    A38 & $(x_{12}, x_{14}, q_{3}(e_{4}))$ \\
    A39 & $(x_{12}, x_{15}, q_{3}(e_{3}))$ \\
    A40 & $(x_{12}, x_{16}, q_{3}(e_{2}))$ \\
    A41 & $(x_{17}, x_{18}, b_{4}(R_{41}))$ \\
    A42 & $(x_{17}, x_{18}, b_{4}(R_{34}))$ \\
    A43 & $(x_{17}, x_{18}, b_{4}(R_{23}))$ \\
    A44 & $(x_{17}, x_{18}, b_{4}(R_{12}))$ \\
    A45 & $(x_{17}, x_{19}, p_{4}(e_{1}))$ \\
    A46 & $(x_{17}, x_{20}, p_{4}(e_{4}))$ \\
    A47 & $(x_{17}, x_{21}, p_{4}(e_{3}))$ \\
    A48 & $(x_{17}, x_{22}, p_{4}(e_{2}))$ \\
    A49 & $(x_{18}, x_{19}, q_{4}(e_{1}))$ \\
    A50 & $(x_{18}, x_{20}, q_{4}(e_{4}))$ \\
    A51 & $(x_{18}, x_{21}, q_{4}(e_{3}))$ \\
    A52 & $(x_{18}, x_{22}, q_{4}(e_{2}))$ \\
    \bottomrule
  \end{tabular}
  \par\medskip
  \textbf{(b) Ordered six-colour tuples $(01,02,03,12,13,23)$}\par\smallskip
  \begin{tabular}{@{}c l@{}}
    \toprule
    No. & Tuple \\ 
    \midrule
    S1 & $(b_{1}(R_{12}), p_{1}(e_{1}), p_{1}(e_{2}), q_{1}(e_{1}), q_{1}(e_{2}), x_{1})$ \\
    S2 & $(b_{1}(R_{41}), p_{1}(e_{1}), p_{1}(e_{4}), q_{2}(e_{1}), q_{2}(e_{4}), x_{2})$ \\
    S3 & $(b_{2}(R_{34}), p_{2}(e_{3}), p_{2}(e_{4}), q_{2}(e_{3}), q_{2}(e_{4}), x_{3})$ \\
    S4 & $(b_{1}(R_{23}), p_{1}(e_{2}), p_{1}(e_{3}), q_{2}(e_{2}), q_{2}(e_{3}), x_{4})$ \\
    S5 & $(p_{1}(e_{1}), p_{1}(e_{2}), p_{1}(e_{4}), x_{1}, x_{2}, x_{5})$ \\
    S6 & $(p_{1}(e_{2}), p_{1}(e_{3}), p_{1}(e_{4}), x_{4}, x_{5}, x_{6})$ \\
    S7 & $(x_{1}, x_{7}, x_{2}, x_{8}, x_{5}, x_{3})$ \\
    S8 & $(q_{2}(e_{1}), q_{2}(e_{2}), q_{2}(e_{3}), x_{9}, x_{7}, x_{4})$ \\
    S9 & $(q_{2}(e_{1}), q_{2}(e_{3}), q_{2}(e_{4}), x_{7}, x_{2}, x_{3})$ \\
    S10 & $(b_{2}(R_{12}), p_{2}(e_{1}), p_{2}(e_{2}), q_{2}(e_{1}), q_{2}(e_{2}), x_{9})$ \\
    S11 & $(b_{2}(R_{41}), p_{2}(e_{1}), p_{2}(e_{4}), q_{1}(e_{1}), q_{1}(e_{4}), x_{10})$ \\
    S12 & $(b_{1}(R_{34}), p_{1}(e_{3}), p_{1}(e_{4}), q_{1}(e_{3}), q_{1}(e_{4}), x_{6})$ \\
    S13 & $(b_{2}(R_{23}), p_{2}(e_{2}), p_{2}(e_{3}), q_{1}(e_{2}), q_{1}(e_{3}), x_{8})$ \\
    S14 & $(p_{2}(e_{1}), p_{2}(e_{2}), p_{2}(e_{4}), x_{9}, x_{10}, x_{5})$ \\
    S15 & $(p_{2}(e_{2}), p_{2}(e_{3}), p_{2}(e_{4}), x_{8}, x_{5}, x_{3})$ \\
    S16 & $(x_{9}, x_{7}, x_{10}, x_{4}, x_{5}, x_{6})$ \\
    S17 & $(q_{1}(e_{1}), q_{1}(e_{2}), q_{1}(e_{3}), x_{1}, x_{7}, x_{8})$ \\
    \bottomrule
  \end{tabular}
  \hspace{14pt}%
  \begin{tabular}{@{}c l@{}}
    \toprule
    No. & Tuple \\ 
    \midrule
    S18 & $(q_{1}(e_{1}), q_{1}(e_{3}), q_{1}(e_{4}), x_{7}, x_{10}, x_{6})$ \\
    S19 & $(x_{11}, x_{12}, x_{13}, b_{3}(R_{12}), p_{3}(e_{1}), q_{3}(e_{1}))$ \\
    S20 & $(x_{11}, x_{12}, x_{13}, b_{3}(R_{41}), p_{3}(e_{1}), q_{3}(e_{1}))$ \\
    S21 & $(x_{11}, x_{12}, x_{14}, b_{3}(R_{41}), p_{3}(e_{4}), q_{3}(e_{4}))$ \\
    S22 & $(x_{11}, x_{12}, x_{14}, b_{3}(R_{34}), p_{3}(e_{4}), q_{3}(e_{4}))$ \\
    S23 & $(x_{11}, x_{12}, x_{15}, b_{3}(R_{34}), p_{3}(e_{3}), q_{3}(e_{3}))$ \\
    S24 & $(x_{11}, x_{12}, x_{15}, b_{3}(R_{23}), p_{3}(e_{3}), q_{3}(e_{3}))$ \\
    S25 & $(x_{11}, x_{12}, x_{16}, b_{3}(R_{23}), p_{3}(e_{2}), q_{3}(e_{2}))$ \\
    S26 & $(x_{11}, x_{12}, x_{16}, b_{3}(R_{12}), p_{3}(e_{2}), q_{3}(e_{2}))$ \\
    S27 & $(x_{17}, x_{18}, x_{19}, b_{4}(R_{12}), p_{4}(e_{1}), q_{4}(e_{1}))$ \\
    S28 & $(x_{17}, x_{18}, x_{19}, b_{4}(R_{41}), p_{4}(e_{1}), q_{4}(e_{1}))$ \\
    S29 & $(x_{17}, x_{18}, x_{20}, b_{4}(R_{41}), p_{4}(e_{4}), q_{4}(e_{4}))$ \\
    S30 & $(x_{17}, x_{18}, x_{20}, b_{4}(R_{34}), p_{4}(e_{4}), q_{4}(e_{4}))$ \\
    S31 & $(x_{17}, x_{18}, x_{21}, b_{4}(R_{34}), p_{4}(e_{3}), q_{4}(e_{3}))$ \\
    S32 & $(x_{17}, x_{18}, x_{21}, b_{4}(R_{23}), p_{4}(e_{3}), q_{4}(e_{3}))$ \\
    S33 & $(x_{17}, x_{18}, x_{22}, b_{4}(R_{23}), p_{4}(e_{2}), q_{4}(e_{2}))$ \\
    S34 & $(x_{17}, x_{18}, x_{22}, b_{4}(R_{12}), p_{4}(e_{2}), q_{4}(e_{2}))$ \\
    \bottomrule
  \end{tabular}
\end{table}

\begin{table}[H]
  \centering
  \caption{Finite crossing table for type~2.}
  \label{tab:appendix-a-type2}
  \scriptsize
  \setlength{\tabcolsep}{2.5pt}
  \renewcommand{\arraystretch}{1.02}
  \textbf{(a) Admissibility triples}\par\smallskip
  \begin{tabular}{@{}c l@{}}
    \toprule
    No. & Triple \\ 
    \midrule
    A1 & $(x_{1}, x_{4}, b_{1}(R_{41}))$ \\
    A2 & $(x_{2}, x_{5}, b_{2}(R_{41}))$ \\
    A3 & $(x_{3}, x_{6}, b_{3}(R_{41}))$ \\
    A4 & $(x_{1}, x_{4}, b_{1}(R_{34}))$ \\
    A5 & $(x_{2}, x_{6}, b_{2}(R_{34}))$ \\
    A6 & $(x_{3}, x_{5}, b_{3}(R_{34}))$ \\
    A7 & $(x_{1}, x_{5}, b_{1}(R_{23}))$ \\
    A8 & $(x_{2}, x_{6}, b_{2}(R_{23}))$ \\
    A9 & $(x_{3}, x_{4}, b_{3}(R_{23}))$ \\
    A10 & $(x_{1}, x_{5}, b_{1}(R_{12}))$ \\
    A11 & $(x_{2}, x_{4}, b_{2}(R_{12}))$ \\
    A12 & $(x_{3}, x_{6}, b_{3}(R_{12}))$ \\
    A13 & $(x_{1}, x_{7}, p_{1}(e_{1}))$ \\
    A14 & $(x_{2}, x_{7}, p_{2}(e_{1}))$ \\
    A15 & $(x_{3}, x_{8}, p_{3}(e_{1}))$ \\
    A16 & $(x_{1}, x_{9}, p_{1}(e_{4}))$ \\
    A17 & $(x_{2}, x_{10}, p_{2}(e_{4}))$ \\
    A18 & $(x_{3}, x_{10}, p_{3}(e_{4}))$ \\
    A19 & $(x_{1}, x_{11}, p_{1}(e_{3}))$ \\
    A20 & $(x_{2}, x_{12}, p_{2}(e_{3}))$ \\
    A21 & $(x_{3}, x_{11}, p_{3}(e_{3}))$ \\
    A22 & $(x_{1}, x_{13}, p_{1}(e_{2}))$ \\
    A23 & $(x_{2}, x_{14}, p_{2}(e_{2}))$ \\
    A24 & $(x_{3}, x_{14}, p_{3}(e_{2}))$ \\
    \bottomrule
  \end{tabular}
  \hspace{14pt}%
  \begin{tabular}{@{}c l@{}}
    \toprule
    No. & Triple \\ 
    \midrule
    A25 & $(x_{4}, x_{7}, q_{2}(e_{1}))$ \\
    A26 & $(x_{5}, x_{7}, q_{1}(e_{1}))$ \\
    A27 & $(x_{6}, x_{8}, q_{3}(e_{1}))$ \\
    A28 & $(x_{4}, x_{9}, q_{2}(e_{4}))$ \\
    A29 & $(x_{5}, x_{10}, q_{1}(e_{4}))$ \\
    A30 & $(x_{6}, x_{10}, q_{3}(e_{4}))$ \\
    A31 & $(x_{4}, x_{11}, q_{2}(e_{3}))$ \\
    A32 & $(x_{5}, x_{11}, q_{1}(e_{3}))$ \\
    A33 & $(x_{6}, x_{12}, q_{3}(e_{3}))$ \\
    A34 & $(x_{4}, x_{14}, q_{2}(e_{2}))$ \\
    A35 & $(x_{5}, x_{13}, q_{1}(e_{2}))$ \\
    A36 & $(x_{6}, x_{14}, q_{3}(e_{2}))$ \\
    A37 & $(x_{15}, x_{16}, b_{4}(R_{41}))$ \\
    A38 & $(x_{15}, x_{16}, b_{4}(R_{34}))$ \\
    A39 & $(x_{15}, x_{16}, b_{4}(R_{23}))$ \\
    A40 & $(x_{15}, x_{16}, b_{4}(R_{12}))$ \\
    A41 & $(x_{15}, x_{17}, p_{4}(e_{1}))$ \\
    A42 & $(x_{15}, x_{18}, p_{4}(e_{4}))$ \\
    A43 & $(x_{15}, x_{19}, p_{4}(e_{3}))$ \\
    A44 & $(x_{15}, x_{20}, p_{4}(e_{2}))$ \\
    A45 & $(x_{16}, x_{17}, q_{4}(e_{1}))$ \\
    A46 & $(x_{16}, x_{18}, q_{4}(e_{4}))$ \\
    A47 & $(x_{16}, x_{19}, q_{4}(e_{3}))$ \\
    A48 & $(x_{16}, x_{20}, q_{4}(e_{2}))$ \\
    \bottomrule
  \end{tabular}
  \par\medskip
  \textbf{(b) Ordered six-colour tuples $(01,02,03,12,13,23)$}\par\smallskip
  \begin{tabular}{@{}c l@{}}
    \toprule
    No. & Tuple \\ 
    \midrule
    S1 & $(x_{1}, x_{5}, x_{7}, b_{1}(R_{12}), p_{1}(e_{1}), q_{1}(e_{1}))$ \\
    S2 & $(x_{2}, x_{4}, x_{7}, b_{2}(R_{12}), p_{2}(e_{1}), q_{2}(e_{1}))$ \\
    S3 & $(x_{3}, x_{6}, x_{8}, b_{3}(R_{12}), p_{3}(e_{1}), q_{3}(e_{1}))$ \\
    S4 & $(x_{1}, x_{4}, x_{7}, b_{1}(R_{41}), p_{1}(e_{1}), q_{2}(e_{1}))$ \\
    S5 & $(x_{2}, x_{5}, x_{7}, b_{2}(R_{41}), p_{2}(e_{1}), q_{1}(e_{1}))$ \\
    S6 & $(x_{3}, x_{6}, x_{8}, b_{3}(R_{41}), p_{3}(e_{1}), q_{3}(e_{1}))$ \\
    S7 & $(x_{1}, x_{4}, x_{9}, b_{1}(R_{41}), p_{1}(e_{4}), q_{2}(e_{4}))$ \\
    S8 & $(x_{2}, x_{5}, x_{10}, b_{2}(R_{41}), p_{2}(e_{4}), q_{1}(e_{4}))$ \\
    S9 & $(x_{3}, x_{6}, x_{10}, b_{3}(R_{41}), p_{3}(e_{4}), q_{3}(e_{4}))$ \\
    S10 & $(x_{1}, x_{4}, x_{9}, b_{1}(R_{34}), p_{1}(e_{4}), q_{2}(e_{4}))$ \\
    S11 & $(x_{2}, x_{6}, x_{10}, b_{2}(R_{34}), p_{2}(e_{4}), q_{3}(e_{4}))$ \\
    S12 & $(x_{3}, x_{5}, x_{10}, b_{3}(R_{34}), p_{3}(e_{4}), q_{1}(e_{4}))$ \\
    S13 & $(x_{1}, x_{4}, x_{11}, b_{1}(R_{34}), p_{1}(e_{3}), q_{2}(e_{3}))$ \\
    S14 & $(x_{2}, x_{6}, x_{12}, b_{2}(R_{34}), p_{2}(e_{3}), q_{3}(e_{3}))$ \\
    S15 & $(x_{3}, x_{5}, x_{11}, b_{3}(R_{34}), p_{3}(e_{3}), q_{1}(e_{3}))$ \\
    S16 & $(x_{1}, x_{5}, x_{11}, b_{1}(R_{23}), p_{1}(e_{3}), q_{1}(e_{3}))$ \\
    \bottomrule
  \end{tabular}
  \hspace{14pt}%
  \begin{tabular}{@{}c l@{}}
    \toprule
    No. & Tuple \\ 
    \midrule
    S17 & $(x_{2}, x_{6}, x_{12}, b_{2}(R_{23}), p_{2}(e_{3}), q_{3}(e_{3}))$ \\
    S18 & $(x_{3}, x_{4}, x_{11}, b_{3}(R_{23}), p_{3}(e_{3}), q_{2}(e_{3}))$ \\
    S19 & $(x_{1}, x_{5}, x_{13}, b_{1}(R_{23}), p_{1}(e_{2}), q_{1}(e_{2}))$ \\
    S20 & $(x_{2}, x_{6}, x_{14}, b_{2}(R_{23}), p_{2}(e_{2}), q_{3}(e_{2}))$ \\
    S21 & $(x_{3}, x_{4}, x_{14}, b_{3}(R_{23}), p_{3}(e_{2}), q_{2}(e_{2}))$ \\
    S22 & $(x_{1}, x_{5}, x_{13}, b_{1}(R_{12}), p_{1}(e_{2}), q_{1}(e_{2}))$ \\
    S23 & $(x_{2}, x_{4}, x_{14}, b_{2}(R_{12}), p_{2}(e_{2}), q_{2}(e_{2}))$ \\
    S24 & $(x_{3}, x_{6}, x_{14}, b_{3}(R_{12}), p_{3}(e_{2}), q_{3}(e_{2}))$ \\
    S25 & $(x_{15}, x_{16}, x_{17}, b_{4}(R_{12}), p_{4}(e_{1}), q_{4}(e_{1}))$ \\
    S26 & $(x_{15}, x_{16}, x_{17}, b_{4}(R_{41}), p_{4}(e_{1}), q_{4}(e_{1}))$ \\
    S27 & $(x_{15}, x_{16}, x_{18}, b_{4}(R_{41}), p_{4}(e_{4}), q_{4}(e_{4}))$ \\
    S28 & $(x_{15}, x_{16}, x_{18}, b_{4}(R_{34}), p_{4}(e_{4}), q_{4}(e_{4}))$ \\
    S29 & $(x_{15}, x_{16}, x_{19}, b_{4}(R_{34}), p_{4}(e_{3}), q_{4}(e_{3}))$ \\
    S30 & $(x_{15}, x_{16}, x_{19}, b_{4}(R_{23}), p_{4}(e_{3}), q_{4}(e_{3}))$ \\
    S31 & $(x_{15}, x_{16}, x_{20}, b_{4}(R_{23}), p_{4}(e_{2}), q_{4}(e_{2}))$ \\
    S32 & $(x_{15}, x_{16}, x_{20}, b_{4}(R_{12}), p_{4}(e_{2}), q_{4}(e_{2}))$ \\
    \bottomrule
  \end{tabular}
\end{table}

\begin{table}[H]
  \centering
  \caption{Finite crossing table for type~3.}
  \label{tab:appendix-a-type3}
  \scriptsize
  \setlength{\tabcolsep}{2.5pt}
  \renewcommand{\arraystretch}{1.02}
  \textbf{(a) Admissibility triples}\par\smallskip
  \begin{tabular}{@{}c l@{}}
    \toprule
    No. & Triple \\ 
    \midrule
    A1 & $(x_{1}, x_{3}, b_{1}(R_{41}))$ \\
    A2 & $(x_{2}, x_{4}, b_{2}(R_{41}))$ \\
    A3 & $(x_{1}, x_{3}, b_{1}(R_{34}))$ \\
    A4 & $(x_{2}, x_{4}, b_{2}(R_{34}))$ \\
    A5 & $(x_{1}, x_{4}, b_{1}(R_{23}))$ \\
    A6 & $(x_{2}, x_{3}, b_{2}(R_{23}))$ \\
    A7 & $(x_{1}, x_{4}, b_{1}(R_{12}))$ \\
    A8 & $(x_{2}, x_{3}, b_{2}(R_{12}))$ \\
    A9 & $(x_{1}, x_{5}, p_{1}(e_{1}))$ \\
    A10 & $(x_{2}, x_{5}, p_{2}(e_{1}))$ \\
    A11 & $(x_{1}, x_{6}, p_{1}(e_{4}))$ \\
    A12 & $(x_{2}, x_{7}, p_{2}(e_{4}))$ \\
    A13 & $(x_{1}, x_{8}, p_{1}(e_{3}))$ \\
    A14 & $(x_{2}, x_{8}, p_{2}(e_{3}))$ \\
    A15 & $(x_{1}, x_{9}, p_{1}(e_{2}))$ \\
    A16 & $(x_{2}, x_{10}, p_{2}(e_{2}))$ \\
    A17 & $(x_{3}, x_{5}, q_{2}(e_{1}))$ \\
    A18 & $(x_{4}, x_{5}, q_{1}(e_{1}))$ \\
    A19 & $(x_{3}, x_{6}, q_{2}(e_{4}))$ \\
    A20 & $(x_{4}, x_{7}, q_{1}(e_{4}))$ \\
    A21 & $(x_{3}, x_{8}, q_{2}(e_{3}))$ \\
    A22 & $(x_{4}, x_{8}, q_{1}(e_{3}))$ \\
    A23 & $(x_{3}, x_{10}, q_{2}(e_{2}))$ \\
    A24 & $(x_{4}, x_{9}, q_{1}(e_{2}))$ \\
    \bottomrule
  \end{tabular}
  \hspace{14pt}%
  \begin{tabular}{@{}c l@{}}
    \toprule
    No. & Triple \\ 
    \midrule
    A25 & $(x_{11}, x_{13}, b_{3}(R_{41}))$ \\
    A26 & $(x_{12}, x_{14}, b_{4}(R_{41}))$ \\
    A27 & $(x_{11}, x_{14}, b_{3}(R_{34}))$ \\
    A28 & $(x_{12}, x_{13}, b_{4}(R_{34}))$ \\
    A29 & $(x_{11}, x_{14}, b_{3}(R_{23}))$ \\
    A30 & $(x_{12}, x_{13}, b_{4}(R_{23}))$ \\
    A31 & $(x_{11}, x_{13}, b_{3}(R_{12}))$ \\
    A32 & $(x_{12}, x_{14}, b_{4}(R_{12}))$ \\
    A33 & $(x_{11}, x_{15}, p_{3}(e_{1}))$ \\
    A34 & $(x_{12}, x_{16}, p_{4}(e_{1}))$ \\
    A35 & $(x_{11}, x_{17}, p_{3}(e_{4}))$ \\
    A36 & $(x_{12}, x_{17}, p_{4}(e_{4}))$ \\
    A37 & $(x_{11}, x_{18}, p_{3}(e_{3}))$ \\
    A38 & $(x_{12}, x_{19}, p_{4}(e_{3}))$ \\
    A39 & $(x_{11}, x_{20}, p_{3}(e_{2}))$ \\
    A40 & $(x_{12}, x_{20}, p_{4}(e_{2}))$ \\
    A41 & $(x_{13}, x_{15}, q_{3}(e_{1}))$ \\
    A42 & $(x_{14}, x_{16}, q_{4}(e_{1}))$ \\
    A43 & $(x_{13}, x_{17}, q_{3}(e_{4}))$ \\
    A44 & $(x_{14}, x_{17}, q_{4}(e_{4}))$ \\
    A45 & $(x_{13}, x_{19}, q_{3}(e_{3}))$ \\
    A46 & $(x_{14}, x_{18}, q_{4}(e_{3}))$ \\
    A47 & $(x_{13}, x_{20}, q_{3}(e_{2}))$ \\
    A48 & $(x_{14}, x_{20}, q_{4}(e_{2}))$ \\
    \bottomrule
  \end{tabular}
  \par\medskip
  \textbf{(b) Ordered six-colour tuples $(01,02,03,12,13,23)$}\par\smallskip
  \begin{tabular}{@{}c l@{}}
    \toprule
    No. & Tuple \\ 
    \midrule
    S1 & $(x_{1}, x_{4}, x_{5}, b_{1}(R_{12}), p_{1}(e_{1}), q_{1}(e_{1}))$ \\
    S2 & $(x_{2}, x_{3}, x_{5}, b_{2}(R_{12}), p_{2}(e_{1}), q_{2}(e_{1}))$ \\
    S3 & $(x_{1}, x_{3}, x_{5}, b_{1}(R_{41}), p_{1}(e_{1}), q_{2}(e_{1}))$ \\
    S4 & $(x_{2}, x_{4}, x_{5}, b_{2}(R_{41}), p_{2}(e_{1}), q_{1}(e_{1}))$ \\
    S5 & $(x_{1}, x_{3}, x_{6}, b_{1}(R_{41}), p_{1}(e_{4}), q_{2}(e_{4}))$ \\
    S6 & $(x_{2}, x_{4}, x_{7}, b_{2}(R_{41}), p_{2}(e_{4}), q_{1}(e_{4}))$ \\
    S7 & $(x_{1}, x_{3}, x_{6}, b_{1}(R_{34}), p_{1}(e_{4}), q_{2}(e_{4}))$ \\
    S8 & $(x_{2}, x_{4}, x_{7}, b_{2}(R_{34}), p_{2}(e_{4}), q_{1}(e_{4}))$ \\
    S9 & $(x_{1}, x_{3}, x_{8}, b_{1}(R_{34}), p_{1}(e_{3}), q_{2}(e_{3}))$ \\
    S10 & $(x_{2}, x_{4}, x_{8}, b_{2}(R_{34}), p_{2}(e_{3}), q_{1}(e_{3}))$ \\
    S11 & $(x_{1}, x_{4}, x_{8}, b_{1}(R_{23}), p_{1}(e_{3}), q_{1}(e_{3}))$ \\
    S12 & $(x_{2}, x_{3}, x_{8}, b_{2}(R_{23}), p_{2}(e_{3}), q_{2}(e_{3}))$ \\
    S13 & $(x_{1}, x_{4}, x_{9}, b_{1}(R_{23}), p_{1}(e_{2}), q_{1}(e_{2}))$ \\
    S14 & $(x_{2}, x_{3}, x_{10}, b_{2}(R_{23}), p_{2}(e_{2}), q_{2}(e_{2}))$ \\
    S15 & $(x_{1}, x_{4}, x_{9}, b_{1}(R_{12}), p_{1}(e_{2}), q_{1}(e_{2}))$ \\
    S16 & $(x_{2}, x_{3}, x_{10}, b_{2}(R_{12}), p_{2}(e_{2}), q_{2}(e_{2}))$ \\
    \bottomrule
  \end{tabular}
  \hspace{14pt}%
  \begin{tabular}{@{}c l@{}}
    \toprule
    No. & Tuple \\ 
    \midrule
    S17 & $(x_{11}, x_{13}, x_{15}, b_{3}(R_{12}), p_{3}(e_{1}), q_{3}(e_{1}))$ \\
    S18 & $(x_{12}, x_{14}, x_{16}, b_{4}(R_{12}), p_{4}(e_{1}), q_{4}(e_{1}))$ \\
    S19 & $(x_{11}, x_{13}, x_{15}, b_{3}(R_{41}), p_{3}(e_{1}), q_{3}(e_{1}))$ \\
    S20 & $(x_{12}, x_{14}, x_{16}, b_{4}(R_{41}), p_{4}(e_{1}), q_{4}(e_{1}))$ \\
    S21 & $(x_{11}, x_{13}, x_{17}, b_{3}(R_{41}), p_{3}(e_{4}), q_{3}(e_{4}))$ \\
    S22 & $(x_{12}, x_{14}, x_{17}, b_{4}(R_{41}), p_{4}(e_{4}), q_{4}(e_{4}))$ \\
    S23 & $(x_{11}, x_{14}, x_{17}, b_{3}(R_{34}), p_{3}(e_{4}), q_{4}(e_{4}))$ \\
    S24 & $(x_{12}, x_{13}, x_{17}, b_{4}(R_{34}), p_{4}(e_{4}), q_{3}(e_{4}))$ \\
    S25 & $(x_{11}, x_{14}, x_{18}, b_{3}(R_{34}), p_{3}(e_{3}), q_{4}(e_{3}))$ \\
    S26 & $(x_{12}, x_{13}, x_{19}, b_{4}(R_{34}), p_{4}(e_{3}), q_{3}(e_{3}))$ \\
    S27 & $(x_{11}, x_{14}, x_{18}, b_{3}(R_{23}), p_{3}(e_{3}), q_{4}(e_{3}))$ \\
    S28 & $(x_{12}, x_{13}, x_{19}, b_{4}(R_{23}), p_{4}(e_{3}), q_{3}(e_{3}))$ \\
    S29 & $(x_{11}, x_{14}, x_{20}, b_{3}(R_{23}), p_{3}(e_{2}), q_{4}(e_{2}))$ \\
    S30 & $(x_{12}, x_{13}, x_{20}, b_{4}(R_{23}), p_{4}(e_{2}), q_{3}(e_{2}))$ \\
    S31 & $(x_{11}, x_{13}, x_{20}, b_{3}(R_{12}), p_{3}(e_{2}), q_{3}(e_{2}))$ \\
    S32 & $(x_{12}, x_{14}, x_{20}, b_{4}(R_{12}), p_{4}(e_{2}), q_{4}(e_{2}))$ \\
    \bottomrule
  \end{tabular}
\end{table}

\section{Boundary triangles and local transport}
\label{app:boundary-transport}

This appendix records the eight triangles in the boundary of a base banana,
the local transport for each crossing type, and the resulting completeness
of the thirty-two lifted boundary triangles.  These data are used in the
global gluing of Section~6.

\setcounter{table}{0}
\renewcommand{\thetable}{\thesection.\arabic{table}}
\renewcommand{\theHtable}{\thesection.\arabic{table}}

\subsection{Base triangles and local transport}

Table~B.1 identifies the two local sectors adjacent to each half-edge and
hence the eight base triangles $\Delta_{h,R}$.  For each crossing type,
Table~B.2 gives the permutation that sends a local $P$-sheet index to the
local $Q$-sheet index at the other end of the lifted edge in each sector.

\begin{table}[H]
  \centering
  \small
  \setlength{\tabcolsep}{4pt}
  \renewcommand{\arraystretch}{1.08}
  \caption{Half-edges and adjacent sectors of the base banana.}
  \label{tab:appendix-b-boundary}
  \begin{tabular}{@{}c c c@{}}
    \toprule
    Half-edge & One adjacent sector & Other adjacent sector \\
    \midrule
    $h_{1}$ & $R_{41}$ & $R_{12}$ \\
    $h_{2}$ & $R_{12}$ & $R_{23}$ \\
    $h_{3}$ & $R_{23}$ & $R_{34}$ \\
    $h_{4}$ & $R_{34}$ & $R_{41}$ \\
    \bottomrule
  \end{tabular}
\end{table}

\begin{table}[H]
  \centering
  \small
  \setlength{\tabcolsep}{4pt}
  \renewcommand{\arraystretch}{1.08}
  \caption{Local $Q$-endpoint permutations in the three standard models.}
  \label{tab:appendix-b-transport}
  \begin{tabular}{@{}c c c c c@{}}
    \toprule
    Type & $R_{12}$ & $R_{23}$ & $R_{34}$ & $R_{41}$ \\
    \midrule
    1 & $\mathrm{id}$ & $(12)$ & $\mathrm{id}$ & $(12)$ \\
    2 & $\mathrm{id}$ & $(23)$ & $(123)$ & $(12)$ \\
    3 & $\mathrm{id}$ & $(34)$ & $(12)(34)$ & $(12)$ \\
    \bottomrule
  \end{tabular}
\end{table}

\subsection{Completeness of the lifted boundary}

We verify that the notation used in Section~6 accounts for the full boundary
without repetition.  This also gives a uniform description of the cone
components in the standard local triangulations.

\begin{lemma}\label{lem:appendix-b-completeness}
For every crossing $c$, the thirty-two triangles
$\Delta_{c;h,R}^{\,i}$, with $(h,R)$ ranging over Table~B.1 and
$i\in\{1,2,3,4\}$, are precisely the boundary triangles of the lifted local
model.
\end{lemma}

\begin{proof}
The boundary of the base banana is the union of the eight triangles in
Table~B.1.  The interior of each is disjoint from the branch set, so it has
four distinct lifts, uniquely distinguished by the $P$-sheet $i$.  Hence the
eight base triangles give exactly $8\cdot4=32$ lifted triangles.
\end{proof}

Except for the type~1 solid-torus component, each component of a standard
local triangulation is the cone from one interior vertex $O_{\Omega}$ over
the boundary triangles whose sheet labels lie in the component orbit
$\Omega$.  If the entry in Table~B.2 sends $i$ to $j$, the tetrahedron over
$\Delta_{h,R}^{\,i}$ has ordered vertices
\[
  \bigl(O_{\Omega},P_i,Q_j,v(h)^{[i]_{\tau_e}}\bigr).
\]
This rule determines all tetrahedra and their boundary faces in the cone
components.

\section{Finite PL identification of the type~1 local model}
\label{app:type1-pl-model}

This appendix specifies the eighteen-tetrahedron complex $X$ over the
component orbit $\{1,2\}$ of a type~1 crossing and identifies it, with its
boundary marking, with the corresponding component of the lifted banana.
The topology of that component was already established in Lemma~6.2; here we
record only the finite data needed for the PL identification.

\setcounter{table}{0}
\renewcommand{\thetable}{\thesection.\arabic{table}}
\renewcommand{\theHtable}{\thesection.\arabic{table}}

Throughout this appendix, the vertex $v_a^{12}$ is abbreviated to $v_a$.
All relevant edge labels are $(12)$, so the local edge-label orbit $\{1,2\}$
coincides with the component orbit.

\subsection{The eighteen ordered tetrahedra}

Table~C.1 fixes the finite complex before face identifications.  The listed
order is the vertex order used for all face maps and tetrahedron weights.

\begin{table}[H]
  \centering
  \scriptsize
  \setlength{\tabcolsep}{4pt}
  \renewcommand{\arraystretch}{1.08}
  \caption{The eighteen tetrahedra in the type~1 solid-torus component.}
  \label{tab:appendix-c-tetrahedra}
  \begin{tabular}{@{}c l@{}}
    \toprule
    Tet. & Ordered vertices \\
    \midrule
    $\mathrm{T}01$ & $P_{1},Q_{2},v_{1}^{12},v_{4}^{12}$ \\
    $\mathrm{T}02$ & $P_{1},Q_{1},v_{1}^{12},v_{2}^{12}$ \\
    $\mathrm{T}03$ & $P_{2},Q_{1},v_{2}^{12},v_{3}^{12}$ \\
    $\mathrm{T}04$ & $P_{1},Q_{1},v_{3}^{12},v_{4}^{12}$ \\
    $\mathrm{T}05$ & $P_{1},v_{1}^{12},v_{2}^{12},v_{4}^{12}$ \\
    $\mathrm{T}06$ & $P_{1},v_{2}^{12},v_{3}^{12},v_{4}^{12}$ \\
    \bottomrule
  \end{tabular}
  \hspace{14pt}%
  \begin{tabular}{@{}c l@{}}
    \toprule
    Tet. & Ordered vertices \\
    \midrule
    $\mathrm{T}07$ & $v_{1}^{12},v_{2}^{12},v_{3}^{12},v_{4}^{12}$ \\
    $\mathrm{T}08$ & $Q_{1},v_{1}^{12},v_{3}^{12},v_{4}^{12}$ \\
    $\mathrm{T}09$ & $Q_{1},v_{1}^{12},v_{2}^{12},v_{3}^{12}$ \\
    $\mathrm{T}10$ & $P_{2},Q_{1},v_{1}^{12},v_{4}^{12}$ \\
    $\mathrm{T}11$ & $P_{2},Q_{2},v_{1}^{12},v_{2}^{12}$ \\
    $\mathrm{T}12$ & $P_{1},Q_{2},v_{2}^{12},v_{3}^{12}$ \\
    \bottomrule
  \end{tabular}
  \hspace{14pt}%
  \begin{tabular}{@{}c l@{}}
    \toprule
    Tet. & Ordered vertices \\
    \midrule
    $\mathrm{T}13$ & $P_{2},Q_{2},v_{3}^{12},v_{4}^{12}$ \\
    $\mathrm{T}14$ & $P_{2},v_{1}^{12},v_{2}^{12},v_{4}^{12}$ \\
    $\mathrm{T}15$ & $P_{2},v_{2}^{12},v_{3}^{12},v_{4}^{12}$ \\
    $\mathrm{T}16$ & $v_{1}^{12},v_{2}^{12},v_{3}^{12},v_{4}^{12}$ \\
    $\mathrm{T}17$ & $Q_{2},v_{1}^{12},v_{3}^{12},v_{4}^{12}$ \\
    $\mathrm{T}18$ & $Q_{2},v_{1}^{12},v_{2}^{12},v_{3}^{12}$ \\
    \bottomrule
  \end{tabular}
\end{table}

\subsection{The twenty-eight internal face pairings}

Table~C.2 completes the definition of $X$.  Here $F_k$ denotes the face
opposite vertex position $k$.  A permutation $(p_0,p_1,p_2,p_3)$ sends the
vertex in position $k$ of the first tetrahedron to position $p_k$ of the
second; the entry corresponding to the omitted vertex records the two
opposite faces.

\begin{table}[H]
  \centering
  \scriptsize
  \setlength{\tabcolsep}{4pt}
  \renewcommand{\arraystretch}{1.08}
  \caption{The twenty-eight internal face pairings.}
  \label{tab:appendix-c-pairings}
  \begin{tabular}{@{}c c c c c@{}}
    \toprule
    Tet. & Face & Tet. & Face & Permutation \\
    \midrule
    $\mathrm{T}02$ & $F_{1}$ & $\mathrm{T}05$ & $F_{3}$ & $(0,3,1,2)$ \\
    $\mathrm{T}01$ & $F_{1}$ & $\mathrm{T}05$ & $F_{2}$ & $(0,2,1,3)$ \\
    $\mathrm{T}06$ & $F_{3}$ & $\mathrm{T}12$ & $F_{1}$ & $(0,2,3,1)$ \\
    $\mathrm{T}05$ & $F_{1}$ & $\mathrm{T}06$ & $F_{2}$ & $(0,2,1,3)$ \\
    $\mathrm{T}04$ & $F_{1}$ & $\mathrm{T}06$ & $F_{1}$ & $(0,1,2,3)$ \\
    $\mathrm{T}11$ & $F_{1}$ & $\mathrm{T}14$ & $F_{3}$ & $(0,3,1,2)$ \\
    $\mathrm{T}10$ & $F_{1}$ & $\mathrm{T}14$ & $F_{2}$ & $(0,2,1,3)$ \\
    $\mathrm{T}03$ & $F_{1}$ & $\mathrm{T}15$ & $F_{3}$ & $(0,3,1,2)$ \\
    $\mathrm{T}14$ & $F_{1}$ & $\mathrm{T}15$ & $F_{2}$ & $(0,2,1,3)$ \\
    $\mathrm{T}13$ & $F_{1}$ & $\mathrm{T}15$ & $F_{1}$ & $(0,1,2,3)$ \\
    $\mathrm{T}02$ & $F_{0}$ & $\mathrm{T}09$ & $F_{3}$ & $(3,0,1,2)$ \\
    $\mathrm{T}08$ & $F_{3}$ & $\mathrm{T}09$ & $F_{2}$ & $(0,1,3,2)$ \\
    $\mathrm{T}08$ & $F_{2}$ & $\mathrm{T}10$ & $F_{0}$ & $(1,2,0,3)$ \\
    $\mathrm{T}03$ & $F_{0}$ & $\mathrm{T}09$ & $F_{1}$ & $(1,0,2,3)$ \\
    \bottomrule
  \end{tabular}
  \hspace{14pt}%
  \begin{tabular}{@{}c c c c c@{}}
    \toprule
    Tet. & Face & Tet. & Face & Permutation \\
    \midrule
    $\mathrm{T}04$ & $F_{0}$ & $\mathrm{T}08$ & $F_{1}$ & $(1,0,2,3)$ \\
    $\mathrm{T}11$ & $F_{0}$ & $\mathrm{T}18$ & $F_{3}$ & $(3,0,1,2)$ \\
    $\mathrm{T}17$ & $F_{3}$ & $\mathrm{T}18$ & $F_{2}$ & $(0,1,3,2)$ \\
    $\mathrm{T}01$ & $F_{0}$ & $\mathrm{T}17$ & $F_{2}$ & $(2,0,1,3)$ \\
    $\mathrm{T}12$ & $F_{0}$ & $\mathrm{T}18$ & $F_{1}$ & $(1,0,2,3)$ \\
    $\mathrm{T}13$ & $F_{0}$ & $\mathrm{T}17$ & $F_{1}$ & $(1,0,2,3)$ \\
    $\mathrm{T}07$ & $F_{3}$ & $\mathrm{T}09$ & $F_{0}$ & $(1,2,3,0)$ \\
    $\mathrm{T}05$ & $F_{0}$ & $\mathrm{T}07$ & $F_{2}$ & $(2,0,1,3)$ \\
    $\mathrm{T}16$ & $F_{3}$ & $\mathrm{T}18$ & $F_{0}$ & $(1,2,3,0)$ \\
    $\mathrm{T}14$ & $F_{0}$ & $\mathrm{T}16$ & $F_{2}$ & $(2,0,1,3)$ \\
    $\mathrm{T}08$ & $F_{0}$ & $\mathrm{T}16$ & $F_{1}$ & $(1,0,2,3)$ \\
    $\mathrm{T}07$ & $F_{1}$ & $\mathrm{T}17$ & $F_{0}$ & $(1,0,2,3)$ \\
    $\mathrm{T}07$ & $F_{0}$ & $\mathrm{T}15$ & $F_{0}$ & $(0,1,2,3)$ \\
    $\mathrm{T}06$ & $F_{0}$ & $\mathrm{T}16$ & $F_{0}$ & $(0,1,2,3)$ \\
    \bottomrule
  \end{tabular}
\end{table}

\subsection{Boundary faces}

We now identify the faces not used in Table~C.2 and match them with the
sixteen lifted boundary triangles belonging to sheets $1$ and $2$.

\begin{lemma}\label{lem:appendix-c-boundary-faces}
The faces unused by Table~C.2 are exactly $F_2$ and $F_3$ of
$\mathrm{T}01$--$\mathrm{T}04$ and $\mathrm{T}10$--$\mathrm{T}13$.
They are the sixteen boundary faces of $X$.
\end{lemma}

\begin{proof}
The eighteen tetrahedra have seventy-two faces before gluing, and the
twenty-eight face pairs use fifty-six of them.  Reading Table~C.2 by
tetrahedron shows that it uses every face of $\mathrm{T}05$--$\mathrm{T}09$
and $\mathrm{T}14$--$\mathrm{T}18$, and $F_0,F_1$ of the remaining eight
tetrahedra.  The stated sixteen faces are therefore precisely the complement.
\end{proof}

The face corresponding to a boundary triangle $\Delta_{h_a,R}^{\,i}$ is
recovered as follows.  Use Tables~B.1--B.2 to find the $Q$-sheet index $j$
paired with $i$ over $R$ and the two banana vertices adjacent to $R$.
Table~C.1 contains a unique side
tetrahedron with first two vertices $P_i,Q_j$ and those two banana vertices.
Take $F_3$ if $v_a$ is in position $2$, and $F_2$ if it is in position $3$.

\begin{lemma}\label{lem:appendix-c-boundary-bijection}
This rule is a bijection from the sixteen lifted boundary triangles over
sheets $1$ and $2$ to the sixteen boundary faces of $X$.
\end{lemma}

\begin{proof}
For each sector $R$ and each $i\in\{1,2\}$, Table~C.1 has one such side
tetrahedron, and its two boundary faces contain the two different
half-edge vertices adjacent to $R$.  The rule is therefore injective; both
sets have $8\cdot2=16$ elements, so it is bijective.
\end{proof}

\subsection{Deck involution and local links}

We check that the evident exchange of the two sheets is simplicial and has
the local form of simple two-fold branching along the two branch edges.

Define an involution $\iota$ on $X$ by
\[
  P_1\leftrightarrow P_2,\qquad
  Q_1\leftrightarrow Q_2,\qquad
  v_a\longmapsto v_a,\qquad
  \mathrm{T}0k\leftrightarrow\mathrm{T}1k
  \quad(1\leq k\leq9).
\]
It preserves the ordered positions and all face pairings in Table~C.2.  Its
fixed set consists of the two proper edges $v_1v_3$ and $v_2v_4$.

\begin{table}[H]
  \centering
  \small
  \setlength{\tabcolsep}{4pt}
  \renewcommand{\arraystretch}{1.08}
  \caption{Links of the fixed edges.}
  \label{tab:appendix-c-fixed-links}
  \begin{tabular}{@{}c l l@{}}
    \toprule
    Fixed edge & Cyclic tetrahedron list in the interior link & Action \\
    \midrule
    $v_{1}v_{3}$ & $\mathrm{T}07,\mathrm{T}09,\mathrm{T}08,\mathrm{T}16,\mathrm{T}18,\mathrm{T}17$ & shift by three \\
    $v_{2}v_{4}$ & $\mathrm{T}05,\mathrm{T}06,\mathrm{T}16,\mathrm{T}14,\mathrm{T}15,\mathrm{T}07$ & shift by three \\
    \bottomrule
  \end{tabular}
\end{table}

In Table~C.4, an edge incident to the endpoint denotes the corresponding
direction in its vertex link.

\begin{table}[H]
  \centering
  \footnotesize
  \setlength{\tabcolsep}{4pt}
  \renewcommand{\arraystretch}{1.08}
  \caption{Links of the branch endpoints.}
  \label{tab:appendix-c-endpoint-links}
  \begin{tabular}{@{}c c l l@{}}
    \toprule
    Endpoint & Fixed direction & Boundary $4$-cycle of directions & Action \\
    \midrule
    $v_{1}$ & $v_{1}v_{3}$ & $P_{1}v_{1}, Q_{2}v_{1}, P_{2}v_{1}, Q_{1}v_{1}$ & shift by two \\
    $v_{3}$ & $v_{1}v_{3}$ & $P_{2}v_{3}, Q_{1}v_{3}, P_{1}v_{3}, Q_{2}v_{3}$ & shift by two \\
    $v_{2}$ & $v_{2}v_{4}$ & $P_{1}v_{2}, Q_{1}v_{2}, P_{2}v_{2}, Q_{2}v_{2}$ & shift by two \\
    $v_{4}$ & $v_{2}v_{4}$ & $P_{1}v_{4}, Q_{2}v_{4}, P_{2}v_{4}, Q_{1}v_{4}$ & shift by two \\
    \bottomrule
  \end{tabular}
\end{table}

The six-cycles in Table~C.3 are exchanged by a shift of three, and the
boundary four-cycles in Table~C.4 by a shift of two.  Hence the interiors and
endpoints of both fixed edges have the standard local model of simple
two-fold branching.

\subsection{Quotient and PL identification}

It remains to identify the quotient and its branch set.  The twenty face
pairs internal to the two nine-tetrahedron copies descend in pairs to ten
face pairings of the base banana.  The remaining eight crosswise pairings
form the four orbits in Table~C.5.  Thus $X/\iota$ has fourteen internal face
pairings and eight boundary faces, exactly those of the nine-tetrahedron
base banana in Figure~12.

\begin{table}[H]
  \centering
  \footnotesize
  \setlength{\tabcolsep}{4pt}
  \renewcommand{\arraystretch}{1.08}
  \caption{Exceptional quotient face orbits.}
  \label{tab:appendix-c-quotient}
  \begin{tabular}{@{}c c l@{}}
    \toprule
    Orbit & Geometric face & Crosswise lift pairings \\
    \midrule
    C1 & $Q,v_{1},v_{4}$ & $\mathrm{T}08:F_{2}  \leftrightarrow  \mathrm{T}10:F_{0}; \mathrm{T}01:F_{0}  \leftrightarrow  \mathrm{T}17:F_{2}$ \\
    C2 & $P,v_{2},v_{3}$ & $\mathrm{T}06:F_{3} \leftrightarrow \mathrm{T}12:F_{1}; \mathrm{T}03:F_{1} \leftrightarrow \mathrm{T}15:F_{3}$ \\
    C3 & $v_{2},v_{3},v_{4}$ & $\mathrm{T}07:F_{0} \leftrightarrow \mathrm{T}15:F_{0}; \mathrm{T}06:F_{0} \leftrightarrow \mathrm{T}16:F_{0}$ \\
    C4 & $v_{1},v_{3},v_{4}$ & $\mathrm{T}08:F_{0} \leftrightarrow \mathrm{T}16:F_{1}; \mathrm{T}07:F_{1} \leftrightarrow \mathrm{T}17:F_{0}$ \\
    \bottomrule
  \end{tabular}
\end{table}

The two triangles
\[
  \begin{aligned}
    D_{24}&:=\mathrm{T}05:F_1=\mathrm{T}06:F_2=[P,v_2,v_4],\\
    D_{13}&:=\mathrm{T}08:F_3=\mathrm{T}09:F_2=[Q,v_1,v_3]
  \end{aligned}
\]
are disjoint.  Their boundaries show that the two fixed edges are
simultaneously boundary-parallel, so the quotient branch set is the trivial
two-string tangle used in the type~1 base banana.

\begin{proposition}[Boundary-marked local PL identification]
\label{prop:appendix-c-identification}
The complex $X$ is PL homeomorphic to the component of the lifted type~1
banana indexed by $\{1,2\}$.  The homeomorphism preserves the $P$-vertices,
$Q$-vertices, banana vertices, sheet labels, and the sixteen boundary
triangles.
\end{proposition}

\begin{proof}
The involution and Tables~C.3--C.5 identify $X\to X/\iota$ with the connected
simple two-fold branched cover of the base banana along its two labelled
trivial strands.  This is the component considered in Lemma~6.2; its cover
is determined by the same transposition monodromy.  Lemmas~C.1--C.2 identify
the sixteen boundary triangles and all named vertices with the corresponding
lifted boundary data, giving the required boundary-marked PL homeomorphism.
\end{proof}

\bibliographystyle{unsrt}
\begingroup
\emergencystretch=3em
\hbadness=2000
\bibliography{references}

@article{hatakenaka-2010,
  author  = {Hatakenaka, Eri},
  title   = {Invariants of 3-manifolds derived from covering presentations},
  journal = {Mathematical Proceedings of the Cambridge Philosophical Society},
  volume  = {149},
  number  = {2},
  year    = {2010},
  pages   = {263--295},
  doi     = {10.1017/S0305004110000198}
}

@misc{bobtcheva-piergallini-2004,
  author       = {Bobtcheva, Ivelina and Piergallini, Riccardo},
  title        = {Covering moves and {Kirby} calculus},
  year         = {2004},
  eprint       = {math/0407032},
  archiveprefix = {arXiv},
  note         = {arXiv:math/0407032 [math.GT]},
  doi          = {10.48550/arXiv.math/0407032}
}

@article{turaev-viro-1992,
  author  = {Turaev, Vladimir G. and Viro, Oleg Y.},
  title   = {State sum invariants of 3-manifolds and quantum {$6j$}-symbols},
  journal = {Topology},
  volume  = {31},
  number  = {4},
  year    = {1992},
  pages   = {865--902},
  doi     = {10.1016/0040-9383(92)90015-A}
}

@article{burton-maria-spreer-2018,
  author  = {Burton, Benjamin A. and Maria, Cl{\'e}ment and Spreer, Jonathan},
  title   = {Algorithms and complexity for {Turaev--Viro} invariants},
  journal = {Journal of Applied and Computational Topology},
  volume  = {2},
  number  = {1--2},
  year    = {2018},
  pages   = {33--53},
  doi     = {10.1007/s41468-018-0016-2}
}

@article{roberts-1995,
  author  = {Roberts, Justin},
  title   = {Skein theory and {Turaev--Viro} invariants},
  journal = {Topology},
  volume  = {34},
  number  = {4},
  year    = {1995},
  pages   = {771--787},
  doi     = {10.1016/0040-9383(94)00053-0}
}

@book{kauffman-lins-1994,
  author    = {Kauffman, Louis H. and Lins, Sostenes L.},
  title     = {Temperley--Lieb Recoupling Theory and Invariants of 3-Manifolds},
  series    = {Annals of Mathematics Studies},
  volume    = {134},
  publisher = {Princeton University Press},
  year      = {1994},
  doi       = {10.1515/9781400882533}
}

@article{turaev-1992-shadow,
  author  = {Turaev, Vladimir G.},
  title   = {Shadow links and face models of statistical mechanics},
  journal = {Journal of Differential Geometry},
  volume  = {36},
  number  = {1},
  year    = {1992},
  pages   = {35--74},
  doi     = {10.4310/jdg/1214448442}
}

@article{carbone-2000,
  author  = {Carbone, Gaspare},
  title   = {{Turaev--Viro} invariant and {$3nj$} symbols},
  journal = {Journal of Mathematical Physics},
  volume  = {41},
  number  = {5},
  year    = {2000},
  pages   = {3068--3085},
  doi     = {10.1063/1.533292}
}

@article{benedetti-petronio-1995,
  author  = {Benedetti, Riccardo and Petronio, Carlo},
  title   = {A finite graphic calculus for 3-manifolds},
  journal = {Manuscripta Mathematica},
  volume  = {88},
  number  = {3},
  year    = {1995},
  pages   = {291--310},
  doi     = {10.1007/BF02567824}
}

@article{benedetti-petronio-1996,
  author  = {Benedetti, Riccardo and Petronio, Carlo},
  title   = {On {Roberts'} proof of the {Turaev--Walker} theorem},
  journal = {Journal of Knot Theory and Its Ramifications},
  volume  = {5},
  number  = {4},
  year    = {1996},
  pages   = {427--439},
  doi     = {10.1142/S0218216596000266}
}

@article{hatakenaka-nosaka-2012,
  author  = {Hatakenaka, Eri and Nosaka, Takefumi},
  title   = {Some topological aspects of 4-fold symmetric quandle invariants of 3-manifolds},
  journal = {International Journal of Mathematics},
  volume  = {23},
  number  = {7},
  year    = {2012},
  pages   = {1250064},
  doi     = {10.1142/S0129167X12500644}
}

@article{nosaka-2014,
  author  = {Nosaka, Takefumi},
  title   = {On third homologies of groups and of quandles via the {Dijkgraaf--Witten} invariant and {Inoue--Kabaya} map},
  journal = {Algebraic \& Geometric Topology},
  volume  = {14},
  number  = {5},
  year    = {2014},
  pages   = {2655--2691},
  doi     = {10.2140/agt.2014.14.2655}
}

@article{ishii-oshiro-2022,
  author  = {Ishii, Atsushi and Oshiro, Kanako},
  title   = {Quandle twisted {Alexander} invariants},
  journal = {Osaka Journal of Mathematics},
  volume  = {59},
  number  = {3},
  year    = {2022},
  pages   = {683--702},
  doi     = {10.18910/88492}
}

@book{rourke-sanderson-1982,
  author    = {Rourke, Colin P. and Sanderson, Brian J.},
  title     = {Introduction to Piecewise-Linear Topology},
  series    = {Springer Study Edition},
  publisher = {Springer-Verlag},
  address   = {Berlin},
  year      = {1982},
  doi       = {10.1007/978-3-642-81735-9}
}

@article{lackenby-2021,
  author  = {Lackenby, Marc},
  title   = {Links with splitting number one},
  journal = {Geometriae Dedicata},
  volume  = {214},
  number  = {1},
  year    = {2021},
  pages   = {319--351},
  doi     = {10.1007/s10711-021-00618-x}
}
\endgroup

\end{document}